\documentclass[a4paper,10pt,reqno]{amsart}
          \usepackage{amssymb}
	  \usepackage{amsmath}
          \usepackage{amsfonts}
          \usepackage[english]{babel}
          \usepackage[utf8]{inputenc}
\usepackage{bbold}
\usepackage{calrsfs}

\usepackage[margin=3cm]{geometry}

\usepackage{color}
 \usepackage[unicode,colorlinks,plainpages=false,hyperindex=true,bookmarksnumbered=true,bookmarksopen=true,pdfpagelabels]{hyperref}
\makeatletter
\renewcommand*{\p@section}{\S\,}
\renewcommand*{\p@subsection}{\S\,}
\renewcommand*{\p@subsubsection}{\S\,}
\makeatother

\usepackage{tikz}
 \usetikzlibrary{arrows}

\numberwithin{equation}{section}

\newtheorem{theorem}{Theorem}[section]
\newtheorem{lemma}[theorem]{Lemma}
\newtheorem{prop}[theorem]{Proposition}
\newtheorem{cor}[theorem]{Corollary}
\theoremstyle{definition}
\newtheorem{defi}[theorem]{Definition}

\theoremstyle{remark}
\newtheorem{remark}[theorem]{Remark}

\numberwithin{equation}{section}

\newcommand{\la}{\label}
\newcommand{\GL}{\mathtt{GL}}

\newcommand{\Ind}{\mathrm{Ind}}
\newcommand{\Mat}{\mathrm{Mat}}

\newcommand{\End}{\mathtt{End}}

\newcommand{\ov}{\overline}
\newcommand{\id}{\mathrm{id}}

\def\c{\mathbb{C}}
\def\R{\mathbb{R}}

\def\Z{\mathbb{Z}}

\def\N{\mathbb{N}}

\def\D{\mathcal{D}}
\def\tr{\mathrm{tr}}
\def\VV{\widehat V}
\def\vv{\mathbf{v}}
\def\ww{\mathbf{w}}
\def\uu{\mathbf{u}}
\def\WW{W_a}
\def\Wh{\widehat W}
\def\Ah{A_\hbar}
\def\aalpha{\widetilde\alpha}
\def\HH{\widehat{\mathfrak H}}
\def\RR{R}
\def\RH{\widehat{R}}
\def\TH{\mathbb{T}}
\def\YH{\widehat{Y}}
\def\Ra{R_a}

\newcommand\dpr[1]{\langle #1 \rangle} 
\def\ss{\mathfrak{s}}

\begin{document}
%
%
%
\title{Quantum Lax pairs via Dunkl and Cherednik operators}

\author{Oleg Chalykh}
\address{School of Mathematics, University of Leeds, Leeds LS2 9JT, UK}
\email{o.chalykh@leeds.ac.uk}
%

%
\begin{abstract}
We establish a direct link between Dunkl operators and quantum Lax matrices $\mathcal L$ for the Calogero--Moser systems associated to an arbitrary Weyl group $W$ (or an arbitrary finite reflection group in the rational case). This interpretation also provides a companion matrix $\mathcal A$ so that $\mathcal L, \mathcal A$ form a quantum Lax pair. Moreover, such an $\mathcal A$ can  be associated to any of the higher commuting quantum Hamiltonians of the system, so we obtain a family of quantum Lax pairs. These Lax pairs can be of various sizes, matching the sizes of orbits in the reflection representation of $W$, and in the elliptic case they contain a spectral parameter. This way we reproduce universal classical Lax pairs by D'Hoker--Phong and Bordner--Corrigan--Sasaki, and complement them with quantum Lax pairs in all cases (including the elliptic case, where they were not previously known). The same method, with the Dunkl operators replaced by the Cherednik operators, produces quantum Lax pairs for the generalised Ruijsenaars systems for arbitrary root systems. As one of the main applications, we calculate a Lax matrix for the elliptic $BC_n$ case with nine coupling constants (van Diejen system), thus providing an answer to a long-standing open problem. 
\end{abstract}

\maketitle

\section{Introduction}
The notion of a Lax pair has for a long time been  instrumental in both finite- and infinite-dimensional integrable systems, with the earliest examples given by P.~Lax and H.~Flaschka \cite{L, F}. Another famous example is the Lax pair found by J.~Moser \cite{Mo} for the classical rational Calogero--Moser system \cite{Ca}, which is a system of $n$ interacting particles on the line with coordinates $x_1,\dots, x_n$, described by the Hamiltonian
\begin{equation}\la{cm}
H=\frac12 \sum_{k=1}^n p_k^2+g^2\sum_{k<l}^n(x_k-x_l)^{-2}\,. 
\end{equation} 
Here $g$ is a coupling constant.  
The Lax presentation for this system involves two matrices $L$ and $A$ of size $n$ of the following form:
\begin{equation}\la{lp}
L_{kl}=
\begin{cases}
\mathrm{i}g(x_k-x_l)^{-1}\quad&\text{for}\ k\ne l\\
p_k\quad&\text{for}\ k=l\,,
\end{cases}
\qquad\quad
A_{kl}=
\begin{cases}
\mathrm{i}g(x_k-x_l)^{-2}\quad&\text{for}\ k\ne l\\
-\mathrm{i}g\sum_{j\ne k}^n(x_j-x_k)^{-2}\quad&\text{for}\ k=l\,.
\end{cases}
\end{equation} 
This allows presenting the equations of motion in the form
\begin{equation}\la{leq}
\frac{dL}{dt}=[A, L]\,.
\end{equation}
One immediate corollary is that $H_k=\tr\, L^k$, $k\in \N$ are conserved quantities (integrals of motion).

\medskip

The quantum Calogero--Moser system is described by a Schr\"odinger operator
\begin{equation}\la{qcm}
\widehat H=\frac12 \sum_{k=1}^n \hat{p}_k^2+ g(g-\hbar)\sum_{k<l}^{n} (x_k-x_l)^{-2}\,,\qquad \hat{p}_k=- \mathrm{i}\hbar\frac{\partial}{\partial x_k}\,. 
\end{equation}  
Quantum analogues $\mathcal L, \mathcal A$ of the above $L,A$ were introduced
in \cite{UHW} (see also \cite{BGHP, SS}). They are obtained by replacing $p_k$ by $\hat{p}_k$ and multiplying $A$ by $\mathrm{i}\hbar$:
\begin{equation}\la{qlp}
\mathcal L_{kl}=
\begin{cases}
\mathrm{i}g(x_k-x_l)^{-1}\quad&\text{for}\ k\ne l\\
\hat{p}_k\quad&\text{for}\ k=l\,,
\end{cases}
\qquad\quad
\mathcal A=\mathrm{i}\hbar A\,.
\end{equation}
These are matrices of size $n$ whose entries are partial differential operators. To write down the quantum Lax equation, introduce a diagonal matrix $\mathcal H=\widehat H\,\mathbb{1}_n$. The following can then be confirmed by a direct calculation:
\begin{equation}\la{qleq}
[\mathcal L, \mathcal H]=[\mathcal A, \mathcal L]\,.
\end{equation}
To see why this is indeed analogous to \eqref{leq}, note that \eqref{qleq} can be written as 
\begin{equation}\la{cla1}
(\mathrm{i}\hbar)^{-1}(\mathcal L_{kl} \widehat H- \widehat H\mathcal L_{kl})=[A, \mathcal L]_{kl}\,,\qquad k,l=1,\dots, n\,. 
\end{equation}
In the classical limit $\hbar\to 0$, $\mathcal L$ and $\widehat H$ reduce to $L$ and $H$, respectively, and the left-hand side reduces to the Poisson bracket $\{L_{kl}, H\}$. Thus,  
\begin{equation}\la{cla2}
\frac{d L_{kl}}{dt}=\{L_{kl}, H\}=[A,L]_{kl}\,,   
\end{equation}
which is \eqref{leq}.

\medskip

Similarly to the classical case, the above quantum Lax matrix can be used to produce first integrals for the Hamiltonian \eqref{qcm}. Namely, following \cite{UHW, SS, BGHP}, consider a pair $\vv, \ww$ of $n$-component column and row vectors
\begin{equation}\la{vw}
\vv=(1,\dots, 1)^T\,,\ \ww=(1,\dots, 1)\,. 
\end{equation}
Then we have the following easily verified properties of the matrix $\mathcal A$ in \eqref{lp}, \eqref{qlp}:
\begin{equation}\la{van}  
\ww\mathcal A=0\,,\quad \mathcal A\vv=0\,.
\end{equation}
Now define 
\begin{equation}\la{fi}
\widehat H_k=\ww\mathcal L^k\vv\,,\quad k\in\N\,.
\end{equation} 
(Note that $\widehat H_k$ is the sum of all entries of $\mathcal L^k$.) Then 
\begin{equation*}
[\widehat H, \widehat H_k]=\ww[\widehat H\,\mathbb{1}_n, \mathcal L^k]\vv=\ww[\mathcal H, \mathcal L^k]\vv=\ww[\mathcal L^k, \mathcal A]\vv=\ww\mathcal L\mathcal A \vv-\ww\mathcal A\mathcal L\vv=0\,,
\end{equation*}
where we used \eqref{van} and the relation $[\mathcal L^k, \mathcal H+\mathcal A]=0$. Therefore, $\widehat H_k$ are quantum integrals. 

\medskip

The original papers \cite{Mo, UHW} simply present the above Lax pairs but do not explain how they were found (cf. \cite{Ca1} where this was related to solving certain functional equations, leading to a Lax pair in the elliptic case). More conceptual ways of actually {\it deriving} Moser's Lax matrix have been subsequently discovered in \cite{KKS} in the framework of symplectic reduction, and in \cite{Kr2} in connection with the KP hierarchy. There is by now a vast literature devoted to various further generalisations and development of those ideas. However, the quantum Lax pairs lacked such an interpretation, although many authors remark on a similarity between Moser's Lax matrix and the Dunkl operators \cite{D} which played a pivotal role in the theory of Calogero--Moser systems since the works \cite{He1, He2, Ch, Op}.
The present paper fills that gap: as we explain, there is a direct link between the Dunkl operators and quantum Lax pairs. Our approach is inspired by an observation due to Etingof and Ginzburg, who in \cite{EG} derived the classical Moser's Lax matrix from the representation theory of Cherednik algebras. The main difference is that we work at the quantum level (and in a more general situation), so the classical Lax pairs are obtained by letting $\hbar\to 0$.
In the elliptic case we use elliptic Dunkl operators \cite{BFV, EM1} together with some important ideas from \cite{EFMV}. In this case the Lax matrices which we construct contain a spectral parameter (intrinsically linked to the theory of Dunkl operators). This way we reproduce, in a much simpler and more conceptual way, the previously known Lax pairs from the papers \cite{DHP, BCS, BMS, KPS}, as well as find some new ones (for instance, quantum Lax pairs were not known in the elliptic case). Also, our method allows us to associate a Lax partner to each of the commuting Hamiltonians of the Calogero--Moser problem, so we get a family of compatible Lax pairs. As a corollary, this gives a simple uniform proof of the fact that the classical Lax matrix $L$ remains isospectral under all of the commuting flows, implying that the functions $\tr\, L^k$ are in involution. Such a property is well-known in type $A$ \cite{Mo, P} and it is, of course, to be expected in other cases, but no general proof of that fact was available. Note that it is customary to study an integrable system first at the classical level and to use its classical Lax matrix to get an insight into a possibility of a quantisation. By contrast, we derive non-trivial properties of the classical Lax matrix by studying its quantum counterpart. It seems rather surprising that the Lax pairs are actually easier to understand at the quantum level. 

Perhaps more importantly, our construction works for the systems of Ruijsenaars--Schneider type (also referred to as \emph{relativistic Calogero--Moser systems}). The usual Ruijsenaars--Schneider system \cite{RS} corresponds to $R=A_n$; its quantum version was introduced by Ruijsenaars \cite{R87}, who also proved its complete integrability. A classical Lax pair for this system is well-known \cite{R87, BC87, KZ95}, see also \cite{Ha} where a quantum Lax matrix was introduced. However, for the models related to other root systems the question remained open for a long time (for instance, it was raised already by Inozemtsev in \cite{I}). The best result in that direction has been obtained recently by G\"orbe and Pusztai, who constructed in \cite{GoPu} a Lax pair for a two-parameter subfamily of the Koornwinder--van Diejen system (see also \cite{Pu} where a Lax matrix was found for the three-parameter rational case). The Koornwinder--van Diejen system \cite{Ko, vD} is a $BC_n$ version of the trigonometric Ruijsenaars--Schneider system, depending on five coupling parameters. There is also an elliptic version with nine parameters, introduced by van Diejen in \cite{vD1}, whose integrability was shown in \cite{KH97}. However, even in the trigonometric case with five parameters a Lax matrix remained unknown, let alone the nine-parameter elliptic version. Thus, it is rather pleasing that within our approach we are able to calculate it explicitly without much effort. 

We also give a general construction of Lax pairs for the generalised Ruijsenaars systems related to arbitrary root systems. Note that in the relativistic case instead of the Dunkl operators  one needs to use their $q$-analogues, known as Cherednik operators \cite{C1, C2}. The theory of Cherednik operators is well established in the trigonometric case, where they are intimately related to the theory of double affine Hecke algebras and Macdonald polynomials \cite{C3, M03}. Their elliptic analogues were introduced by Komori and Hikami \cite{KH98} following the ideas of Cherednik \cite{C5}, but some features of the trigonometric case seemed missing (or looked puzzlingly different) in the elliptic case. An issue here is that in the elliptic case one has braid (or Yang--Baxter) relations but no quadratic Hecke relations. As a result, elliptic Cherednik operators are only defined up to scaling (which in addition may depend on the dynamical variables), and so a correct way of defining them is not immediately obvious. We observe that a particularly well-behaved choice is the one associated with {\it unitary} $R$-matrices. It is this choice which allows us to draw a parallel with the results of \cite{EFMV} and construct Lax pairs for the generalised elliptic Ruijsenaars systems for all root systems. Calculating these Lax pairs explicitly is not easy in general (or even impractical: for instance, the smallest Lax matrix in the $E_8$ case has size $240$). We carry out such a calculation in two important cases: for the standard elliptic Ruijsenaars system and for the elliptic van Diejen system, i.e. for the $A_n$ and $BC_n$ cases of the theory. A crucial realization that such a calculation was possible came to us after seeing a paper of Nazarov and Sklyanin \cite{NS17} in which they calculated, rather nicely, a quantum Lax matrix for the trigonometric Ruijsenaars system. Similarly to us, they derive their Lax matrix directly from Cherednik operators, and although they do not consider quantum Lax pairs, some of their considerations are very close to ours. However, our approach is more general, in particular it extends to the $BC_n$ case, including the elliptic version. A nice special feature of the relativistic $A_n$ and $BC_n$ cases is that we can construct a Lax pair for each of the commuting Hamiltonians. In the $BC_n$ case this relies on results of Rains, who recently developed a geometric approach to elliptic DAHAs \cite{Rains}. For other root systems we are able to construct Lax pairs only for the Hamiltonians corresponding to minuscule and quasi-minuscule coweights. These Hamiltonians are the Macdonald operators \cite{M87} and their elliptic analogues \cite{KH98}. Since every root system has a (quasi)-minuscule coweight, we obtain at least one Lax pair for each root system. 

Let us remark that there exist various geometric approaches to Calogero--Moser and Ruijsenaars--Schneider systems, see \cite{GN, Ne, KZ95, FR, HM, Kr3, Kr4, KrS, FeK, BZN, LOSZ, FeM, Ko17} (where also many further references can be found). It would be interesting to see whether our quantum Lax pairs admit a geometric interpretation within any of those approaches. We also would like to mention that our interest in this problem was triggered by a paper by Sergeev and Veselov \cite{SV1} in which they construct quantum Lax pairs for certain deformed Calogero--Moser systems. We expect that the methods of the present paper can be adapted to give a conceptual approach to quantum Lax pairs for other deformed systems \cite{SV04, SV09, Fe, FeS}. Our results in the $BC_n$ case are also of crucial importance for constructing action-angle coordinates for the classical Koornwinder--van Diejen system (in the $A_n$ case this was done in \cite{R88}). Also, they give valuable hints towards a possible description of the center of the double affine Hecke algebra of type $CC^\vee_n$ (cf. \cite{Ob} in the $A_n$ case). This is also an important ingredient for describing the moduli space of ideals of the Askey--Wilson algebra. Some of these problems will be a subject of future work.  

The structure of the paper is as follows. In Section \ref{doql} we present the main construction in the rational case and illustrate it by deriving the quantum Lax pair \eqref{qlp}. Section \ref{trigch} generalises this to the trigonometric relativistic case, where the Dunkl operators are replaced by the Cherednik operators. In Section \ref{cc} we apply the results of Section \ref{trigch} for calculating a Lax matrix for the Koornwinder--van Diejen system. Section \ref{ecase} deals with the elliptic Calogero--Moser systems. Here we use elliptic Dunkl operators \cite{BFV}, and while the main idea remains the same, the construction of a Lax pair is more involved and uses ideas from \cite{EFMV}. 
Section \ref{eqcase} is devoted to the elliptic difference case, related to elliptic Cherednik and Macdonald--Ruijsenaars operators from \cite{KH98}. Our main effort here is to establish the existence of a Lax pair for any root system (Proposition \ref{elclq} and Theorem \ref{mainthm}). Subsections \ref{ecc}--\ref{calvd} are devoted to the $BC_n$ case; a Lax matrix for the elliptic van Diejen system is calculated in Subsection \ref{calvd}. 

The structure of the paper reflects how it developed over time: the main constructions and results in Sections \ref{doql}, \ref{trigch} go back to 2015, while the calculations in \ref{aq} and \ref{clax} were inspired by the work \cite{NS17}. The result of Proposition \ref{intq} is also a later addition, prompted by \cite[Corollary 2.6]{NS17}. Sections \ref{ecase}, \ref{eqcase} are more recent.  Note that some of the results in Sections \ref{doql}, \ref{trigch} can be obtained from the elliptic case as a limit. However, we decided to have them derived independently, mainly because some interesting features (for instance, Proposition \ref{intq}) do not seem to have an analogue in the elliptic case.

{\it Acknowledgement.} I would like to thank Yu.~Berest, F.~Calogero, P.~Etingof, L.~Feh{\'e}r, M.~Feigin, T.~G{\"o}rbe, A.~N.~Kirillov, M.~Nazarov, V.~Pasquier, E.~Rains, S.~Ruijsenaars, E.~Sklyanin, A.~Silantyev, A.~Veselov for stimulating dicussions and useful comments. I am especially grateful to Pavel Etingof for his help with proving Proposition \ref{elcl}. This work was partially supported by EPSRC under grant EP/K004999/1.

\section{Dunkl operators and quantum Lax pairs}\la{doql}

\subsection{}\la{2.1} Let us first recall the well-known link between Dunkl operators and rational Calogero--Moser systems \cite{D, He1}, cf. \cite{Po, BHV}. For simplicity, we restrict ourselves to the case of real Coxeter groups, but everything applies with minimal changes to any complex reflection group by replacing the Dunkl operators by their complex analogues \cite{DO}. 

Let $W$ be a finite Coxeter group and $V_\R$ be its reflection representation. We will work over $\c$, so $V:=V_\R\otimes_{\R}\,\c$ will be an $n$-dimensional complex vector space with a fixed $W$-invariant scalar product $\dpr{-, -}$. 
Let $R=R_+\sqcup\,-R_+$ be the root system of $W$ (not necessarily crystallographic). For each $\alpha\in R$ we have the orthogonal reflection $s_\alpha\in W$ acting on $V$ by the formula $s_\alpha(x)=x-2\frac{\dpr{\alpha,x}}{\dpr{\alpha,\alpha}}\alpha\,$, and these reflections generate the group $W$. We assume that the set $R$ is $W$-invariant. 
Below we will always identify $V$ with its dual by using the scalar product $\dpr{-,-}$, and hence equip $V\times V$ with a symplectic form transferred from $V\times V^*=T^*V$.  

Denote by $ \c(V) $ and $ \D(V) $
the rings of meromorphic functions and differential operators on $V$ with meromorphic coefficients, respectively. The group $ W $ acts naturally on $ \c(V) $ and $ \D(V) $, so we form the crossed products
$ \c(V)*W $ and $\D(V)*W$. As an algebra,
$\D(V)*W$ is generated by the elements
$w\in W $, $f\in \c(V) \,$, and derivations $\,
\partial_\xi \,$, $\, \xi \in V $, subject to the relations 
\begin{equation*}
w\,\partial_\xi=\partial_{w\xi}\, w\,,\quad w\, f=f^w\, w\,,\quad\text{where }\, f^w(x)=f(w^{-1}x)\,.
\end{equation*}
Any $a\in \D(V)*W$, admits a unique presentation 
\begin{equation}\la{form}
a=\sum_{w\in W} a_w w\quad\text{with}\quad a_w\in \D(V)\,. 
\end{equation}

Let us fix parameters $t\ne 0$ and a $W$-invariant function $c\,:\, R\to\c$, with $c(\alpha)$ abbreviated to $c_\alpha$. 
Introduce {\it Dunkl operators}
as the following elements of  $\D(V)*W$:
\begin{equation}
\label{du} y_\xi :=
t\partial_\xi+\sum_{\alpha\in R_+}
\frac{\dpr{\alpha,\xi}}{\dpr{\alpha, x}}c_\alpha s_\alpha\ , \quad \xi \in V\ .
\end{equation}
The two main properties of the Dunkl operators are their commutativity and equivariance: for all $\,\xi, \eta \in V\,$ and $ w \in W $,
\begin{equation}\la{duprop}
 y_{\xi}\,y_{\eta} = y_{\eta}\,y_{\xi}\,,\qquad\qquad \,w\,y_\xi w^{-1}=
y_{w(\xi)}\,.
\end{equation}
Therefore, the assignment $\,\xi \mapsto y_\xi\,$
extends to a $W$-equivariant injective algebra map
\begin{equation}
\la{hom}
SV=\bigoplus_{i\ge 0} S^iV \to \D(V)*W \,.
\end{equation}
The image of $q\in SV$ under this map is denoted by $q(y)$.
Let $\partial_i=\partial_{\xi_i}$ and $y_i=y_{\xi_i}$, where $\{\xi_i\,|\, i=1 \dots n\}$ is an orthonormal basis in $V$. Writing $\dpr{y, y}:=y_1^2+\dots +y_n^2$, we have
by \cite{D}:  
\begin{equation}\la{dcm}
\dpr{y, y}=t^2\Delta_V-\sum_{\alpha\in R_+} \dpr{\alpha, \alpha}\dpr{\alpha, x}^{-2}c_\alpha(c_\alpha+ts_\alpha)\,, \qquad \Delta_V=\sum_{i=1}^n \partial_{i}^2\,.
\end{equation} 
Let 
\begin{equation}\la{ew}
e=\frac{1}{|W|}\sum_{w\in W} w
\end{equation}
be the symmetrizing idempotent in the group algebra $\c W$. 
For any $W$-invariant element $q\in (SV)^W$, we have 
\begin{equation}\la{dq}
q(y)e={L}_qe\,,\quad {L}_q\in\D(V)^W 
\end{equation}
for some uniquely defined $W$-invariant differential operator ${L}_q$. Explicitly, if $q(y)$ is presented in the form \eqref{form}, $q(y)=\sum_{w\in W} a_ww$, then ${L}_q=\sum_{w\in W} a_w$. In particular, for $\xi^2:=\xi_1^2+\dots +\xi_n^2$ we find from \eqref{dcm} that 
\begin{equation}\label{cmo}
{L}_{\xi^2}=t^2\Delta_V-\sum_{\alpha\in R_+} \frac{c_\alpha(c_\alpha+t)\dpr{\alpha,\alpha}}{\dpr{\alpha,x}^2}\,.
\end{equation}
Substituting $t=-\mathrm{i}\hbar$ and $c_\alpha=\mathrm{i}g_\alpha$ gives the quantum Calogero--Moser Hamiltonian associated to the group $W$ \cite{OP1, OP2}:  
\begin{equation}\label{cmoq}
\widehat H=\sum_{k=1}^n \hat p_k^2 + \sum_{\alpha\in R_+} \frac{g_\alpha(g_\alpha-\hbar)\dpr{\alpha,\alpha}}{\dpr{\alpha,x}^2}\,, \qquad \hat p_k=-\mathrm{i}\hbar\,{\partial_{\xi_k}}\,.
\end{equation}
From the commutativity of the Dunkl operators it follows that the operators ${L}_q$, $q\in (SV)^W$ pairwise commute \cite{He1}. Since $(SV)^W$ is a free polynomial algebra on $n=\dim V$ generators, this proves that the quantum Hamiltonian \eqref{cmoq} is completely integrable. 

\subsection{} \la{qdl}
Let us now explain how to construct a quantum Lax pair for the Hamiltonian \eqref{cmoq}. For that we will work in a special representation of $\D(V)*W$. Namely, let us view $\c(V)$ as a left $\D(V)$-module with the usual action by differential operators, and consider the induced module 
\begin{equation*}
M=\Ind_{\D(V)}^{\D(V)*W}\,\c(V)\,.
\end{equation*}
We can write elements of $M$ as $f=\sum_{w\in W} w f_w$ with $f_w\in\c(V)$, thus identifiying $M$ and $\c W\otimes \c(V)$ (as a vector spaces). The algebra $\End_\c(M)$ then is identified with $\End_\c(\c W)\otimes \End_\c(\c(V))$, i.e. with \emph{operator-valued} matrices of size $|W|$. As a result, the (left) action of $\D(V)*W$ on $M$ gives a faithful representation
\begin{equation}\la{rep0}
\D(V)*W\to \Mat(|W|, \D(V))\,.
\end{equation}
For a $W$-invariant $a\in\D(V)$ we have:
$a\left(\sum_{w\in W} w f_w\right)=\sum_{w\in W} w (af_w)$.  
Therefore, in the above representation such $a$ acts as $a\mathbb{1}$.

\medskip 
Now pick a Dunkl operator $y_\xi$; obviously, it commutes with $\dpr{y, y}$.
From \eqref{dcm} we have: 
\begin{equation}\la{acal}
\frac12 \dpr{y, y}=\widehat H+\widehat A\,,\quad\text{where}\quad \widehat H:=\frac 12 L_{\xi^2}\,,\ \ \widehat A := \frac{t}{2} \sum_{\alpha\in R_+} c_\alpha\dpr{\alpha, \alpha}\dpr{\alpha, x}^{-2}(1-s_\alpha)\,.
\end{equation}
 As a result, if we set $\mathcal L$, $\mathcal H$, $\mathcal A$  to be the matrices representing under \eqref{rep0} the action of $y_\xi$, $\widehat H$ and $\widehat A$, respectively, we obtain   
\begin{equation}\la{leqg}
[\mathcal L, \mathcal H+ \mathcal A]=0\,,
\end{equation}
which is \eqref{qleq}. Since $\widehat H$ is $W$-invariant, the matrix $\mathcal H$ is $\widehat H\mathbb{1}$.
Therefore, we have obtained a quantum Lax pair $\mathcal L, \mathcal A$ of matrices of size $|W|$.

In fact, using this approach one can associate a suitable $\mathcal A$ to any of the commuting quantum Hamiltonians ${L}_q$, $q\in (SV)^W$. Indeed, suppose $q(y)=\sum_{w\in W} a_ww$ with $a_w\in\D(V)$. Then we have
\begin{equation}\la{acalg}
{q}(y)={L}_q+\widehat A\,,\quad\text{where}\quad {L}_{q}=\sum_{w\in W} a_w\,,\ \ \widehat A := \sum_{w\in W} a_w(w-1)\,,
\end{equation}
so the above construction gives a Lax pair with the same $\mathcal L$ but with different $\mathcal H, \mathcal A$. 

\begin{remark}
In the above construction one can replace the ring $\D(V)$ with a smaller ring $\D(V_{\mathrm{reg}})$ of algebraic differential operators on $V_{\mathrm{reg}}$, the complement to the reflection hyperplanes. Furthermore, when constructing the module $M$, one can induce from any $\D(V_{\mathrm{reg}})$-module, e.g. space of analytic functions on a small neighbourhood of a point in $V_{\mathrm{reg}}$. Therefore, one can allow elements of $M$ to be multivalued, with branching along the reflection hyperplanes in $V$.    
\end{remark}

\subsection{}\la{cld}
The classical limit corresponds to taking $t\to 0$. More precisely, we set $t=-\mathrm{i}\hbar$ and view the Dunkl operators as elements of the algebra
\begin{equation*}
\Ah*W=\c(V)[\hat p_1, \dots, \hat p_n][[\hbar]]*W\,,
\end{equation*}
where the quantum momenta $\hat p_k=-\mathrm{i}\hbar\partial_{\xi_k}$ satisfy the standard relations $[\hat p_k, f]=-\mathrm{i}\hbar\,\partial_{\xi_k}f$ for $f\in\c(V)$. We have an algebra isomorphism
\begin{equation*}
\eta_0:\ \Ah*W/\hbar\Ah*W\to A_0*W\,,\qquad f\mapsto f\,,\ \ \hat p_k\mapsto p_k\,,\ \ w\mapsto w\,,
\end{equation*}
where $A_0=\c(V)[p_1, \dots, p_n]$ is the classical version of $\Ah$.
Therefore, $\Ah$ (resp. $\Ah*W$) is a formal deformation of $A_0$ (resp. $A_0*W$), see, e.g., \cite[3.1]{Et}. 
Note that $A_0$ is commutative, with the standard Poisson bracket satisfying $[a,b]=\mathrm{i}\hbar\{\eta_0(a), \eta_0(b)\}+o(\hbar)$ for $a,b\in \Ah$. For any $a\in\Ah*W$, we call $\eta_0(a)$ the \emph{classical limit} of $a$. For example, the classical limit of \eqref{du} is
\begin{equation*}
y_\xi^{c}=p_\xi+\sum_{\alpha\in R_+}
\frac{\dpr{\alpha,\xi}}{\dpr{\alpha, x}}c_\alpha s_\alpha\,,
\end{equation*}
which is called a classical Dunkl operator, see \cite[6.30]{Et}. Here $p_\xi$ is the classical momentum in direction $\xi$. The operators $y_\xi^{c}$ are commuting elements of $A_0*W$, so we have a classical variant of the map \eqref{hom}:
\begin{equation*}
SV=\bigoplus_{i\ge 0} S^iV \to A_0*W \ ,\quad q \mapsto q(y^{c})\,.
\end{equation*} 
The classical limits of \eqref{dq} and \eqref{acalg} can be obtained by replacing Dunkl operators $y_\xi$ by their classical counterparts $y_\xi^{c}$. We will use the following important fact.
\begin{lemma}[Lemma 2.2, \cite{EFMV}]\la{le}
For any $q\in (SV)^W$, when writing $q(y^{c})=\sum_{w\in W} a_ww$ with $a_w\in A_0$, we have $a_w=0$ for all $w\ne \id$. 
\end{lemma}   
This tells us that the classical limit of $\widehat A$ in \eqref{acalg} is zero, in other words, the classical limit of $\hbar^{-1}\widehat A$ is well-defined. 
As a result, we have well-defined classical limits $L, H, A$ of $\mathcal L$, $\widehat H$ and $(\mathrm{i}\hbar)^{-1}\mathcal A$, respectively, after which the classical Lax equation follows in the same way as in \eqref{cla1}, \eqref{cla2}.

\subsection{}\la{small}
In \cite{BCS,BMS}, classical and quantum Lax pairs of various sizes were constructed, so let us explain how they arise within our approach. We start again by picking $y_\xi$ and $q(y)$ with $\xi\in V$, $q\in (SV)^W$ and writing $q(y)={L}_q+\widehat A$, as in \eqref{acalg}. To get a Lax pair of a smaller size, we choose $\xi$ with non-trivial stabiliser, writing
\begin{equation*} 
W'=\{w\in W\, \mid w\xi=\xi\}\,,\qquad e'=\frac{1}{|W'|}\sum_{w\in W'} w\,.
\end{equation*}
Obviously, $y_\xi e'=e'y_\xi$; also,  $q(y)e'=e'q(y)$, ${L}_qe'=e'{L}_q$ by their $W$-invariance.
As a result, the operators $y_\xi$, $q(y)$, ${L}_q$ and $\widehat A$ preserve the subspace
\begin{equation*}
M'=e'M \cong e'\c W\otimes \c(V)\,.  
\end{equation*}
The left $W$-module $e'\c W$ has dimension equal to $|W/W'|$, i.e. to the size of the orbit $W(\xi)$. Therefore, restricting $\mathcal L$, $\mathcal H$, $\mathcal A$ onto $M'$ produces a quantum Lax pair of size $|W/W'|$, with the smaller sizes achieved when $\xi$ is a fundamental weight. This agrees with the Lax pairs in \cite{BCS, BMS, KPS}.   

\subsection{} \la{2.5}
We can also explain why \eqref{van} holds in general. Namely, pick representatives $w_i$ for the cosets in $W'\backslash W$. Elements of $M'$ are linear combinations of $e'w_if_i$ with $f_i\in\c(V)$. Now, from \eqref{dq} and \eqref{acalg} we have $\widehat Ae=0$. Similarly, $eq(y)=e{L}_q$, so $e\widehat A=0$. Multiplication by $e$ acts on $M'$, and it is easy to see that the associated matrix is $\frac{1}{|W/W'|}\vv\ww$, where $\vv, \ww$ are the column/row vectors as in \eqref{vw} (of size $|W/W'|$). The relations $\widehat Ae=e\widehat A=0$ easily imply \eqref{van}. Therefore, the formula \eqref{fi} always produces quantum integrals; for instance, it works for any complex reflection group $W$. 

An alternative explanation of \eqref{fi} is as follows. Consider the operators $ey_\xi^ke$ with $k\in\N$. The symmetrizer $e$ acts on $M'$ by $\frac{1}{|W/W'|}\vv\ww$, and $y_\xi$ acts by $\mathcal L$. Therefore, up to a constant factor, $ey_\xi^ke$ acts as $\vv\ww\mathcal L^k\vv\ww=(\ww\mathcal L^k\vv)\vv\ww=|W/W'|(\ww\mathcal L^k\vv)e$. On the other hand, $ey_\xi^ke=\frac{1}{|W|}\left(\sum_{w\in W} y_{w\xi}^k\right)e$, so it commutes with any of the operators $eq(y)e={L}_qe$, $q\in (SV)^W$. 
As a result, $\ww\mathcal L^k\vv$ and ${L}_q$ commute.        

In Proposition \ref{intq} below we generalise this result to the relativistic case. 

\subsection{}\la{typea}
As an illustration, let us derive the quantum Lax pair \eqref{qlp}. 
We consider $W=\mathfrak{S}_n$ acting on $V=\c^n$ by permuting the basis vectors; it is generated by permutations $\ss_{ij}$, $i\ne j$.  The ring $\c(V)=\c(x)$ is the ring of functions of $n$ variables $x_1,\dots, x_n$. We have $n$ commuting Dunkl operators,
\begin{equation*}
y_i=t\partial_i+\sum_{j\ne i}c(x_i-x_j)^{-1}\ss_{ij}\,,\qquad i=1\dots n\,.
\end{equation*}
Choose $y_1$, so $\xi=(1,0,\dots, 0)$ and $W'=\mathfrak{S}_{n-1}$ is the permutation group on $\{2,\dots, n\}$, with 
\begin{equation*}
e'=\frac{1}{(n-1)!}\sum_{w\in \mathfrak{S}_{n-1}}w\,.
\end{equation*} 
We pick $\ss_{1j}$, $i=1\dots n$ as representatives for the cosets in $W'\backslash W$ (with $\mathfrak{s}_{11}:=\id$), and write elements of $M'=e' M$ as 
\begin{equation*}
f=\sum_{i=1}^n e'\mathfrak{s}_{1i} f_i\,,\quad\text{with}\ f_i\in\c(x)\,.
\end{equation*}
To find the matrix representing the action of $y_1$ on $M'$, it will be useful to work in a greater generality.
\begin{lemma}\la{alem}  Suppose that we have an element 
$Z=\sum_{i=1}^n Z_i\ss_{1i}$ with $Z_i\in\D(V)$,
whose action preserves $M'=e'M$.

(1) $Z_1$ is symmetric in $x_2, \dots, x_n$ and $Z_i$ for $i>1$ is symmetric in $x_j$ with $j\ne 1, i$. Also, we have $Z_i^{\ss_{ij}}=Z_j$ for any $i, j>1$.

(2) For each $1\le i,j \le n$, we have
\begin{equation*}
e'\ss_{1i}\ss_{1j}=e'\ss_{1k}\,,\qquad\text{with}\quad k=k(i,j)=\begin{cases}
i&\text{for}\ i\ne 1, j\\
j&\text{for}\ i=1\\
1\quad&\text{for}\ i=j\,.
\end{cases}
\end{equation*} 

(3) For any $i, j$ and $k=k(i, j)$ as above, the $(k, j)$-th entry the matrix representing the action of $Z$ on $M'$ is calculated as $(Z_i)^{\ss_{1j}\ss_{1i}}$.  
\end{lemma}

\begin{proof} Parts $(1)$ follows from $\ss_{ij}Ze'=Ze'$ for $i,j>1$. Part $(2)$ is straightforward. For part $(3)$, 
\begin{equation*}
Z\sum_{j=1}^n e'\mathfrak{s}_{1j} f_j=\sum_{j=1}^n e' Z\mathfrak{s}_{1j} f_j=
\sum_{i,j=1}^n e'\ss_{1i}\ss_{1j} Z_i^{ \ss_{1j}\ss_{1i}}f_j=\sum_{i,j=1}^n e'\ss_{1k} Z_i^{ \ss_{1j}\ss_{1i}}f_j\,.
\end{equation*}
\end{proof}

Applying the lemma to $y_1=t\partial_1\ss_{11}+\sum_{i\ne 1}c(x_1-x_i)^{-1}\ss_{1i}$, we find that it is represented by a matrix $\mathcal L$ with
\begin{equation*}
\mathcal L_{kl}=
\begin{cases}
c(x_k-x_l)^{-1}\quad&\text{for}\ k\ne l\\
t\partial_k\quad&\text{for}\ k=l\,.
\end{cases}
\end{equation*}
To calculate a Lax partner $\mathcal A$, we need to consider the action of $\widehat A$ \eqref{acal}. Note that $(1-\mathfrak{s}_{ij})e'=0$ if $i,j>1$; as a result, the action of $\widehat A$ can be replaced by
$ct\sum_{i\ne 1} (x_1-x_i)^{-2}(1-\mathfrak{s}_{1i})$\,. 
This element is $W'$-invariant, and its action on $M'$ is again calculated from the above lemma. The result is: 
\begin{equation*}
\mathcal A_{kl}=
\begin{cases}
-ct(x_k-x_l)^{-2}\quad&\text{for}\ k\ne l\\
ct\sum_{j\ne k}^n(x_j-x_k)^{-2}\quad&\text{for}\ k=l\,.
\end{cases}
\end{equation*}
Upon a substitution $t=-\mathrm{i}\hbar$, $c=\mathrm{i}g$, these $\mathcal L, \mathcal A$ coincide with the Lax pair \eqref{qlp}.

\medskip

\begin{remark}
By analogy with \cite{EG}, we can also consider the action of $x_1$ on $M'$. The corresponding matrix is $\mathcal X=\mathrm{diag}(x_1, \dots, x_n)$. 
We then have $\mathcal X\mathcal L-\mathcal L\mathcal X+\mathrm{i}(g-\hbar)\mathbb 1_n=\mathrm{i}g\vv\ww$. This is a quantum version of the well-known relation from \cite{KKS}.  
\end{remark}

\begin{remark}
In a similar manner one can calculate Lax matrices for the Calogero--Moser systems associated to complex reflection groups, e.g., for the case $W=G(m, p, n)$. 
Note that classical Lax matrices for $W=G(m, 1, n)$ (generalised symmetric group) were found by a different method in \cite{CS}.  
\end{remark}

\section{Cherednik operators and Lax pairs for relativistic Calogero--Moser systems} \la{trigch}

We start by outlining how the relativistic Calogero--Moser systems of Ruijsenaars--Schneider type can be constructed using affine Hecke algebras and Cherednik operators. This construction is due to Cherednik \cite{C1, C2}, and in the $\GL_n$ case it reproduces the quantum Ruijsenaars system \cite{R87}. For other root systems some of the commuting Hamiltonians are expressed by Macdonald difference operators appearing in the theory of Macdonald polynomials \cite{M87}. In the $C^\vee C_n$ case, the simplest Hamiltonian is the Koornwinder operator \cite{Ko}, and higher commuting Hamiltonians were constructed by van Diejen \cite{vD}. We therefore will refer to this case as Koornwinder--van Diejen system. Its interpretation in the framework of affine Hecke algebras can be found in \cite{No, Sa, St1}. 

We will be largely following Macdonald's book \cite{M03}, see also \cite{C3, Ki95}. Our setting is not the most general (it corresponds to the case \cite[(1.4.1)]{M03}), but the method is exactly the same in all other cases. The $C^\vee C_n$ case is treated separately in Section \ref{cc}.

\subsection{}\la{3.1}
Let $R$ be a reduced, irreducible root system in a (complexified) Euclidean vector space $V$ with an inner product denoted as $\dpr{-,-}$, and $W$ be the Weyl group of $R$, generated by the orthogonal reflections $s_\alpha$, $\alpha\in R$. We write $R^\vee=\{\alpha^\vee\}$ for the dual system formed by the coroots $\alpha^\vee=2\alpha/\dpr{\alpha, \alpha}$. Let $a_1, \dots, a_n$ be a fixed basis of simple roots in $R$, associated with a decomposition $R=R_+\sqcup R_-$. We have the coroot and coweight lattices: $Q^\vee=\sum_{i=1}^n \Z a_i^\vee$ and $P^\vee=\sum_{i=1}^n \Z b_i$, where the fundamental coweights $b_i\in V$ are defined by $\dpr{a_i, b_j}=\delta_{ij}$. We write $P^\vee_+:=\sum_{i=1}^n\Z_{\ge 0} b_i$ for the cone of {\it dominant} coweights, .

The \emph{affine Weyl group} is defined as $\WW=W\ltimes t(Q^\vee)$, where $t(Q^\vee)$ denotes the group of translations $t(\lambda)$, $\lambda\in Q^\vee$ acting on $V$ by 
\begin{equation}\la{tl}
t(\lambda)x=x-c\lambda\,,
\end{equation}
where $c$ is a fixed parameter. The \emph{extended} affine Weyl group is $\Wh:=W\ltimes t(P^\vee)$. The group $\Wh$ acts natrually on the ring of meromorphic functions $\c(V)$ by 
\begin{equation}\la{fa}
\hat wf(x)=f(\hat w^{-1}x)\qquad\text{for}\  f(x)\in\c(V)\,,\ \hat w\in\Wh\,.
\end{equation}
 In particular, a translation $t(\lambda)$, $\lambda\in P^\vee$ acts on functions by
\begin{equation}\la{tact}
t(\lambda)f(x)=f(x+c\lambda)\,.
\end{equation} 
Wrting $q=e^c$, we have $t(\lambda)=q^{\partial_\lambda}$. We form a crossed product $\c(V)*\Wh$ which we view as a subalgebra of $\End_\c(\c(V))$, with $\c(V)$ acting on itself by multiplication. Inside $\c(V)*\Wh$ we have an algebra ${\D}_q$ generated by $\c(V)$ and $t(P^\vee)$; this is the algebra of difference operators on $V$. Clearly, $\c(V)*\Wh\cong {\D}_q*W$, with every element admitting a unique presentation as
\begin{equation}\la{genf}
a=\sum_{w\in W}a_ww\,,\qquad a_w\in{\D}_q\,.
\end{equation}  

Let $\VV$ denote the space of affine-linear functions on $V$. We identify $\VV$ with $V\oplus\c\delta$, where vectors in $V$ are considered as linear functionals on $V$ via the scalar product $\dpr{-,-}$ and where $\delta\equiv c$ on $V$ (so $e^\delta=e^c=q$).
Let
\begin{equation}\la{rrelq}
\Ra=\{\aalpha=\alpha+k\delta\,, k\in\Z\,, \alpha\in R\}\subset \VV
\end{equation} 
be the affine root system associated with $R$. The action of $\Wh$ on $\VV\subset \c(V)$ permutes affine roots. For any $\aalpha=\alpha+k\delta$ we have the orthogonal reflection with respect to the hyperplane $\aalpha(x)=0$ in $V$,
\begin{equation*}
s_{\aalpha}(x)=x-\aalpha(x)\alpha^\vee\,,\quad x\in V\,.
\end{equation*}
We extend the set of simple roots $a_i$ to a basis in $\Ra$ by adding $a_0=\delta-\varphi$, where $\varphi$ is the highest root in $R_+$. 
Then the reflections $s_{i}=s_{a_i}$, $i=0,\dots, n$ generate the group $\WW$, and the {\it length} $l(w)$ of $w\in\WW$ is defined as the length $l$ of a reduced decomposition 
\begin{equation}\la{red}
w=s_{i_1}\dots s_{i_l}\,,\quad\text{with}\ 0\le i_k\le n\,.
\end{equation}
Let $\Omega$ be the subgroup of the elements $\pi\in\Wh$ which map the basis $a_0, \dots, a_n$ to itself. It is known that $\Omega$ is an abelian group, isomorphic to $P^\vee/Q^\vee$, and the extended affine Weyl group is isomorphic to $\WW\rtimes\Omega$. Each $w\in\Wh$ admits a unique presentation as $w=\tilde w \pi$ with $\tilde w\in\WW$ and $\pi\in\Omega$. We use this to extend the notion of the length from $\WW$ to $\Wh$ by setting $l(\tilde w\pi)=l(\tilde w)$, so $l(\pi)=0$ for all $\pi\in\Omega$. 

The \emph{braid group} $\mathfrak B$ of $\Wh$ is the group with generators $T_w$, $w\in\Wh$, and relations 
\begin{equation*}
T_vT_w=T_{vw}\quad\text{if}\quad l(v)+l(w)=l(vw)\,.
\end{equation*} 
Write $T_i:=T_{s_i}$ for $i=0, \dots, n$. Then for any reduced decomposition $w=s_{i_1}\dots s_{i_l}\pi$ we have
$T_w=T_{i_1}\dots T_{i_l}T_\pi$ . It follows that $\mathfrak B$ is generated by $T_i$, $i=0, \dots, n$ and $T_\pi$, $\pi\in\Omega$, subject to the following relations \cite[(3.1.6)]{M03}:
\begin{align}\la{braid}
&T_{i}T_{j}\dots=T_{j}T_{i}\dots \qquad\text{for $i\ne j$, with $m_{ij}$ factors on either side}\,,\\\la{braid1}
&T_\pi T_{\pi'}=T_{\pi \pi'}\quad\text{for}\ \pi, \pi'\in\Omega\,,\\\la{braid2}
&T_\pi T_iT_\pi^{-1}=T_j\qquad\text{if}\ \pi s_i\pi^{-1}=s_j\,.
\end{align}
Here $m_{ij}=2, 3, 4, 6$ is the order of $s_is_j\in \WW$.

The braid group $\mathfrak B$ contains an abelian subgroup $\{Y^\lambda\,|\,\lambda\in P^\vee\}$ \cite[3.2]{M03}. Namely, for dominant $\lambda$ we define $Y^\lambda=T_{t(\lambda)}$ and then extend this definition to all $\lambda\in P^\vee$ by setting $Y^\lambda=Y^\mu(Y^\nu)^{-1}$ whenever $\lambda=\mu-\nu$ with dominant $\mu, \nu$. 

Choose nonzero parameters $\tau_i$, $i=0, \dots, n$ such that $\tau_i=\tau_j$ if $s_i$ and $s_j$ are conjugated in $\Wh$. The (extended) \emph{affine Hecke algebra} $\HH$ is the quotient of the group algebra $\c\mathfrak B$ by relations 
\begin{equation}\la{hecke}  
(T_i-\tau_i)(T_i+\tau_i^{-1})=0\,,\qquad i=0,\dots n\,.
\end{equation}
The image of $T_i$ (resp. $T_w$, $Y^\lambda$) in $\HH$ will be denoted by the same symbol $T_i$ (resp. $T_w$, $Y^\lambda$). By \cite[(4.1.3)]{M03},  the elements $T_w$, $w\in\Wh$ form a $\c$-basis of $\HH$.

\subsection{}
The algebra $\widehat{\mathfrak H}$ can be realized by difference-reflection operators, as was observed by Cherednik. This is an injective algebra map 
\begin{equation*}
\beta:\, \widehat{\mathfrak H}\,\to\, {\D}_q*W\,,
\end{equation*}
called the \emph{basic representation} \cite[(4.3.10)]{M03}. To describe it, let us extend the set of parameters $\tau_i$ to $\tau_\alpha$, $\alpha\in\Ra$ so that $\tau_\alpha=\tau_{w(\alpha)}$ for $w\in\Wh$, and introduce functions $\mathbf{c}_\alpha$ as follows:
\begin{equation}\la{cal}
\mathbf{c}_\alpha=\frac{\tau_\alpha^{-1}-\tau_\alpha e^{\alpha}}{1-e^{\alpha}}\,,\qquad\alpha\in\Ra\,.
\end{equation}
\begin{theorem}[cf. (4.3.10), (4.3.12) \cite{M03}]\la{beta} The extended affine Hecke algebra admits a faithful representation $\beta: \widehat{\mathfrak H} \to \D_q*W$ such that   
\begin{align}\la{br}
&\beta:\,T_i\mapsto \tau_i+\mathbf{c}_i(x)(s_{i}-1)\,, \qquad \mathbf{c}_i=\mathbf{c}_{a_i} \qquad (i=0, \dots, n)\,,\\
&\beta:\,T_\pi\mapsto \pi\qquad\text{for all}\quad \pi\in\Omega\,.
\end{align}
\end{theorem}
Let us identify $\widehat{\mathfrak H}$ with its image under $\beta$, so the affine Hecke algebra from now on will be viewed as a subalgebra of ${\D}_q*W$. The \emph{Cherednik operators}, by definition, are the images of $Y^\lambda$ under $\beta$. They form a commutative family of difference-reflection operators, and should be viewed as $q$-analogues of the Dunkl operators. In comparison, they are rather complicated. For example, for a dominant $\lambda$,  one obtains $Y^\lambda$ by first finding a reduced decomposition $t(\lambda)=s_{i_1}\dots s_{i_l}\pi$ and then writing the product $Y^\lambda=T_{i_1}\dots T_{i_l}T_\pi$ in the basic representation. Below we will often write $Y^\lambda$ in terms of the elements $\RR(\alpha)$ (``$R$-matrices'') defined by
\begin{equation}\la{rdef}
\RR(\alpha)=\tau_{\alpha} s_{\alpha} +\mathbf{c}_{\alpha}(x)(1-s_{\alpha})\,,\quad \alpha\in \Ra\,. 
\end{equation} 
The following property of these elements is important: 
\begin{equation*}
w\RR(\alpha)w^{-1}=\RR(w(\alpha))\,,\qquad\text{for any}\  w\in\WW\,. 
\end{equation*}
Using this and the fact that $\RR(a_i)=T_is_i$ for $i=0, \dots, n$, it is straightforward to rewrite $Y^\lambda$ in terms of $\RR(\alpha)$ instead of $T_i$. 

The commutative subalgebra generated by the Cherednik operators will be denoted as $\c[Y]$, so elements $f(Y)\in\c[Y]$ are arbitrary linear combinations of $Y^\lambda$, $\lambda\in P^\vee$. Inside $\c[Y]$ we have a subalgebra $\c[Y]^W$, spanned by the orbitsums $f=\sum_{\mu\in W\lambda}Y^{\mu}$. 

\subsection{}
The (finite) Hecke algebra $\mathfrak H$ of $W$ is a subalgebra of $\HH$, generated by $T_i$, $i=1, \dots, n$. 
That is, $\mathfrak H$ is generated by $T_1, \dots, T_n$ which satisfy the relations \eqref{braid}, \eqref{hecke}. 
We have an isomorphism $\HH\cong \mathfrak H\otimes \c[Y]$, as vector spaces. The cross-relations between $T_i$ and $Y^\lambda$ are the so-called Lusztig relations \cite[(4.2.4)]{M03}. They imply the following property, cf. \cite[(4.2.8)]{M03}:
\begin{equation}\la{dl}
T_if(Y)=f(Y)T_i\quad \text{for all $f\in\c[Y]$ such that $s_{i}.f=f$}\,, 
\end{equation}
where the action of $s_{i}$ on $f\in \c[Y]$ is by $s_{i}.Y^{\lambda}=Y^{s_{i}(\lambda)}$.

For any $w\in W$ with a reduced decomposition $w=s_{{i_1}}\dots s_{{i_r}}$, we define $\tau_w=\tau_{i_1}\dots \tau_{i_l}$. The following is a Hecke-algebra analogue of the symmetrizer \eqref{ew}:
\begin{equation}\la{etau}
e_\tau=\frac{1}{\sum_{w\in W}\tau_w^2}\sum_{w\in W} \tau_wT_w\,.
\end{equation}
By \cite[5.5.17]{M03}, 
\begin{equation}\la{tet}
T_ie_\tau=e_\tau T_i=\tau_i e_\tau\,,\qquad e_\tau^2=e_\tau\,.
\end{equation}
Let us also define parabolic symmetrizers. Given $I\subset \{1,\dots, n\}$, we define $W_I\subset W$ as the subgroup generated by $s_{i}$ with $i\in I$. It is known that $W_I$ is isomorphic to a Weyl group with Dynkin diagram obtained by removing the vertices $j\notin I$ from the diagram of $W$. Similarly, define a parabolic subalgebra $\mathfrak H_{I}\subset \mathfrak H$ as the subalgebra generated by $T_i$ with $i\in I$. By \cite[4.4.7]{GeP}, $\mathfrak H_I$ is spanned by $T_w$ with $w\in W_I$ and it is isomorphic to the Hecke algebra of $W_I$ with parameters $\tau_i$, $i\in I$. As a result, if we define 
\begin{equation}\la{ewj}
e_{\tau, I}=\frac{1}{\sum_{w\in W_I}\tau_w^2}\sum_{w\in W_I} \tau_wT_w\,,
\end{equation}
then 
\begin{equation*}
T_ie_{\tau, I}=e_{\tau, I} T_i=\tau_i e_{\tau, I}\quad\forall i\in I\,,\qquad e_{\tau, I}^2=e_{\tau, I}\,.
\end{equation*}

\subsection{}\la{mrh} The Hamiltonians of the quantum relativistic Calogero--Moser system based on a root system $R$ are given by certain difference operators. The following construction for them was given by Cherednik. We start by picking an arbitrary $f(Y)\in\c[Y]^W$; then there is a unique difference operator ${L}_f\in{\D}_q$ such that $f(Y)e={L}_fe$, where $e$ is the symmetrizer \eqref{ew}. By \eqref{dl}, $(T_i-\tau_i)f(Y)=f(Y)(T_i-\tau_i)$. From \eqref{br} we have $T_i-\tau_i=\mathbf{c}_i(x)(s_{i}-1)$, therefore,
\begin{equation}\la{sym}
\mathbf{c}_i(x)(s_{i}-1){L}_fe=(T_i-\tau_i)f(Y)e=f(Y)(T_i-\tau_i)e=f(Y)\mathbf{c}_i(x)(s_{i}-1)e=0\,.
\end{equation}  
As a result, $(s_{i}-1){L}_fe=0$, from which ${L}_f^{s_{i}}={L}_f$ for all $i$, i.e. ${L}_f$ is a $W$-invariant difference operator, with $e{L}_fe=e{L}_f={L}_fe=f(Y)e=ef(Y)e$. (Warning: $f(Y)e\ne ef(Y)$!) These operators pairwise commute, for if ${L}_f, L_g\in{\D}_q^W$ are constructed from $f, g\in \c[Y]^W$ then
\begin{equation*}
{L}_fL_ge={L}_feL_ge=f(Y)eL_ge=f(Y)L_ge=f(Y)g(Y)e\,,
\end{equation*} 
which is invariant under exchanging $f$ and $g$. This gives an algebra map
\begin{equation}\la{rcm}
\c[Y]^W\to {\D}_q^W\,,\quad f\mapsto {L}_f\,,
\end{equation}
hence a quantum integrable system with commuting Hamiltonians ${L}_f$, $f\in\c[Y]^W$.
In the case $R=A_{n-1}$, this is the trigonometric Ruijsenaars system \cite{R87}. In general, the difference operators ${L}_f$, $f\in\c[Y]^W$ are complicated; some of them are known explicitly \cite{M87, Ko, vD, vDI, vDE}. 

\subsection{} Quantum Lax pairs are now constructed in the same way as before. First, the algebra ${\D}_q$ of difference operators acts naturally on $\c(V)$, with translations $t(\lambda)$, $\lambda\in P^\vee$ acting by \eqref{tact}. Consider the induced module 
\begin{equation}\la{mod}
M=\Ind_{{\D}_q}^{{\D}_q*W}\,\c(V)\,.
\end{equation}
We have $M\cong \c W\otimes \c(V)$ as vector spaces, so $\End_\c(M)$ can be identified with \emph{operator-valued} matrices of size $|W|$. As a result, the (left) action of ${\D}_q*W$ on $M$ gives a representation
\begin{equation}\la{rep}
{\D}_q*W\to \Mat(|W|, {\D}_q)\,.
\end{equation}
Again, any $a\in{\D}_q^W$ is represented by $a\mathbb{1}$.
 
\medskip

Now pick a Cherednik operator $Y^{\lambda}$, $\lambda\in P^\vee$ and choose an element $f(Y)$, $f\in\c[Y]^W$, writing it as $f(Y)=\sum_{w\in W} a_ww$ with $a_w\in{\D}_q$. Then
\begin{equation}\la{acalg1}
f(Y)={L}_f+\widehat A\,,\quad\text{where}\quad {L}_{f}=\sum_{w\in W} a_w\,,\ \ \widehat A := \sum_{w\in W} a_w(w-1)\,.
\end{equation}
Note that ${L}_f$ is one of the Hamiltonians \eqref{rcm}, so ${L}_f\in{\D}_q^W$. We now set $\mathcal L$, $\mathcal H$, $\mathcal A$ to be the images of $Y^\lambda$, $L_f$ and $\widehat A$, respectively, under \eqref{rep}. Since $Y^\lambda$ and $f(Y)=L_f+\widehat A$ commute, we obtain \eqref{leq} with $\mathcal H={L}_f\, \mathbb{1}$, that is, a quantum Lax pair of size $|W|$.

\medskip

\subsection{}\la{climq}
The classical limit corresponds to $q=e^c\to 1$, and the procedure is similar to \ref{cld}. Namely, we set $c=-\mathrm{i}\hbar\beta$, with some fixed $\beta$, and consider the algebra
\begin{equation*}
\Ah*W=\c(V)[t_1^{\pm 1},  \dots , t_n^{\pm 1}][[\hbar]]*W\,,\qquad t_k:=t(b_k)\,,
\end{equation*}
where $b_k$ are the fundamental coweights. We have 
\begin{equation*}
[t_k, f]=\sum_{l=1}^\infty (-\mathrm{i}\hbar\beta)^l\partial_{k}^l(f) t_k\,,\qquad \forall\  f\in\c(V)\,,
\end{equation*}
where we denote $\partial_k:=\partial_{b_k}$.

Consider a commutative algebra $A_0=\c(V)[P^\vee]$. Elements of $\c[P^\vee]$ are linear combinations with coefficients in $\c(V)$ of $e^\lambda$, $\lambda\in P^\vee$. We will view $e^\lambda$ as a function of the classical momenta $p\in V$ by setting $e^\lambda:=e^{\beta p_\lambda}$, where $p_\lambda$ is the momentum in direction $\lambda$. Writing $p_k:=p_{b_k}$, 
we have an algebra isomorphism
\begin{equation*}
\eta_0:\ \Ah*W/(\hbar\Ah*W)\to A_0*W\,,\qquad f\mapsto f\,,\ \ t_k\mapsto e^{\beta p_{k}}\,,\ \ w\mapsto w\,.
\end{equation*}
We may view $\Ah$ (resp. $\Ah*W$) as a formal deformation of $A_0$ (resp. $A_0*W$). The algebra $A_0$ is commutative, with the Poisson bracket determined by
$[a,b]=\mathrm{i}\hbar\{\eta_0(a), \eta_0(b)\}+o(\hbar)$ for $a, b\in \Ah$. Explicitly, $\{e^{\beta p_k}, f\}=-\beta\partial_k(f)e^{\beta p_k}$. 
For $a\in\Ah*W$, we call $\eta_0(a)$ the \emph{classical limit} of $a$.

The classical Cherednik operators  $Y^{\lambda}_{c}:=\eta_0(Y^\lambda)$ can be computed by using the classical version of the basic representation, cf. \cite[Section 5]{Ob}; let us write $f(Y)_{c}$ for the classical version of $f(Y)\in\c[Y]$. We then have the following analogue of Lemma \ref{le}.   
\begin{lemma}
For any $f(Y)\in\c[Y]^W$, when writing $f(Y)_{c}=\sum_{w\in W} a_ww$ with $a_w\in A_0$, we have $a_w=0$ for $w\ne \id$.  
\end{lemma}   
\begin{proof} Indeed, by \cite[Lemma 5.1]{Ob} the algebra $\c[Y]^W$ belongs to the center of the double affine Hecke algebra at the classical level $q=1$. Therefore, the same proof as in \cite[Lemma 2.2]{EFMV} applies.
\end{proof}
This tells us that the classical limit of $\widehat A$ in \eqref{acalg1} is zero. Therefore, $\hbar^{-1}\widehat A$ has a well-defined classical limit, and so the classical Lax pair is obtained in the same way as in
\ref{cld}.

\subsection{}\la{3.6}
To get a Lax pair of a smaller size, we take $\lambda$ to be on the boundary of the Weyl chamber, i.e. $\dpr{\lambda, a_i}\ge 0$ for $i=1, \dots, n$, with at least one $i$ such that $\dpr{\lambda, a_i}=0$. Set $I=\{i\,|\, \dpr{a_i, \lambda}=0\}$; then the stabiliser $W'=W_I$ of $\lambda$ is generated by $s_{i}$ with $i\in I$. Write
\begin{equation}\la{es} 
e'=\frac{1}{|W'|}\sum_{w\in W'} w\,, \qquad M'=e'M\cong e'\c W\otimes \c(V)\,.
\end{equation}
By \eqref{dl}, we have $T_iY^\lambda =Y^\lambda T_i$ for all $i\in I$. Then, similarly to \eqref{sym}, we have $(s_{i}-1)Y^\lambda e'=0$ for $i\in I$, implying $e'Y^\lambda e'=Y^\lambda e'$. This tells us that the action of $Y^\lambda$ preserves $M'=e'M$. Similarly, $f(Y)$ for any $f\in\c[Y]^{W}$ preserves $M'$. Therefore, the operators $Y^\lambda$, $f(Y)={L}_f+\widehat A$, as well as ${L}_f$, can be restricted onto $M'$, resulting in a quantum Lax pair $\mathcal L$, $\mathcal A$ of size $|W/W'|$.

\subsection{}
There is a $q$-analogue of the formula \eqref{fi} for constructing first integrals. It generalises a formula discovered in the $\GL_n$-case by Nazarov and Sklyanin \cite{NS17}, see \ref{ns} below. As before, we choose a dominant $\lambda\in P^\vee$ and write $W'=W_I$ for the stabiliser of $\lambda$ and $e_{\tau, I}$, $e'$ for the symmetrizers \eqref{ewj}, \eqref{es}. Denote by $R'\subset R$ the root system of $W'$, and choose coset representatives $w_i$ for $W'\backslash W$. We write elements of $M'=e'M$ as linear combinations of $e'w_if_i$ with $f_i\in \c(V)$. 
Let $\mathcal L$ be a matrix of size $|W/W'|$ which represents the action of $Y^\lambda$ on $M'$.

Recall the functions $\mathbf{c}_\alpha (x)$ \eqref{cal}.
Introduce a pair of row/column vectors with $l:=|W/W'|$ components as follows:  
\begin{equation}\la{uva}
\vv=(1,\dots, 1)^T\,,\quad \uu=(\phi_1, \dots, \phi_l)\,,\quad \phi_i(x):=\phi(w_ix)\,,\quad\text{where}\ \phi(x)=\prod_{\alpha\in R_+\setminus R'_+} \mathbf{c}_\alpha(-x)\,. 
\end{equation}

\begin{prop}\la{intq} 
The difference operators $\widehat H_k=\uu\mathcal L^k \vv$, $k\in\Z$ belong to the algebra of quantum Hamiltonians ${L}_f$, $f\in \c[Y]^W$. 
\end{prop}

\begin{proof} First, we claim that for any $f\in\c[Y]^W$ the elements $e_\tau Y^\lambda e$ and $ef(Y)e={L}_fe=e{L}_f$
commute:
\begin{equation}\la{co}
[e_\tau Y^\lambda e, ef(Y)e]=0\,.
\end{equation}
Indeed, by \eqref{tet} and \eqref{br} we have $0=(T_i-\tau_i)e_\tau=\mathbf{c}_i(x)(s_{i}-1)e_\tau$. Thus, $s_{i}e_\tau=e_\tau$ for all $i$, implying 
\begin{equation}\la{etw} 
we_\tau=e_\tau\quad\forall  w\in W\,.
\end{equation}
Therefore, $ee_\tau=e_\tau$. Next, recall that by \eqref{dl} any $f(Y)\in \c[Y]^W$ commutes with all $T_i$, so $e_\tau f(Y)=f(Y)e_\tau$. Now we have
\begin{equation*}
(e_\tau Y^\lambda e) (e {L}_f e)= e_\tau Y^\lambda {L}_fe =e_\tau Y^\lambda f(Y)e\,,
\end{equation*}
and similarly,
\begin{equation*}
(e{L}_f e)(e_\tau Y^\lambda e)=f(Y)e e_\tau Y^\lambda e=f(Y)e_\tau Y^\lambda e=e_\tau f(Y) Y^\lambda e\,,
\end{equation*}
which implies \eqref{co}.

Next, by \cite[(5.5.14), (5.5.15)]{M03} we have $e_\tau=e\,\mathbf{c}_+$ where $\mathbf{c}_+(x)=\prod_{\alpha\in R_+} \mathbf{c}_\alpha(-x)$ (up to a constant factor). Similarly,      
$e_{\tau, I}=e'\,\mathbf{c}'_+$ with $\mathbf{c}'_+(x)=\prod_{\alpha\in R'_+} \mathbf{c}_\alpha(-x)$. As a result, the action of the symmetrizer $e_\tau$ on $M'=e'M$ is calculated as follows:  
\begin{equation}\la{eta}
e_\tau e'= e_\tau e_{\tau, I}/\mathbf{c}'_+=e_\tau /\mathbf{c}'_+=e\mathbf{c}_+/\mathbf{c}'_+=e\phi\,,\qquad \phi(x)=\prod_{\alpha\in R_+\setminus R'_+} \mathbf{c}_\alpha(-x)\,.
\end{equation}
The function $\phi$ is $W'$-invariant. If $f=\sum_{i} e' w_i f_i$ is an element of $M'$, then $\phi f=\sum_{i} e' w_i (\phi)^{w_i^{-1}}f_i$, where $\phi^{w_i^{-1}}(x)=\phi(w_i x)$.
It follows that, up to a constant factor, $e_\tau=e\phi$ acts on $M'$ by a matrix $\vv\uu$ as given in \eqref{uva}, and $e$ acts as $\vv\ww$, where $\ww$ is a row of ones. 
We conclude that, up to a constant factor, $e_\tau Y^{k\lambda} e$ acts as $\vv \uu\mathcal L^k \vv\ww=(\uu\mathcal L^k \vv)e$, and it is clear from the construction that $\uu\mathcal L^k \vv$ is $W$-invariant. Meanwhile, $ef(Y)e$ acts on $M'$ as ${L}_fe$. Hence, $\uu\mathcal L^k \vv$ and ${L}_f$ commute by \eqref{co}.

This proves that each of $\widehat H_k=\uu\mathcal L^k \vv$ is a $W$-invariant difference operator commuting with all ${L}_f$, $f\in\c[Y]^W$. By \cite[Theorem 3.15]{LS}, $\widehat H_k$ must belong to the algebra of the Hamiltonians ${L}_f$, $f\in\c[Y]^W$. In particular, $\widehat H_k$ pairwise commute. 
\end{proof}

\subsection{}\la{aq}
Let us calculate a quantum Lax pair for the trigonometric Ruijsenaars system. This corresponds to the case of a root system $R=A_{n-1}$ (or more precisely, to the $\GL_n$-case). This will provide an alternative derivation of some of the results of Nazarov and Sklyanin \cite{NS17}. Let us first describe the Cherednik operators in this case. 

We take $V=\c^n$, with the orthonormal basis $\epsilon_1, \dots, \epsilon_n$ and the associated coordinates $x_1, \dots, x_n$. As in \ref{typea}, the group $W=\mathfrak{S}_n$ acts on $V$ by permuting the basis vectors, and we write $W'=\mathfrak{S}_{n-1}$ for the subgroup fixing $\epsilon_1$. As before, $e, e'$ are the corresponding symmetrizers. 
The roots in $R$ are $\alpha=\epsilon_i-\epsilon_j$ with $i\ne j$.
The role of $P^\vee$ in the $\GL_n$-case is played by the lattice $\Lambda=\sum_{i=1}^n \Z\epsilon_i$. There is only one coupling parameter $\tau$, so  $\tau_\alpha=\tau$ for all $\alpha$.
Fix $q=e^c$ and consider the algebra of difference operators ${\D}_q=\c(x)\ltimes t(\Lambda)$, generated by $\c(x)$ and the shift operators $t(\epsilon_k)=q^{\partial_k}$.   For $i\ne j$ we set $\RR_{ij}:=\RR(\epsilon_i-\epsilon_j)$ according to \eqref{rdef}. Explicitly,  
\begin{equation}\la{rij}
\RR_{ij}=a_{ij}+b_{ij}\ss_{ij}\,,\qquad a_{ij}=a({x_i-x_j})\,,\quad b_{ij}=b({x_i-x_j})\,,
\end{equation}
where  
\begin{equation}\la{ab}
a(z)=\frac{\tau^{-1}-\tau e^z}{1-e^z}\,,\quad b(z)=\frac{\tau-\tau^{-1}}{1-e^z}\,.
\end{equation}
The Cherednik operators are the following commuting elements of ${\D}_q*W$ \cite{C1, BGHP}, cf. \cite[(2.7)]{NS17}:
\begin{equation}\la{yi}
Y_i=\RR_{i, i+1}\RR_{i, i+2}\dots \RR_{i, n}\,t(\epsilon_i)\,\RR_{1i}^{-1}\dots \RR_{i-1, i}^{-1}\quad (i=1, \dots, n)\,.
\end{equation}  
The operators ${L}_f$, $f\in\c[Y_1^{\pm 1}, \dots, Y_n^{\pm 1}]^{\mathfrak{S}_n}$ are obtained according to \ref{mrh}. They can be given explicitly \cite{R87, M87}; the simplest one is
\begin{equation}\la{mr1}
{L}_f=\sum_{i=1}^n \left(\prod_{l\ne i} a_{il}\right)\,t(\epsilon_i)\,,\qquad f=Y_1+\dots +Y_n\,.
\end{equation}  
A Lax matrix will be constructed from $Y_1=Y^{\epsilon_1}$:
\begin{equation}\la{y1}
Y_1=\RR_{12}\RR_{13}\dots \RR_{1n}\,t(\epsilon_1)\,.
\end{equation}
As explained in \ref{3.6}, 
$Y_1$ acts on $M'=e'M$. 
We have the following result, proved by Nazarov and Sklyanin (cf. a similar statement in the elliptic case in \cite[Lemma 4.4]{KH97}).

\begin{lemma}[\cite{NS17}, Proposition 2.4]\la{ns}
Upon restricting onto ${M'}$, $Y_1$ coincides with 
\begin{equation*}
(A+\sum_{i\ne 1}^n B_i\ss_{1i})\,t(\epsilon_1)\,,\quad\text{where}\quad A=\prod_{l\ne 1}^n a_{1l}\,,\quad B_i=b_{1i}\prod_{l\ne 1, i} a_{il}\,.
\end{equation*}
\end{lemma}

In \cite{NS17} this is proved by a cleverly organised induction. We give a simpler proof using symmetry arguments. 
\begin{proof}
Since $t(\epsilon_1)$ commutes with $e'$, it is enough to show that
\begin{equation}\la{ide}
\RR_{12}\RR_{13}\dots \RR_{1n}e'=(A+\sum_{i\ne 1}^n B_i\ss_{1i})e'\,,
\end{equation}
with $A$, $B_i$ as stated.
Using the definition of $\RR_{1i}$, we expand the product:
\begin{equation*}
\RR_{12}\RR_{13}\dots \RR_{1n}=\sum_{2\le i_1<\dots <i_k\le n} g_{i_1\dots i_k} \ss_{1i_1}\dots \ss_{1i_k}\,,\quad\text{with some}\  g_{i_1\dots i_k}\in\c(x)\,. 
\end{equation*}
It is easy to see that 
$\ss_{1i_1}\dots \ss_{1i_k}\,e'$ equals $\ss_{1i_k}e'$ (or $e'$ if the set $\{i_1, \dots, i_k\}$ is empty). It follows that in \eqref{ide} one has $A=\prod_{l\ne 1}^n a_{1l}$.
Similarly, $B_2$ is determined from
\begin{equation*}
B_2 \ss_{12}=(b_{12}\ss_{12})a_{13}\dots a_{1n}=b_{12}\left(\prod_{l\ne 1,2}a_{2l}\right)\ss_{12}\,.
\end{equation*}
Finally, using Lemma \ref{alem}$(1)$ we find that $B_i=B_2^{\ss_{2i}}=b_{1i}\prod_{l\ne 1, i} a_{il}$, as needed.
\end{proof}

Now a Lax matrix $\mathcal L$ can be calculated using Lemma \ref{alem}. This gives 
\begin{equation*}
\mathcal L_{ij}=
\begin{cases}
\left(\prod_{l\ne j} a_{jl}\right) t(\epsilon_j)\quad&\text{for}\ i=j\\
\left(\prod_{l\ne i,j}a_{jl}\right)b_{ij} t(\epsilon_j)\quad&\text{for}\ i\ne j\,.
\end{cases}
\end{equation*}
This is the quantum Lax matrix obtained in \cite{NS17} (where it is denoted as $Z$), see also \cite{Ha}. Here $t(\epsilon_j)=q^{\partial_j}=e^{\beta\widehat p_j}$, assuming $q=e^{-\mathrm{i}\hbar\beta}$.
The classical Lax matrix is obtained by replacing $\hat p_j$ with the classical momentum $p_j$: 
\begin{equation*}
L_{ij}=
\begin{cases}
\left(\prod_{l\ne j} a_{jl}\right) e^{\beta p_j}\quad&\text{for}\ i=j\\
\left(\prod_{l\ne i,j}a_{jl}\right)b_{ij} e^{\beta p_j}\quad&\text{for}\ i\ne j\,.
\end{cases}
\end{equation*}
It is equivalent to the Lax matrix \cite[(3.19)]{R87} for the trigonometric Ruijsenaars--Schneider system.    

\subsection{}\la{ns}
We can now apply Proposition \ref{intq}. In our case we have $R_+=\{\epsilon_i-\epsilon_j\,|\, 1\le i<j\le n\}$ and $R'_+=\{\epsilon_i-\epsilon_j\,|\, 2\le i<j\le n\}$. The function $\mathbf{c}_\alpha(-x)$ for $\alpha=\epsilon_i-\epsilon_j$ coincides with $a_{ji}$ as given in \eqref{ab}. We can write down the function $\phi$ and the row/column vectors $\uu, \vv$ \eqref{uva}:
\begin{equation*}
\phi=\prod_{l\ne 1}^n a_{l1}\,,\qquad  \phi_i=\phi^{\sigma_{1i}}=\prod_{l\ne i}^n a_{li}\,, \quad \uu=(\phi_1, \dots, \phi_n)\,,\quad\vv=(1, \dots, 1)^T\,.
\end{equation*}  
According to Proposition \ref{intq}, the difference operators $\uu\mathcal L^j\, \vv$, $j\in\Z$ belong to the commutative algebra of the Macdonald--Ruijsenaars operators ${L}_f$, $f\in\c[Y_1^{\pm 1}, \dots, Y_n^{\pm 1}]^{\mathfrak{S}_n}$. 

We can compare this with one of the results of Nazarov and Sklyanin. Namely, \cite[Corollary 2.6]{NS17} claims that the quantities ${U}{Z}^jE$ belong to the algebra of the Macdonald--Ruijsenaars operators. Here $Z$ and $E$ coincide with our $\mathcal L$ and $\vv$, but the row vector ${U}=(U_1, \dots, U_n)$ is different from our $\uu$ (in particular, the expressions for $U_i$ contain the shift operators $t(\epsilon_i)$, see 
\cite[(2.19)]{NS17}). These two formulas may seem different, however, it can be checked that up to a constant factor, ${U}=\uu\mathcal L$. Therefore, $U{Z}^jE=\uu \mathcal L^{j+1} \vv$, so both results agree.   

\subsection{}
A Lax partner $\mathcal A$ for the Hamiltonian ${L}_f$ \eqref{mr1} corresponding to $f=Y_1+\dots +Y_n$ can be calculated by a similar method. We skip the details (see a similar calculation for the elliptic case in \ref{arelel}), and will only state the result:
\begin{equation*}
\mathcal A_{ij}=\left(\prod_{l\ne i,j}a_{jl}\right)(b_{ij} t(\epsilon_j)-t(\epsilon_j)b_{ij})\quad\text{for}\ i\ne j\,,\qquad \mathcal A_{ii}=-\sum_{j\ne i}\mathcal A_{ij}\,.
\end{equation*}
To calculate its classical counterpart, we use that $t(\epsilon_j)=e^{-\mathrm{i}\hbar \beta\partial_j}$ and
\begin{equation*}
b_{ij} t(\epsilon_j)-t(\epsilon_j)b_{ij}=\mathrm{i}\hbar \beta \frac{\partial b_{ij}}{\partial x_j}t(\epsilon_j)+o(\hbar)\,.
\end{equation*}
Therefore, the classical limit $A$ of $(\mathrm{i}\hbar)^{-1}\mathcal A$ is given  by
\begin{equation*}
A_{ij}=\left(\prod_{l\ne i,j}a_{jl}\right)\beta \frac{\partial b_{ij}}{\partial x_j} e^{\beta p_j} \quad\text{for}\ i\ne j\,,\qquad A_{ii}=-\sum_{j\ne i}A_{ij}\,.
\end{equation*}

\section{Lax matrix for the Koornwinder--van Diejen system}\la{cc}

\subsection{} Let us describe the Cherednik operators in the case $C^\vee C_n$, corresponding to the Koornwinder--van Diejen system \cite{Ko, vD}. 
We will follow \cite{St1} fairly closely, so the reader should consult that paper for further details. Let $V=\c^n$ with the standard orthonormal basis $\{\epsilon_i\}_{i=1}^{n}$ and the associated coordinates $x_i$. Let  $R$ be the root system of type $C_n$,
\begin{equation}\la{cn}
R=\{\pm 2\epsilon_i\,|\,1\le i\le n\}\cup \{\pm \epsilon_i\pm \epsilon_j\,|\, 1\le i<j\le n\}\,.
\end{equation}
The Weyl group $W=\mathfrak{S}_n\ltimes \{\pm 1\}^n$ of $R$ consists of the transformations that permute the basis vectors $\epsilon_i$ and change their signs arbitrarily. As in \ref{3.1}, we write $\VV=V\oplus\c\delta$ for the space of affine-linear functions on $V$, with $\delta\equiv c$ on $V$. 
Let $\Ra$ be the affine root system associated with $R$ \eqref{rrelq}. We choose a basis of simple roots in $\Ra$,
\begin{equation*}
a_0=\delta-2\epsilon_1\,,\qquad a_i=\epsilon_i-\epsilon_{i+1}\quad (i=1,\dots, n-1)\,,\qquad \alpha_n=2\epsilon_n\,.
\end{equation*}
The afiine Weyl group $\WW$ of $R$ is generated by $s_{i}=s_{a_i}$, $i=0,\dots, n$. The action of the generators in coordinates $x_i$ on $V$ looks as follows:
\begin{align}\la{wwact}
s_{0}\,(x_1,\dots, x_n)&=(c-x_1, x_2, \dots, x_n)\,,\nonumber\\
s_{i}\,(x_1,\dots, x_n)&=(x_1,\dots, x_{i-1}, x_{i+1}, x_i, \dots, x_n)\quad (i=1, \dots, n-1)\,,\\ 
s_{n}(x_1, \dots, x_n)&=(x_1, \dots, x_{n-1}, -x_n)\,.\nonumber
\end{align}
We have $\WW\cong W\ltimes \Lambda$, where $\Lambda=\sum_{i=1}^n \Z\epsilon_i$ is the coroot lattice of $R$, acting on $V$ by \eqref{tl}. 
We consider the associated action \eqref{fa}, \eqref{tact} of $\WW$ on $\c(V)$, and form the algebra $\c(V)*\WW\cong {\D}_q * W$ of reflection-difference operators on $V$.   

The affine Hecke algebra $\HH$ associated with $\WW$ is generated by $T_0, \dots, T_n$ subject to the following relations:
\begin{align}
&T_iT_{i+1}T_iT_{i+1}=T_{i+1}T_iT_{i+1}T_i\quad (i=0, i=n-1)\,,\la{b1}\\
&T_iT_{i+1}T_i=T_{i+1}T_iT_{i+1}\quad (i=1, \dots, n-2)\,,\la{b2}\\
&T_iT_{j}=T_jT_{i}\,,\quad |i-j|\ge 2\,,\la{b3}\\
&(T_i-\tau_i)(T_i+\tau_i^{-1})=0\quad(i=0,\dots, n)\,,
\end{align}
where $\tau_i$ are deformation parameters, with $\tau_1=\dots =\tau_{n-1}=\tau$.

The basic representation $\beta: \HH \to {\D}_q*W$ is due to Noumi \cite{No}. To describe it, we choose two additional parameters, $\tau_0^\vee, \tau_n^\vee$. It will be convenient to introduce parameters $\tau_{\alpha}$ and functions $\mathbf{c}_{\alpha}$ for $\alpha\in \Ra$ as follows (cf. \cite[(4.2.2), (4.3.9)]{M03}):
\begin{align}\label{vb1}
\tau_{\alpha} &=\tau\,,\quad \mathbf{c}_{\alpha} =\frac{\tau^{-1}-\tau e^{\alpha}}{1-e^{\alpha}}\,\quad \text{for\ }\alpha=k\delta\pm\epsilon_i\pm\epsilon_j\quad(k\in\Z,\ i\ne j)\,, \\\la{vb2}
\tau_{\alpha} &=\tau_0\,, \quad \mathbf{c}_{\alpha} =\tau_0^{-1}\frac{(1-\tau_0\tau_{0}^\vee e^{\alpha/2})(1+\tau_{0}(\tau_{0}^\vee)^{-1} e^{\alpha/2})}{(1-e^{\alpha})}\,\quad \text{for\ }\alpha=(2k+1)\delta\pm 2\epsilon_i\quad(k\in\Z)\,,\\
\tau_{\alpha} &=\tau_n\,,\quad \mathbf{c}_{\alpha} =\tau_n^{-1}\frac{(1-\tau_n\tau_{n}^\vee e^{\alpha/2})(1+\tau_{n}(\tau_{n}^\vee)^{-1} e^{\alpha/2})}{(1-e^{\alpha})}\,\quad \text{for\ }\alpha=2k\delta\pm 2\epsilon_i\quad(k\in\Z)\,.\la{vb3}
\end{align}
With this notation, we define $\beta$ by setting 
\begin{equation*}
\beta:\, T_i\mapsto \tau_i+\mathbf{c}_{a_i} (s_{i}-1)\,, \quad i=0, \dots, n\,,
\end{equation*}
and extend it to the whole of $\HH$ by multiplicativity, see \cite[Theorem 3.2]{St1}. This defines a subalgebra of ${\D}_q*W$, depending on five parameters $\tau_0, \tau_0^\vee, \tau_n, \tau_n^\vee, \tau$.

By \cite[(3.5)]{St1}, inside $\HH$ we have a commutative subalgebra $\c[Y]$ generated by $Y_i^{\pm 1}$, where 
\begin{equation*}
Y_i=T_i\dots T_{n-1}T_nT_{n-1}\dots T_1T_0T_1^{-1}T_2^{-1}\dots T_{i-1}^{-1}\,,\quad i=1,\dots, n\,.
\end{equation*}
The Hamiltonians of the Koornwinder--van Diejen system can be obtained as in \ref{mrh}, by taking symmetric combinations of $Y_i^{\pm 1}$. For $f(Y)=\sum_{i=1}^n (Y_i+Y_i^{-1})$ this reproduces the Koornwinder operator \cite{Ko}; explicit formulas for the higher Hamiltonians can be found in \cite{vD}.   


\subsection{}\la{clax} Our goal to calculate a quantum Lax matrix correspoding to $Y_1$. 
We will use the following notation for the reflections in $W$: ${\ss}_{ij}$ is an elementary transposition of $x_i$ and $x_j$, ${\ss}_i$ is a sign reversal in the $i$th direction, and 
${\ss}_{ij}^+={\ss}_{ij}{\ss}_i{\ss}_j$, which acts by $x_i\mapsto -x_j$, $x_j\mapsto -x_i$. Comparison with \eqref{wwact} gives $s_{i}={\ss}_{i, i+1}$ for $i=1, \dots, n-1$, $s_{n}={\ss}_n$, $s_{0}={\ss}_1 t(\epsilon_1)$.

We have $Y_1=Y^{\epsilon_1}$, so following \ref{3.6} we consider the stabiliser $W'$ of $\lambda=\epsilon_1$. This is the subgroup of signed permutations of $x_2, \dots, x_n$, and the Lax matrix will be of size $|W|/ |W'|=2n$. To calculate it, we need to determine the action of $Y_1$ on $M'=e'M$. 
By \cite[(4.2)]{St1},
\begin{equation}\la{y1}
Y_1=\RR (\epsilon_1-\epsilon_2)\RR (\epsilon_1-\epsilon_3)\dots \RR (\epsilon_1-\epsilon_n)\RR (2\epsilon_1)\RR (\epsilon_1+\epsilon_n)\dots \RR (\epsilon_1+\epsilon_2)\RR (\delta+2\epsilon_1)t(\epsilon_1)\,.
\end{equation}
Here $\RR(\alpha)$ are defined by \eqref{rdef} together with 
\eqref{vb1}--\eqref{vb3}. 

Let us introduce some shorthand notation. In addition to $a_{ij}$, $b_{ij}$, $\RR_{ij}$ \eqref{rij}, \eqref{ab}, we define
$a_{ij}^{+}=a({x_i+x_j})$, $b_{ij}^+=b({x_i+x_j})$, $a_{ij}^{-}=a({-x_i-x_j})$, $b_{ij}^-=b({-x_i-x_j})$, $\RR_{ij}^+=a_{ij}^++b_{ij}^+{\ss}_{ij}^+$.
Let us also introduce functions $u, v, \widetilde u, \widetilde v$ as follows:
\begin{align}
u(z)&=\tau_n^{-1}\frac{(1-\tau_n\tau_n^\vee e^z)(1+\tau_n(\tau_n^\vee)^{-1}e^z)}{1-e^{2z}}\,,& v(z)&=\tau_n-u(z)\,,\la{uv1}\\
\widetilde u(z)&=\tau_0^{-1}\frac{(1-\tau_0\tau_0^\vee q^{1/2}e^z)(1+\tau_0(\tau_0^\vee)^{-1}q^{1/2}e^z)}{1-qe^{2z}}\,,&\widetilde v(z)&=\tau_0-\widetilde u(z)\,.\la{uv2}
\end{align}
Here $q=e^c$. Below we will use $u_i, u_i^-$ to denote $u_i=u({x_i})$ and $u({-x_i})$, respectively, and similarly for $v, \widetilde u, \widetilde v$. 
With this notation, we have:
\begin{gather*}
Y_1=\RR_{12}\RR_{13}\dots \RR_{1n}\,\RR(2\epsilon_1) \RR_{1n}^+\dots \RR_{12}^+\, \RR(\delta+2\epsilon_1) \,t(\epsilon_1)\,,
\\
\RR(2\epsilon_1)=u_1+v_1{\ss}_1\,, \qquad \RR(\delta+2\epsilon_1)\,t(\epsilon_1)=\widetilde u_1\,t(\epsilon_1)+\widetilde v_1{\ss}_1\,.
\end{gather*}


\medskip
\begin{lemma} Let 
$e''=\frac{1}{n!}\sum_{{w} \in \mathfrak{S}_{n-1}} {w}$\,,
where $\mathfrak{S}_{n-1}$ is the group of permutations on $\{2, \dots, n\}$, viewed as a subgroup of $W=\mathfrak{S}_n\ltimes \Z_2^n$. 
The operators 
$\mathcal R=\RR_{12}\dots \RR_{1n}$ and $\mathcal R^+=\RR_{1n}^+\dots \RR_{12}^+$ preserve the subspace $M''=e''M$. Their restriction onto $M''$ is calculated as follows: 
\begin{gather*}
\mathcal R |_{M''}=U+\sum_{i\ne 1}^n V_i{\ss}_{1i}\,,\qquad \mathcal R^+ |_{M''}=U^++\sum_{i\ne 1}^n V_i^+{\ss}_{1i}^+\,, \quad\text{where}
\\
U=\prod_{l\ne 1}^n a_{1l}\,,\quad V_i=b_{1i}\prod_{l\ne 1, i} a_{il}\,,\quad U^+=\prod_{l\ne 1}^n a_{1l}^+\,,\quad V_i^+=b_{1i}^+\prod_{l\ne 1, i} a_{li}\,. 
\end{gather*}
\end{lemma}

\begin{proof} 
For $\mathcal R$ this has been shown in Lemma \ref{ns}. The statement for $\mathcal R^+$ follows from the fact that $\mathcal R^+=\mathcal R^{\omega}$, where $\omega\in W'$ is given by
\begin{equation}\la{omeg}
\omega: V\to V\,,\qquad (x_1, x_2, \dots, x_n)\mapsto (x_1, -x_n, \dots, -x_2)\,. 
\end{equation}

\end{proof}

Now let us restrict $Y_1$ further onto $M'=e'M\subset e''M$. From the above, $Y_1$ can be replaced with 
\begin{equation}\la{split} 
(U+\sum_{i\ne 1}^n V_i{\ss}_{1i})R(2\epsilon_1)
(U^++\sum_{i\ne 1}^n V_i^+{\ss}_{1i}^+)
R(\delta+2\epsilon_1)
\end{equation}
Let us first work out the action of the product of the first three factors, 
\begin{equation}\la{split1} 
(U+\sum_{i\ne 1}^n V_i{\ss}_{1i})(u_1+v_1{\ss}_1)(U^++\sum_{i\ne 1}^n V_i^+{\ss}_{1i}^+)\,.
\end{equation}

\medskip
\begin{lemma}\la{abcd}
The operator \eqref{split1} 
preserves the subspace $M'=e'M$, and its restriction onto $M'$ is given by $A+B{\ss}_1+\sum_{i\ne 1} (C_i{\ss}_{1i}+D_i{\ss}_{1i}^+)$
with 
 \begin{gather*}
A=u_1\prod_{l\ne 1}^n a_{1l}a_{1l}^+\,,\quad 
B=\tau^{2n-2}\tau_n-A-\sum_{i\ne 1} (C_i+D_i)\,, 
\\
C_i=u_ib_{1i}a_{1i}^+\prod_{l\ne 1, i} a_{il}a_{il}^+\,,\quad D_i=u_i^-b_{1i}^+a_{1i}\prod_{l\ne 1, i} a_{li}^-a_{li}\,. 
\end{gather*}
\end{lemma}

\begin{proof} The operator \eqref{split1} is composed of $Y_1$ and (the inverse of) $\RR (\delta+2\epsilon_1)t(\epsilon_1)$. Both of these preserve $M'$, and so does \eqref{split1}. Next, for any element in $W$, its action on $M'$ can be replaced with one of the elements $w=\id, {\ss}_1, {\ss}_{1i}, {\ss}_{1i}^+$, since these elements represent the cosets in $W/W'$. Thus, restricting \eqref{split1} to $M'$ we obtain an expression of the form  $A+B{\ss}_1+\sum_{i\ne 1} (C_i{\ss}_{1i}+D_i{\ss}_{1i}^+)$. And since the operator \eqref{split1} preserves $M'$, the resulting expression will be $W'$-invariant, in particular, $D_i=(C_i)^{{\ss}_i}$. 

Let us now expand the product \eqref{split1}, moving the group elements to the right. We can pick either $w=\id$ or ${\ss}_{1i}$ from the first factor, $w=\id$ or ${\ss}_1$ from the second, and $w=\id$ or ${\ss}_{1j}^+$ from the third. It is easy to check that the product of three elements picked this way will represent the trivial coset $\id\cdot W'$ if and only if $w=\id$ is chosen from each factor. Therefore,  $A=Uu_1U^+$. Also, the only way to obtain the representative ${\ss}_{1i}$ is by picking ${\ss}_{1i}$, $\id$, and $\id$, respectively, from each of the factors. Therefore, $C_i{\ss}_{1i}=V_i{\ss}_{1i} u_1U^+$. This gives the required expressions for $A$ and $C_i$,  after which we find $D_i$ as $D_i=(C_i)^{{\ss}_i}$.  

To determine $B$, we use that $\RR_{1i}e=\RR_{1i}^+e=\tau e$, $\RR(2\epsilon_1)e=\tau_n e$. As a result,
\begin{equation*}
\RR_{12}\RR_{13}\dots \RR_{1n}\,\RR(2\epsilon_1) \RR_{1n}^+\dots \RR_{12}^+e=(\tau^{2n-2}\tau_n)e\,,
\end{equation*}
from which $A+B+\sum_{i\ne 1}(C_i+D_i)=\tau^{2n-2}\tau_n$. Lemma is proven.
\end{proof}

\subsection{}\la{ccn}
Now we translate the obtained results into a matrix form. We choose $2n$ elements representing the (right and left) $W'$-cosets in $W$: $r_i={\ss}_{1i}$, $r_{n+i}={\ss}_{1i}^+$, where $1\le i\le n$ and $r_1={\ss}_{11}:=\id$, $r_{n+1}={\ss}_{11}^+:={\ss}_1$. Suppose we have $Z=\sum_{i=1}^{2n} Z_ir_i$, with $Z_i\in\D_q(V)$. Assume that $Z$ is $W'$-invariant, $wZ=Zw$ for all $w\in W'$, hence the multiplication by $Z$ preserves $M'=e'M$. We write elements of $M'$ as $\sum_{j=1}^{2n} e'r_j\,f_j$, with $f_j\in\c(x)$. Then we have
\begin{equation}\la{zact}
Z\,(e'\sum_{j=1}^{2n} r_j\,f_j)=e' Z\sum_{j=1}^{2n} r_j\,f_j=\sum_{i,j=1}^{2n} e' r_ir_j(Z_i)^{r_jr_i}\,f_j\,.
\end{equation}
For each individual term in this sum we have: $e' r_ir_j(Z_i)^{r_jr_i}\,f_j=e'r_k(Z_i)^{r_jr_i}f_j$, where $k$ is a uniquely defined index $k=k(i,j)$ such that $e'r_ir_j=e'r_k$. This means that this term represents the $(k, j)$-th entry of the matrix of $Z$. For the reader's convenience, here is an explicit description:
\begin{equation*}\la{kij}
e'r_ir_j=e'r_k\,,\qquad k=k(i,j)=\begin{cases}
i&\text{for}\ i\notin \{1, n+1, j, j\pm n\}\\
j&\text{for}\ i=1\\
j+n&\text{for}\ i=n+1\,, 1\le j\le n\\
j-n&\text{for}\ i=n+1\,, n+1\le j\le 2n\\
1&\text{for}\ i=j\\
n+1&\text{for}\ i=j\pm n\,.
\end{cases}
\end{equation*}
Using this, we easily find all the entries of the matrix of $Z$. For example, for $1<i\ne j\le n$ we have $e'r_ir_j=e'r_i$, and so the $(i,j)$-th matrix entry is $(Z_i)^{r_jr_i}$.    

Applying this procedure successively to $Z=A+B{\ss}_1+\sum_{i\ne 1} (C_i{\ss}_{1i}+D_i{\ss}_{1i}^+)$ and to $Z=R(\delta+2\epsilon_1)=\widetilde u_1\,t(\epsilon_1)+\widetilde v_1{\ss}_1$, we calculate the corresponding  matrices. Let us state the result, leaving its (routine) verification to the reader.

\begin{prop}\la{lmct} 
Denote by $\mathcal P, \mathcal Q$ the $2n\times 2n$ matrices representing the action on $M'$ of $A+B{\ss}_1+\sum_{i\ne 1} (C_i{\ss}_{1i}+D_i{\ss}_{1i}^+)$ and $\widetilde u_1\,t(\epsilon_1)+\widetilde v_1{\ss}_1$, respectively. Let us extend the set of vectors $\epsilon_i$ and variables $x_i$ to the range $1\le i \le 2n$ by setting $\epsilon_{i+n}=-\epsilon_i$ and $x_{n+i}=-x_i$ for $1\le i\le n$. We also extend the definitions of $a_{ij}, a^+_{ij}$, etc., accordingly. For instance, $a_{n+i,j}^+=a(-x_i+x_j)$, $b_{i, n+j}=b(x_i+x_j)$, $u_{n+j}=u(-x_j)$ if $1\le i,j\le n$. With this notation, we have:
\begin{align*}
\mathcal P_{ij}&=u_jb_{ij}a_{ij}^+\prod_{l=1}^{2n}{\vphantom{\prod}}'
a_{jl}\qquad(i-j\ne 0, \pm n)\,,\qquad
\mathcal P_{ii}=u_i\prod_{l=1}^{2n}{\vphantom{\prod}}' a_{il}\,,\\
\mathcal P_{ij}&=\tau^{2n-2}\tau_n-\sum_{l\ne j}^{2n}\mathcal P_{il}\qquad (i-j=\pm n)\,,\\
\mathcal Q_{ij}&=\widetilde u_j t(\epsilon_j)\qquad (i=j)\,,\qquad \mathcal Q_{ij}=\widetilde v_i \qquad (i-j=\pm n)\,,\qquad
\mathcal Q_{ij}=0\qquad (i-j\ne 0, \pm n)\,.
\end{align*}
Here the symbol $\prod{\vphantom{\prod}}'$ in the formula for $\mathcal P_{ij}$ indicates that we exclude those values of $l$ where either $l-i$ or $l-j$ equals $0, \pm n$ (e.g., two values are excluded if $i=j$).
\end{prop} 

\medskip

\begin{cor}
The quantum Lax matrix $\mathcal L$ for the Koornwinder--van Diejen system is $\mathcal L=\mathcal P\mathcal Q$, with $\mathcal P, \mathcal Q$ given above. The classical Lax matrix is found as $L=PQ$ where $P=\mathcal P$, while $Q$ is obtained by setting $q=1$ in the definitions of $\widetilde u, \widetilde v$ and by replacing $t(\epsilon_j)$ with $e^{\beta p_j}$ (with $p_{n+i}=-p_i$ for $1\le i\le n$).  
\end{cor}

 \subsection{}
Finally, let us write down explicitly the analogue of Proposition \ref{intq} in the $C^\vee C_n$ case. We have 
\begin{equation*}
R_+=\{2\epsilon_i\,|\,1\le i\le n\}\cup \{\epsilon_i\pm \epsilon_j\,|\, 1\le i<j\le n\}\,,\qquad R_+\setminus R'_+=\{2\epsilon_1\}\cup \{\epsilon_1\pm \epsilon_i\,|\, 2\le i\le n\}\,.
\end{equation*}
The functions $\mathbf{c}_\alpha$ are defined in \eqref{vb1}--\eqref{vb3}. Substituting them into \eqref{uva} gives 
\begin{equation*}
\phi= u_1^-\prod_{l\ne 1}^n a_{l1}a_{l1}^-\,,\quad\uu=(\phi_1, \dots, \phi_{2n}) \,,\quad \vv=(1, \dots, 1)^T\,,
\end{equation*}
where
\begin{equation*}
\phi_i=\phi^{{\ss}_{1i}}=u_i^-\prod_{l\ne i}^n a_{li}a_{li}^-\,,\qquad \phi_{n+i}=\phi^{{\ss}_{1i}^+}=u_i\prod_{l\ne i}^n a_{li}^+a_{il}\qquad(1\le i\le n)\,.
\end{equation*}

\begin{prop} The difference operators $\uu\,\mathcal L^k\, \vv$, $k\in\Z$ with the above $\uu$, $\vv$ belong to the commutative algebra of the quantum integrals ${L}_f$, $f\in\c[Y_1^{\pm 1}, \dots, Y_n^{\pm 1}]^{\mathfrak{S}_n}$ of the Koornwinder--van Diejen system. 
\end{prop}
The proof is identical to that of Proposition \ref{intq}. By passing to the classical limit, the same result holds for the classical system. Of course, in the classical case one can also produce first integrals as $h_k=\tr (L^k)$.  

\section{Lax pairs for the elliptic Calogero--Moser systems}\la{ecase}

In this section we explain how our approach extends to the elliptic Calogero--Moser systems associated to any root system, including the $BC_n$ case (Inozemtsev system). This will lead to Lax matrices with a spectral parameter, reproducing, in particular, the classical Lax pairs \cite{Kr2, BMS, DHP}. For the Inozemtsev system we obtain a Lax pair of size $2n$, different from \cite{I, DHP}. We will employ the theory of elliptic Dunkl operators, developed in \cite{BFV, C4, EM1, EFMV}.
Note that in the elliptic case no quantum Lax pairs were known previously, although a quantum Lax operator in type $A_n$ was considered in \cite{Ha}. 

\subsection{} In the setting of Section \ref{2.1}, let $W$ be a Weyl group 
with a root system $R\subset V=\c^n$ and a $W$-invariant set of parameters $c_\alpha$, $\alpha\in R$. The Dunkl operators in the elliptic case depend on $t\ne 0$, the elliptic modulus $\tau$, and additional {\it dynamical variables} represented by a vector $\lambda\in V$.  They are as follows \cite{BFV, EM1, EFMV}:
\begin{equation}
\label{edu} y_\xi :=
t\partial_\xi+\sum_{\alpha\in R_+}
c_\alpha \langle\alpha, \xi\rangle \sigma_{\langle\alpha^\vee, \lambda\rangle}(\langle\alpha, x\rangle)s_\alpha\ , \quad \xi \in V\ .
\end{equation}
Here $\alpha^\vee=2\alpha/\langle\alpha,\alpha\rangle$, and 
\begin{equation}\la{sigm}
\sigma_\mu(z) =\frac{\theta_1(z-\mu)\theta_1'(0)}{\theta_1(z)\theta_1(-\mu)}\,, 
\end{equation}
where $\theta_1(z)=\theta_1(z |\tau)$ is the odd Jacobi theta function, associated with the elliptic curve $\c/\Z+\Z\tau$. Sometimes we will write $y_\xi(\lambda)$ to emphasize dependence on the dynamical variables. Note that as a function of $\lambda$, $y_\xi$ has poles along the hyperplanes $\dpr{\alpha^\vee, \lambda}=m+n\tau$ with $m,n\in\Z$. Below we will also need a classical limit of $y_\xi$, which
is the following element of $\c(V)[p_1, \dots, p_n]*W$,  cf. \ref{cld}:
 \begin{equation}
\label{educ} y_\xi^{c} :=
p_\xi+\sum_{\alpha\in R_+}
c_\alpha \langle\alpha, \xi\rangle \sigma_{\langle\alpha^\vee, \lambda\rangle}(\langle\alpha, x\rangle)s_\alpha\ , \quad \xi \in V\ .
\end{equation}

Again, the two main properties of the Dunkl operators are their commutativity and equivariance \cite{BFV}: for all $\,\xi, \eta \in V\,$ and $ w \in W $,
\begin{equation}\la{eduprop}
 y_{\xi}\,y_{\eta} = y_{\eta}\,y_{\xi}\,,\qquad \,w\,y_\xi(\lambda) =
y_{w\xi}(w\lambda)\,w\,.
\end{equation}
Note that in the second relation the group action changes both $\xi$ and $\lambda$. As before, the assignment $\,\xi \mapsto y_\xi\,$
extends to an algebra map
\begin{equation*}
SV=\bigoplus_{i\ge 0} S^iV \to \D(V)*W \ ,\quad q\mapsto q(y)\,.
\end{equation*}
However, unlike in the rational case, this map is not $W$-equivariant. Despite that, there is a method for constructing commuting $W$-invariant quantum Hamiltonians from $y_\xi$, but it requires a certain regularization procedure \cite{EFMV}. For the quadratic Hamiltonian this is straightforward (this goes back to \cite[Sec. 6]{BFV}). Namely, let $\partial_i=\partial_{\xi_i}$ and $y_i=y_{\xi_i}$, where $\{\xi_i\,|\, i=1 \dots n\}$ is an orthonormal basis in $V$. Let $\dpr{y, y}=y_{1}^2+\dots +y_{n}^2$. Then, by direct calculation,    
\begin{multline*}
\dpr{y, y}=t^2\Delta_V+t\sum_{\alpha\in R_+}c_\alpha\langle\alpha, \alpha\rangle\sigma'_{\langle\alpha^\vee, \lambda\rangle}(\langle\alpha, x\rangle)s_\alpha+
\sum_{\alpha\in R_+} c^2_\alpha\langle\alpha, \alpha\rangle\sigma_{\langle\alpha^\vee, \lambda\rangle}(\langle\alpha, x\rangle)\sigma_{\langle\alpha^\vee, \lambda\rangle}(-\langle\alpha, x\rangle)\,.
\end{multline*} 
Here $\sigma'_\mu(z)=\frac{d}{dz}\sigma_\mu(z)$. Using the identity $\sigma_\mu(z)\sigma_\mu(-z)=\wp(\mu)-\wp(z)$, we rewrite this as 
\begin{gather}
\frac12 \dpr{y, y}-\frac12\sum_{\alpha\in R_+} c^2_\alpha\langle\alpha, \alpha\rangle \wp(\langle\alpha^\vee, \lambda\rangle)=
\widehat H+\widehat A \,,\qquad \label{twos}
\intertext{where}
\widehat H=\frac12 t^2\Delta_V-\frac12 \sum_{\alpha\in R_+} c_\alpha(c_\alpha+t)\langle\alpha, \alpha\rangle \wp(\langle\alpha, x\rangle)\,,\la{elcm}\\
\widehat A=\frac{t}{2}\sum_{\alpha\in R_+}c_\alpha\langle\alpha, \alpha\rangle\left(\wp(\dpr{\alpha, x})+\sigma'_{\langle\alpha^\vee, \lambda\rangle}(\langle\alpha, x\rangle)s_\alpha\right)\,.\la{elcm1}
\end{gather} 
This is an elliptic analog of \eqref{acal}. Note that $\lim_{\mu\to 0}\sigma'_\mu(z)=-\wp(z)-2\eta_1$, where $\eta_1=\zeta(\frac{1}{2})$, see \cite[Prop. 19(v)]{BFV}. Thus, when $\lambda$ approaches the hyperplane $\dpr{\alpha^\vee, \lambda}=0$, the term 
\begin{equation*}
c_\alpha\langle\alpha, \alpha\rangle\left(\wp(\dpr{\alpha, x})+\sigma'_{\langle\alpha^\vee, \lambda\rangle}(\langle\alpha, x\rangle)s_\alpha\right)
\end{equation*}
in $\widehat A$ can be replaced by its limit, i.e., by $c_\alpha\langle\alpha, \alpha\rangle\left(\wp(\dpr{\alpha, x})(1-s_\alpha)-2\eta_1s_\alpha\right)$. This shows that the operator $\dpr{y, y}$ after subtracting a $\lambda$-dependent term becomes regular in the neighbourhood of $\lambda=0$.

Now, we obvioulsy have $[y_\xi, \widehat H+\widehat A]=0$, therefore we can construct a quantum Lax pair $\mathcal L, \mathcal A$ of size $|W|$ in exactly the same way as in \ref{qdl}, but now the matrix $\mathcal A$ will depend on the dynamical variables. Since $\widehat A$ vanishes in the classical limit $t=-\mathrm{i}\hbar\to 0$, the constructed Lax pair admits a classical limit. 

To construct Lax pairs of smaller sizes, we use the same approach as in \ref{small}. We start by choosing $\xi$ with nontrivial stabiliser $W'$. Denote by $R'$ the root system of $W'$. Then $\dpr{\alpha, \xi}=0$ for $\alpha\in R'$, so  
\begin{equation}\la{rdu}    
y_\xi=t\partial_\xi+\sum_{\alpha\in R_+\setminus R'_+}
c_\alpha \langle\alpha, \xi\rangle \sigma(\langle\alpha, x\rangle; \langle\alpha^\vee, \lambda\rangle)s_\alpha\,.
\end{equation}  
Choose $\lambda$ (close to $\lambda=0$) with the same stabiliser as $\xi$: this is possible because the singular terms with $\alpha\in R'$ are no longer present in \eqref{rdu}. We also specialize $\lambda$ in \eqref{twos} (which is possible because the right-hand side is regular near $\lambda=0$). It easily follows from \eqref{eduprop} that under such a specialisation $y_\xi$ and both sides of \eqref{twos} become $W'$-invariant, that is, $wy_\xi=y_\xi w$ and $w(\widehat H+\widehat A)=(\widehat H+\widehat A)w$ for all $w\in W'$. Since $\widehat H$ is $W$-invariant, we obtain that $w\widehat A=\widehat A w$. As a result, both $y_\xi$ and $\widehat A$ preserve the subspace $M'=e'M$. Therefore, the same construction as in \ref{small} applies, producing a Lax pair of size $|W|/|W'|$.  

\subsection{}
Let us illustrate the method in the case $W=\mathfrak{S}_n$. This is a modification of \ref{typea}. We have $n$ commuting Dunkl operators, depending on $c,t$ and the dynamical parameters $\lambda=(\lambda_1, \dots, \lambda_n)$:
\begin{equation*}
y_i=t\partial_i+c\sum_{j\ne i}\sigma_{\lambda_i-\lambda_j}(x_i-x_j)\ss_{ij}\,,\qquad i=1\dots n\,.
\end{equation*}
Choose $y_1$, so $\xi=(1,0,\dots, 0)$ so $W'$ and $e'$ are the same as in \ref{typea}, and 
\begin{equation*}
R_+'=\{\epsilon_i-\epsilon_j\,|\, 2\le i<j\le n\}\,.
\end{equation*} 
Let us specialise $\lambda$ in $y_1$ to $(\mu, 0,\dots, 0)$, where $\mu$ is an arbitrary parameter:  
\begin{equation*}
y_1=t\partial_1+c\sum_{j\ne 1}^n\sigma_{\mu}(x_1-x_j)\ss_{1j}\,.
\end{equation*}
We also specialise $\lambda$ in $\widehat A$, obtaining
\begin{equation*}
\widehat A=ct\sum_{j\ne 1}^n \left(\wp(x_1-x_j)+\sigma'_{\mu}(x_1-x_j)\ss_{1j}\right)+ ct\sum_{2\le i<j\le n}\left(\wp(x_i-x_j)(1-\ss_{ij})-2\eta_1\ss_{ij}\right)\,.
\end{equation*}
Since $\ss_{ij}=\id$ on $M'$, the restriction of $\widehat A$ onto $M'$ can be replaced by 
\begin{equation*}
\widehat A=ct\sum_{j\ne 1}^n \left(\wp(x_1-x_j)+\sigma'_{\mu}(x_1-x_j)\ss_{1j}\right)- ct\eta_1 (n-1)(n-2)\,.
\end{equation*}
By removing an unimportant constant term, we may change $\widehat A$ to
 \begin{equation*}
\widehat A=ct\sum_{j\ne 1}^n \left(\wp(x_1-x_j)+\sigma'_{\mu}(x_1-x_j)\ss_{1j}\right)\,.
\end{equation*}  
The quantum Lax pair $\mathcal L, \mathcal A$ is now calculated in exactly the same way as in \ref{typea}. This gives, after setting $t=-\mathrm{i}\hbar$, $c=\mathrm{i}g$, the following matrices:
\begin{equation}\la{ael}
\mathcal L_{kl}=
\begin{cases}
\mathrm{i}g\sigma_{\mu}(x_k-x_l)\quad&\text{for}\ k\ne l\\
\hat p_k\quad&\text{for}\ k=l\,,
\end{cases}\qquad
\mathcal A_{kl}=
\begin{cases}
g\hbar \,\sigma'_{\mu}(x_k-x_l)\quad&\text{for}\ k\ne l\\
g\hbar \,\sum_{j\ne k}^n \wp(x_j-x_k)\quad&\text{for}\ k=l\,.
\end{cases}
\end{equation}
This is an elliptic generalisation of \eqref{qlp}--\eqref{qleq} with
\begin{equation*}
\widehat H=\frac12\sum_{k=1}^n\hat p_k^2+g(g-\hbar)\sum_{k<l}^n\wp(x_k-x_l)\,.
\end{equation*} 
In the classical limit $\hbar\to 0$ it gives the well-known Krichever's Lax pair with a spectral parameter \cite{Kr2}. 

\subsection{}\la{inoz} 
Let us describe the $BC_n$-case related to the Inozemtsev system \cite{I}. 
This system depends on five coupling constants $c, g_0, g_1, g_2, g_3$. We have $n$ commuting Dunkl operators \cite{EFMV}:
\begin{equation*}
y_i=t\partial_i+v_{\lambda_i}(x_i)\ss_i+c\sum_{j\ne i}\left(\sigma_{\lambda_i-\lambda_j}(x_i-x_j)\ss_{ij}+\sigma_{\lambda_i+\lambda_j}(x_i+x_j)\ss_{ij}^+\right)\,,\qquad i=1\dots n\,.
\end{equation*}
Here $\lambda=(\lambda_1, \dots, \lambda_n)$ are the dynamical parameters, and 
\begin{equation}\la{vmu}
v_\mu (z)=v_\mu(z; g_0, g_1, g_2, g_3)=\sum_{r=0}^3 g_r\sigma^r_{2\mu}(z)\,,\qquad \sigma^r_\mu(z):=\frac{\theta_{r+1}(z-\mu)\theta_1'(0)}{\theta_{r+1}(z)\theta_1(-\mu)}\,,
\end{equation}
where $\theta_r(z)=\theta_r(z|\tau)$, $r=0\dots 3$ are the Jacobi theta functions, with $\theta_4(z):=\theta_0(z)$. Clearly, $\sigma^0_\mu(z)$ coincides with \eqref{sigm}. 
For an account of the properties of $\sigma^r_\mu(z)$, see \cite[Appendix]{KH97}.
Note the following identities:
\begin{equation}\la{vv}
v_{-\mu}(-z)=-v_\mu(z)\,,\qquad v_\mu(z)v_\mu(-z)=\sum_{r=0}^3\left( (g^\vee_r)^2\wp(\mu+\omega_r)-g_r^2\wp(z+\omega_r)\right)\,,
\end{equation}
where $(\omega_0, \omega_1, \omega_2, \omega_3)=(0, \frac12, \frac{1+\tau}{2}, \frac{\tau}{2})$ are half-periods, and
\begin{equation}\la{ghat} 
\left( {
\begin{array}{c}
g^\vee_0\\
g^\vee_1\\ 
g^\vee_2\\
g^\vee_3\\
\end{array}
} \right)
= 
\frac12 \left( {
\begin{array}{rrrr}
1 & 1 & 1& 1\\
1 & 1 & -1 & -1\\
1 & -1 & 1 & -1\\
1 & -1 & -1 & 1\\
\end{array}
} \right)
\left( {
\begin{array}{c}
g_0\\
g_1\\
g_2\\
g_3\\
\end{array}
} \right)\,.
\end{equation}
Another property of $v_\mu(z)$ which will be needed later is the following symmetry between $\mu$ and $z$:
\begin{equation}\la{vsym}
v_\mu(z; g_0, g_1, g_2, g_3)=v_{-z}(-\mu;  g^\vee_0, g^\vee_1, g^\vee_2, g^\vee_3)=-v_z(\mu;  g^\vee_0, g^\vee_1, g^\vee_2, g^\vee_3)\,.
\end{equation}
This can be checked by comparing translation properties and residues in the $z$-variable. 

In the formulas below we will use the abbreviations $x_{ij}:=x_i-x_j$, $x_{ij}^{+}:=x_i+x_j$, and similarly for the $\lambda$-variables. One calculates $\dpr{y, y}=y_1^2+\dots + y_n^2$ to get:
\begin{multline*}
\dpr{y, y}=t^2\Delta+2c^2\sum_{i<j}^n\left(\sigma_{\lambda_{ij}}(x_{ij})\sigma_{\lambda_{ij}}(-x_{ij})+\sigma_{\lambda_{ij}^+}(x_{ij}^+)\sigma_{\lambda_{ij}^+}(-x_{ij}^+)\right)\\
+
2ct\sum_{i<j}^n \left(\sigma_{\lambda_{ij}}(x_{ij})\ss_{ij}+\sigma_{\lambda_{ij}^+}(x_{ij}^+)\ss_{ij}^+\right)
+
\sum_{i=1}^n v_{\lambda_i}(x_i)v_{\lambda_i}(-x_i)+t\sum_{i=1}^n v'_{\lambda_i}(x_i)\ss_i\,.
\end{multline*} 
We can rewrite this as
\begin{gather}\la{bcreg}
\dpr{y, y}-
2c^2\sum_{i<j}^n (\wp(\lambda_{ij})+\wp(\lambda_{ij}^+))-\sum_{i=1}^n\sum_{r=0}^3 (g^\vee_r)^2\wp(\lambda_i+\omega_r)
=
\widehat H+\widehat A \,,
\intertext{where}
\nonumber
\widehat H=t^2\Delta_V-2c(c+t)\sum_{i<j}^n (\wp(x_{ij})+\wp(x_{ij}^+))-\sum_{i=1}^n\sum_{r=0}^3 g_r(g_r+t)\wp(x_i+\omega_r)\,,\\
\nonumber\widehat A=
2ct\sum_{i<j}^n \left(\wp(x_{ij})+\wp(x_{ij}^+)+\sigma'_{\lambda_{ij}}(x_{ij})\ss_{ij} + \sigma'_{\lambda_{ij}^+}(x_{ij}^+)\ss_{ij}^+\right)
+t\sum_{i=1}^n\sum_{r=0}^3 g_r\wp(x_i+\omega_r)+t\sum_{i=1}^n v'_{\lambda_i}(x_i)\ss_i\,.\la{bcreg1}
\end{gather} 

We choose $\xi=(1,0,\dots, 0)$ so $W'$ and $e'$ are the same as in \ref{ccn}, and the elements of $M'=e'M$ have the form $\sum_{i=1}^{2n} e'r_i f_i$, where $r_i=\ss_{1i}$, $r_{n+i}=\ss_{1i}^+$ ($1\le i\le n$). The Lax matrix, therefore, will be of size $2n$. We will construct it from the operator $y_1$, in which we set $\lambda=(\mu, 0,\dots, 0)$:  
\begin{equation*}
y_1=t\partial_1+v_\mu(x_1)\ss_1+c\sum_{j\ne 1}^n(\sigma_{\mu}(x_{1j})\ss_{1j}+\sigma_{\mu}(x_{1j}^+)\ss_{1j}^+)\,.
\end{equation*}
Similarly, we specialise $\lambda$ in $\widehat A$ and obtain (after removing a constant):
\begin{equation*}
\widehat A|_{M'}=2ct\sum_{j\ne 1}^n \left(\wp(x_{1j})+\wp(x_{1j}^+)+\sigma'_{\mu}(x_{1j})\ss_{1j}+\sigma'_\mu(x_{1j}^+)\ss_{1j}^+\right)+t\sum_{r=0}^3g_r\wp(x_1+\omega_r)+tv'_\mu(x_1)\ss_1\,.
\end{equation*}  
The quantum Lax pair $\mathcal L, \mathcal A$ is now calculated in exactly the same way as in \ref{ccn}. To write down the answer in compact form, let us extend the range of the variables $x_i$ to $1\le i \le {2n}$ by setting $x_{n+i}=-x_i$; we also set $\partial_{n+i}=-\partial_i$. Then we obtain the following $2n\times 2n$ matrices: 
\begin{align*}
\mathcal L_{ij}&=
\begin{cases}
c\sigma_{\mu}(x_i-x_j)\quad&\text{for}\ i-j\ne 0, \pm n\\
v_\mu(x_i)\quad&\text{for}\ i-j=\pm n\\
t\partial_i\quad&\text{for}\ i=j\,,
\end{cases}\\
\mathcal A_{ij}&=
\begin{cases}
2ct\sigma'_{\mu}(x_i-x_j)\quad&\text{for}\ i-j\ne 0, \pm n\\
tv'_\mu(x_i)\quad&\text{for}\ i-j=\pm n\\ \displaystyle{
2ct\sum_{l:\,l-i\ne 0, \pm n}^{2n} \wp(x_i-x_l)+t\sum_{r=0}^3 g_r\wp(x_i+\omega_r)}\quad&\text{for}\ i=j\,.
\end{cases}
\end{align*}
The classical Lax pair is obtained in the limit $t=-\mathrm{i}\hbar\to 0$, resulting in
\begin{equation}\la{inolax}
L_{ij}=
\begin{cases}
c\sigma_{\mu}(x_i-x_j)\quad&\text{for}\ i-j\ne 0, \pm n\\
v_\mu(x_i)\quad&\text{for}\ i-j=\pm n\\
p_i\quad&\text{for}\ i=j\,,
\end{cases}
\qquad\qquad A=-t^{-1}\mathcal A\,.
\end{equation} 
Here we keep the same convention, $p_{n+i}=-p_i$. 
Note that the previously known Lax pairs for the classical Inozemtsev system were of a larger size ($3n$ as in \cite{I}, or $2n+1$ as in \cite{DHP}). Probably, they can be brought to the above form by a suitable reduction.

\subsection{}  According to \cite{C4, EFMV}, the elliptic Calogero--Moser Hamiltonian \eqref{elcm} is completely integrable: there is a commutative algebra of quantum Hamiltonians ${L}_q\in\D(V)$, $q\in (SV)^W$, each with the leading term $q(\partial)$. These Hamiltonians are $W$-invariant partial differential operators with {\it elliptic coefficients}, i.e. they are invariant under translations $t(v)$, $v\in P^\vee\oplus \tau P^\vee$. In \cite{EFMV}, a general procedure was given for constructing these higher Hamiltonians from the elliptic Dunkl operators. It consists of the following three steps, see \cite[Theorem 3.1]{EFMV}. First, one substitutes the elliptic Dunkl operators $y_\xi(\lambda)$ as momenta into suitable {\it classical} rational Calogero--Moser Hamiltonians, with the dynamical parameters $\lambda$ as the position variables. Next, one goes to the limit $\lambda=0$ (it is shown in \cite{EFMV} that this limit exists), obtaining a reflection-difference operator. Finally, one restricts this reflection-differential operator to $W$-invariant functions (so it becomes a differential operator). There is also a parallel construction in the classical case \cite[Theorem 3.4]{EFMV}.

For our purposes, we need a modification of that procedure, where the substitution is made into the classical {\it elliptic} Calogero--Moser Hamiltonians (cf. Remark 3.8 in \cite{EFMV}). Namely, consider the following classical Hamiltonian $H^{\vee, c}(\xi, \lambda)$:
\begin{equation}\la{dh}
H^{\vee, c}=\frac12 \dpr{\xi, \xi}-\frac12\sum_{\alpha\in R_+} c^2_\alpha\langle\alpha, \alpha\rangle \wp(\langle\alpha^\vee, \lambda\rangle)\,.
\end{equation}
By the above results of \cite{EFMV}, there is a family of Poisson-commuting Hamiltonians ${L}^{\vee, c}_q(\xi, \lambda)$, $q\in (SV)^W$, with the above $H^{\vee, cl}$ corresponding to $q=\frac12\dpr{\xi, \xi}$. Note that each of these Hamiltonians is elliptic in $\lambda$ with respect to the lattice $P\oplus \tau P$. 
Then we have the following result, whose proof was suggested to the author by P.~Etingof.

\begin{prop}\label{elcl} Let $y_\xi(\lambda)$, $y_\xi^{c}(\lambda)$ denote the quantum and classical elliptic Dunkl operators associated to a root system $R$ according to \eqref{edu}, \eqref{educ}. Let 
$L_q^{\vee, c}(\xi, \lambda)$, $q\in (SV)^W$ be the ``dual'' classical higher Hamiltonians associated with 
\eqref{dh}. Then:

(i) $L_q^{\vee, c}(y_\xi(\lambda), \lambda)$, viewed as an element of $\D(V)*W$ depending on $\lambda$, is regular for $\lambda$ near $\lambda=0$;

(ii)  $L_q^{\vee, c}(y_\xi^{c}(\lambda), \lambda)$, viewed as an element of $\c(V\times V)*W$ depending on $\lambda$, is regular for all $\lambda\in V$;

(iii) $L_q^{\vee, c}(y_\xi^{c}\lambda), \lambda)$ is constant in $\lambda$. Moreover, expanding $L_q^{\vee, c}(y_\xi^{c}(\lambda), \lambda)$ as $\sum_{w\in W} a_w w$ with $a_w\in\c(V\times V)$, we have $a_w=0$ for $w\ne \id$.  
\end{prop}     

\begin{proof}
{\it Part (i)}. In the case when $L_q^{\vee, c}$ is replaced by the Hamiltonian of the {\it rational} Calogero--Moser system, the regularity statement is Theorem 3.1 in \cite{EFMV}, and we can use the same method. More precisely, two different proofs of the regularity are given in \cite{EFMV}. The first proof does not extend easily to the elliptic case because it requires establishing (iii) in advance, which we cannot do. However, the second proof as in \cite[5.3]{EFMV} extends to the elliptic case verbatim. 

{\it Part (ii)}. The operators $y_\xi^{c}(\lambda)$ are the $\hbar=0$ limit of $y_\xi(\lambda)$, thus, the expression $L_q^{\vee, c}(y_\xi^{c}(\lambda), \lambda)$ remains regular near $\lambda=0$ by (i). Other possible singularities are along the hyperplanes $\dpr{\alpha^\vee,\lambda}=m+n\tau$ with $m,n\in\Z$. To rule them out, let us see how $y_\xi^{c}(\lambda)$ changes under translations $\lambda\mapsto \lambda+u+\tau v$ with $u,v\in P$. Note that $\sigma_{\mu+1}(z)=\sigma_\mu(z)$ and $\sigma_{\mu+\tau}(z)=e^{2\pi \mathrm{i} z}\sigma_\mu(z)$, by properties of $\theta_1(z|\tau)$. This gives
\begin{equation*}
y_\xi^{c}(\lambda+u+\tau v)=p_\xi+\sum_{\alpha>0}c_\alpha\dpr{\alpha, \xi} e^{2\pi \mathrm{i}\dpr{\alpha, x}\dpr{\alpha^\vee, v}}\sigma_{\dpr{\alpha^\vee, \lambda}}(\dpr{\alpha, x})s_\alpha=e^{2\pi \mathrm{i}\dpr{v, x}}y_\xi^{c}(\lambda)e^{-2\pi \mathrm{i}\dpr{v,x}}\,.  
\end{equation*}
Each Hamiltonian $L_q^{\vee, c}$ is elliptic in $\lambda$, so the expression $L_q^{\vee, c}(y_\xi^{c}(\lambda), \lambda)$ has the same translation properties:  
\begin{equation}\la{trp}
L_q^{\vee, c}(y_\xi^{c}(\lambda), \lambda)\  \xrightarrow{\lambda\mapsto \lambda+u+\tau v} \ e^{2\pi \mathrm{i}\dpr{v,x}}L_q^{\vee, c}(y_\xi^{c}(\lambda), \lambda)e^{-2\pi \mathrm{i}\dpr{v,x}}\,.
\end{equation}
Since we know that the right-hand side is regular along each of the hyperplane $\dpr{\alpha^\vee, \lambda}=0$, the left-hand side must be regular along the shifted hyperplanes. As a result, it is regular everywhere.

{\it Part (iii)}. Let us expand $L_q^{\vee, c}(y_\xi^{c}(\lambda), \lambda)$ as $\sum_{w\in W} a_w w$. Each coefficient $a_w$ is a function of $x, p, \lambda$, and by (ii) it is globally holomorphic in $\lambda$. From \eqref{trp} we have that, as a function of $\lambda$, $a_w$ is quasi-periodic with respect to the lattice $P\oplus \tau P$. However, a holomorphic quasi-periodic function must be a constant, which proves that each $a_w$ is constant in $\lambda$. Invoking \eqref{trp} once again, we obtain $a_w w=e^{2\pi \mathrm{i}\dpr{v,x}}\,a_w w\, e^{-2\pi \mathrm{i}\dpr{v,x}}$ for all $v\in P$, from which it follows that $a_w=0$ for $w\ne \id$.  
\end{proof}

\medskip

\begin{remark} For $q=\frac12\dpr{\xi, \xi}$, all the statements of the proposition follow directly from the calculations in \eqref{twos}--\eqref{elcm1}. Note that when the quantum Dunkl operators are substituted, $\widehat A$ will contain terms of the form $\sigma'_\mu(z)$, which are not everywhere regular in $\mu$. This shows that the global regularity property (ii) does not hold in the quantum setting (contrary to the claim made in \cite[Remark 3.8]{EFMV}).   
\end{remark}

\begin{remark}
For the $BC_n$-\,case with coupling constants $c, g_0, g_1, g_2, g_3$, the proposition remains true, with the same proof as above. Note that the dual classical Hamiltonians in this case have coupling constants $c$ and $g^\vee_0, g^\vee_1, g^\vee_2, g^\vee_3$ as defined in \eqref{ghat} (cf. Example 3.9 in \cite{EFMV} and \eqref{bcreg} above). A small modification is required for the proof of regularity at $\xi_i=\omega_r$ with $r=1, 2, 3$. For this, one needs to employ shifts $\xi\mapsto \xi+\omega_r (\epsilon_1+\dots+\epsilon_n)$. For such translations there is a formula similar to \eqref{trp}, but now it also involves a permutation of the parameters $g_r$. The rest of the proof does not change. 
\end{remark}

\begin{remark}
Lax pairs for the trigonometric Calogero--Moser system can be obtained by replacing $\sigma_\mu(z)$ and $\wp(z)$ by their trigonometric versions, $\sigma_\mu(z)=\cot z-\cot \mu$ and $\wp(z)=\sin^{-2} z$. 
In the $BC_n$ case, the function $v_\mu(z)$ should also be replaced by its trigonometric version, $v_\mu(z)=g_0(\cot z-\cot 2\mu)+g_1(-\tan z-\cot 2\mu)$, cf. \cite{FeP, Pu}. It is customary in the trigonometric case to set the spectral parameter to a specific value. For instance, a trigonometric version of the Lax pair \eqref{ael} would have off-diagonal entries $\mathcal L_{kl}=\mathrm{i}g\sigma_\mu(x_k-x_l)=\mathrm{i}g\cot(x_k-x_l)-\cot\mu$, and setting $\mu=\pi/2$ would make it into  $\mathcal L_{kl}=\mathrm{i}g\cot(x_k-x_l)$.      
\end{remark}

\begin{remark}
By the same method, the proposition can be proved for a more general class of {\it crystallographic elliptic Calogero--Moser systems} constructed in \cite{EFMV}, for which the group $W$ is not necessarily a Weyl group. 
\end{remark}

\subsection{} We can now construct a Lax partner $\mathcal A$ for any of the higher quantum Hamiltonians. 
Namely, consider $L_q^{\vee, c}(y_\xi(\lambda), \lambda)$ with $q\in (SV)^W$. From the $W$-invariance of $L_q^{\vee, c}$ 
and $W$-equivariance of $y_\xi(\lambda)$, we have 
\begin{equation}\la{eqq}
w_x\,L_q^{\vee, c}(y_\xi(\lambda), \lambda)\,w_x^{-1}=L_q^{\vee, c}(y_{w\xi}(w\lambda), \lambda)=L_q^{\vee, c}(y_\xi(w\lambda), w^{-1}\lambda)\qquad\text{for all}\  w\in W\,,
\end{equation}
where we use the subscript to indicate that $w_x$ acts in the $x$-variable. In the limit $\lambda=0$ this becomes $W$-invariant, so we have
\begin{equation*}
L_q^{\vee, c}(y_\xi(0), 0)e=\widehat H e\,,\qquad \widehat H\in\D(V)^W\,,
\end{equation*}
where $\widehat H$ is one of the commuting Hamiltonians of the Calogero--Moser system \eqref{elcm}. By Proposition \ref{elcl}, the classical limit 
of $\widehat H$ can be obtained as   
\begin{equation}\la{d1}
H=L_q^{\vee, c}(y_\xi^{c}(\lambda), \lambda)\,.
\end{equation}
Now write 
\begin{equation}\la{d2}
L_q^{\vee, c}(y_\xi(\lambda), \lambda)=\widehat H+\widehat A\,,
\end{equation} 
for a suitable $\widehat A\in \D(V)*W$ (depending on $\lambda$). 

Obviously, $[y_\xi(\lambda), \widehat H +\widehat A]=0$, so a quantum Lax pair of size $|W|$ can be constructed as before. Let us now specialise both $\xi$ and $\lambda$ to have the same stabiliser $W'$. In this case, $y_\xi(\lambda)$ will be $W'$-invariant by \eqref{eduprop}. Also,  $L_q^{\vee, c}(y_\xi(\lambda), \lambda)$ and, therefore, $\widehat A$ are $W'$-invariant by \eqref{eqq}. As a result, $\mathcal L$, $\mathcal H$, $\mathcal A$ can be restricted onto $M'=e'M$, giving a Lax pair of size $|W/W'|$.  

It remains to explain why the constructed Lax pairs have a classical limit. We know that the classical limit of $\widehat H$ is the classical Hamiltonian $H$. Now, comparing \eqref{d1} and \eqref{d2}, we see that the classical limit of $\widehat A$ is zero. Therefore, the classical limit of $(\mathrm{i}\hbar)^{-1}\widehat A$ is well-defined and this produces the classical Lax partner $A$ in the same way as before.   
 
\subsection{} Let us mention some consequences of the above for the classical systems. 

\begin{prop} 
Let $\xi=b_i$ be a fundamental (co)weight for a root system $R$, and $W'$ denote the stabiliser of $\xi$ in the Weyl group $W$ of $R$. The classical elliptic Calogero--Moser system for a root system $R$ admits a Lax matrix $L$ of size $|W/W'|$ with a spectral parameter. Each of the commuting Hamiltonian flows of the system induces an isospectral deformation of $L$. The functions $h_k=\tr L^k$, $k\in\N$, form an involutive family, that is, $\{\tr\, L^a, \tr \,L^b\}=0$ for all $a,b$.   
\end{prop}

\begin{proof} The Lax matrix $L$ is constructed from $y_\xi(\lambda)$ by taking $\xi=b_i$ and $\lambda=\mu b_i$, so it has $\mu$ as a spectral parameter. The isospecrality of $L$ is a direct corollary of the existence of a Lax partner for each of the Hamiltonians. Let $H_1, \dots, H_n$ be the full set of the commuting Hamiltonians. From the isospectrality of $L$ we know that each $h_k=\tr L^k$ remains constant under any of the commuting flows. Therefore, $\{h_k, H_i\}=0$ for all $i=1, \dots n$. An abelian Poisson subalgebra in $\c(V\times V)$ cannot have more than $n$ functionally independent elements, therefore, each $h_k$ is a function of $H_1, \dots, H_n$. It follows that $\{h_k, h_l\}=0$, as needed.     
\end{proof}
As a corollary, we can derive the following result. 
\begin{cor}\la{corinoz}
Let $L$ be the classical Lax matrix \eqref{inolax} for the Inosemtsev system. Then the functions $h_{k}=\tr L^{2k}$, $k=1, \dots, n$ form a complete set of first integrals in involution.
\end{cor}

\begin{proof}
From the structure of $L$ it is clear that each $h_{k}$ is polynomial in momenta, with the leading terms $p_1^{2k}+\dots +p_n^{2k}$. Thus, $h_1, \dots, h_n$ are functionally independent. Their involutivity was shown in the proposition above.   
\end{proof}

\section{Elliptic difference case}\la{eqcase} Here we extend our approach to the elliptic Ruijsenaars--Schneider system and its versions for other root systems.  
The corresponding generalisation of the Cherednik and Macdonald operators to the elliptic case was found by Komori and Hikami in \cite{KH98}, by developing the approach of \cite{C5}. 
We will refer to these systems as \emph{generalised Ruijsenaars systems}. The main tool is elliptic $R$-matrices, first introduced in type $A$ in \cite{SU}. We adjust some of the notation of \cite{KH98} to make it closer to ours. 
The $C^\vee C_n$ case is related to the elliptic van Diejen system \cite{vD1} and is treated separately in  Subsections \ref{ecc}--\ref{calvd}.  

\subsection{} The setting is the same as in \ref{3.1}: $R\subset V$ is a reduced, irreducible root system with Weyl group $W$, $\Ra\subset\VV=V\oplus\c\delta$ is the associated affine root system, $\WW=W\ltimes t(Q^\vee)$ (resp.  $\Wh=W\ltimes t(P^\vee)$) is the affine (resp. extended affine) Weyl group. 
We choose a basis $a_0, \dots, a_n$ of $\Ra$, writing $s_i$ for the corresponding simple reflections. Recall that $\Wh=\WW\rtimes\Omega$ with $\Omega\cong P^\vee/Q^\vee$, and we have the length function $l(w)$ on $\Wh$.  

As in \ref{3.1}, we choose $q=e^c$ and consider the algebra $\c(V)*\WW\cong {\D}_q * W$ of reflection-difference operators on $V$. Fix a set of $W$-invariant coupling constants ${m}_\alpha$, $\alpha\in R$ (so in particular $m_{-\alpha}=m_\alpha$). For $\aalpha=\alpha+k\delta\in\Ra$, define $R$-matrices $\RR(\aalpha)$ to be the following elements of $\c(V)*\WW$ :
\begin{equation}\la{raa}
\RR(\aalpha)=\sigma_{{m}_\alpha}(\aalpha)-\sigma_{\dpr{\alpha^\vee,\,\xi}}(\aalpha)s_{\aalpha}\,,
\end{equation}
where $\xi\in V$ are the dynamical variables, and $\sigma_{\mu}(z)$ is the function \eqref{sigm}. According to \cite[(4.5)]{KH98}, we have
\begin{equation}\la{uni}
\RR_\alpha \RR_{-\alpha}=\wp({m}_\alpha)-\wp(\dpr{\alpha^\vee, \xi})\,.
\end{equation} 

\begin{defi}\la{rwdef}
Define a set $\{\RR_w\,| w\in\Wh\}$ by taking a reduced decomposition $w=s_{i_1}\dots s_{i_l}\pi$, $\pi\in\Omega$ and setting
\begin{equation}\la{rw}
\RR_w=\RR(\alpha^{1})\dots \RR(\alpha^{l})\,,\quad\text{where}\quad \alpha^{1}=a_{i_1},\ \alpha^{2}=s_{i_1}(a_{i_2})\,,\ \dots,\ \alpha^{l}=s_{i_1}\dots s_{i_{l-1}}(a_{i_l})\,. 
\end{equation}
In particular, we have $\RR_{s_i}=\RR(a_i)$, $i=0 \dots n$, and $\RR_\pi=1$ for $\pi\in\Omega$. {\it Elliptic Cherednik operators} are defined as $Y^b=\RR_{t(b)}\,t(b)$, $b\in P^\vee$. 
\end{defi}

\begin{theorem}[\cite{KH98}, Theorems 3.2 \& 4.3] 
\la{yb}
(i) The elements $\RR_w$ do not depend on the choice of a reduced decomposition for $w$;

(ii) $Y^b$, $b\in P^\vee_+$ are pairwise commuting elements of $\D_q*W$. 
\end{theorem}
The proof is based on the fact that $\RR(\aalpha)$ satisfy the affine Yang--Baxter relations as in \cite[(3.1a)--(3.1d)]{KH98}; this idea goes back to Cherednik \cite{C5}. 
The elliptic Macdonald--Ruijsenaars operators are obtained from the operators $Y^b$ in the following way. Introduce the vector
\begin{equation}\la{xi0}
\rho_m=\frac12 \sum_{\alpha\in R_+}m_\alpha\alpha\,.
\end{equation}
It satisfies the following equations, cf. \cite[(4.9)]{KH98}:   
\begin{equation}\la{system}
\dpr{a_i^\vee, \rho_m}={m}_{a_i}\,,\qquad i=1, \dots, n\,.
\end{equation} 

\begin{theorem}[\cite{KH98}, Theorem 4.5]\la{emr} Given $b\in P_+^\vee$, let $L^b\in \D_q$ be the unique difference operator such that $Y^b e= L^b e$. If $\xi=-\rho_m$ then each $L^b$ is $W$-invariant, 
and the difference operators $L^b$, $b\in P^\vee_+$ form a commutative family.
\end{theorem}


\begin{remark}
In \cite{KH98} the operators $Y^b$ with antidominant $b\in P^\vee_-$ are considered, but in fact they are the same as ours with $b\in P^\vee_+$. This is because the translations $t(v)$ are defined in \cite[(2.10)]{KH98} with the opposite sign compared to our conventions. 
\end{remark}

The operators $L^b$ are complicated in general, but admit an explicit description when $b$ is minuscule or quasi-minuscule. 

\begin{theorem}[\cite{KH98}, Theorems 6.1\,\& 6.5]\la{lb}  (i) Let $b$ be a minuscule coweight, so that $\dpr{\alpha, b}$ is either $0$ or $1$ for any $\alpha\in R_+$. Then we have
\begin{equation}
L^b=\sum_{\pi\in Wb} \,A_\pi t(\pi)\,,\qquad A_\pi=\prod_{\genfrac{}{}{0pt}{}{\alpha\in R}{\dpr{\pi, \alpha}>0}}\, \sigma_{{m}_\alpha}(\alpha)\,.
\end{equation}

(ii) Let $b$ be a quasi-minuscule coweight of the form $b=\varphi^\vee$, with $\varphi\in R_+$ the highest root. In this case, $\dpr{\alpha, b}\in\{0, 1\}$ for any $\alpha\in R_+\setminus\{b\}$. Then
\begin{align}\la{lbq1}
L^{b}&=\sum_{\pi\in Wb} \, (A_\pi t(\pi)-B_\pi)\,,\qquad A_\pi=\sigma_{{m}_{\varphi}}(\pi^\vee+\delta)
\prod_{\genfrac{}{}{0pt}{}{\alpha\in R}{\dpr{\pi, \alpha}>0}}\, \sigma_{{m}_\alpha}(\alpha)\,,
\\\la{lbq2}
B_\pi&=\sigma_{\dpr{\varphi^\vee, -\rho_m}}(\pi^\vee+\delta)
\prod_{\genfrac{}{}{0pt}{}{\alpha\in R}{\dpr{\pi, \alpha}>0}}\, \sigma_{{m}_\alpha}(\alpha)\,.
\end{align}
In these formulas the roots are viewed as affine-linear functions, so, for example,  $\sigma_{{m}_\alpha}(\alpha+\delta)=\sigma_{{m}_\alpha}(\dpr{\alpha, x}+c)$.
\end{theorem}



\begin{remark} To connect the operators $Y^b$ to the trigonometric Cherednik operators from \ref{trigch}, we rescale $x\mapsto x/2\pi \mathrm{i}$ and let $\tau\to+\mathrm{i}\infty$. In this limit
\eqref{raa} becomes
\begin{equation}\la{rdeftr}
\RR(\aalpha)=\frac{\sinh((\aalpha-{m}_\alpha)/2)}{\sinh(\aalpha/2)\,\sinh(-{m}_\alpha/2)}-\frac{\sinh(\aalpha-\dpr{\alpha^\vee, \xi}/2)}{\sinh(\aalpha/2)\,\sinh(-\dpr{\alpha^\vee,\,\xi}/2)}s_{\aalpha}\,.
\end{equation}
Let us set $\tau_{\aalpha}^2=e^{-{m}_\alpha}$\,, $\eta_\alpha^2=e^{\dpr{\alpha^\vee, \xi}}$ and rescale $\RR(\aalpha)$ by multiplying it by $\sinh(-{m}_\alpha/2)$. This gives
\begin{equation*}
\RR(\aalpha)=\frac{\tau_{\aalpha}^{-1}-\tau_{\aalpha} e^{\aalpha}}{1-e^{\aalpha}}+\frac{(\tau_{\aalpha}-\tau_{\aalpha}^{-1})(\eta_\alpha-\eta_\alpha^{-1}e^{\aalpha})}{(\eta_\alpha-\eta_\alpha^{-1})(1-e^{\aalpha})}s_{\aalpha}\,.
\end{equation*}
If we take $\xi\to\infty$ deep inside the positive Weyl cone, then $\eta_\alpha\to\infty$ for $\alpha>0$, which turns the above formua into \eqref{rdef} (assuming $\alpha>0$). It is known that for $w=t(b)$ with $b\in P^\vee_+$ all the $R$-matrices that appear in the decomposition \eqref{rw} will be of the form $R(\alpha+k\delta)$ with $\alpha\in R_+$, $k\ge 0$. Thus, the elliptic operators $Y^b$, $b\in P^\vee_+$ in the trigonometric limit coincide with the operators from \ref{trigch}, up to an overall factor.  
\end{remark}

\begin{remark} In the $\GL_n$ case, the operators $L^b$ from Theorem \ref{lb} are equivalent to the commuting Hamiltonians of the elliptic Ruijsenaars system \cite{R87}. For other root systems, the explicit operators $L^b$ can be viewed as elliptic generalisations of the Macdonald difference operators \cite{M87}. 
\end{remark}


\subsection{}\la{dura} For later purposes we also need a ``dual'' version of the Cherednik operators for the affine root system $(\Ra)^\vee$. To any affine root $\aalpha=\alpha+k\delta$ we associate a coroot $\aalpha^\vee$ by the formula 
\begin{equation*}
(\alpha+k\delta)^\vee=\frac{2}{\dpr{\alpha, \alpha}}(\alpha+k\delta)\,.
\end{equation*} 
The coroots $\aalpha^\vee$ with $\aalpha\in \Ra$ form a dual affine root system $(\Ra)^\vee\subset \VV$. The hyperplanes $\aalpha=0$ and $\aalpha^\vee=0$ are the same, so $s_{\aalpha}=s_{\aalpha^\vee}$ and both $\Ra$ and $(\Ra)^\vee$ share the same affine Weyl group $\WW$. The group $\Wh=W\rtimes P^\vee=\WW\rtimes\Omega$ permutes both coroots and roots, so we view it as an extended affine Weyl group for both systems. 
We can take $a_0^\vee, \dots, a_n^\vee$ as a basis of $(\Ra)^\vee$; note that $\Omega$ acts on this basis by permutations.
For $\aalpha=\alpha+k\delta\in \Ra$ we define  
\begin{equation}\la{raad}
\RR(\aalpha^\vee)=\sigma_{{m}_\alpha}(\aalpha^\vee)-\sigma_{\dpr{\alpha,\,\zeta}}(\aalpha^\vee)s_{\aalpha}\,,
\end{equation}
where the dynamical variable is $\zeta\in V$. We also define a set of elements $\{\RR_w^\vee\,|\, w\in\Wh\}$ by ``dualising'' \eqref{rw}:
\begin{equation*}
\RR_w^\vee=\RR(\alpha^{1})\dots \RR(\alpha^{l})\,,\quad\text{where}\quad \alpha^{1}=a_{i_1}^\vee,\ \alpha^{2}=s_{i_1}(a_{i_2}^\vee)\,,\ \dots,\ \alpha^{l}=s_{i_1}\dots s_{i_{l-1}}(a_{i_l}^\vee)\,. 
\end{equation*}
In particular, $\RR_{s_i}^\vee=\RR(a_i^\vee)$, $i=0 \dots n$, and $\RR_\pi^\vee=1$ for $\pi\in\Omega$.
Then the same arguments as in \cite{KH98} (see also Proposition \ref{prophat} below) prove the following results.
\begin{theorem} The elements $\RR_w^\vee$ do not depend on the choice of a reduced decomposition for $w\in\Wh$. The elements $Y^{b,\vee}:=\RR_{t(b)}^\vee\,t(b)$, $b\in P^\vee_+$ pairwise commute. 
\end{theorem}

\begin{theorem}\la{dham} Assume $\zeta=-\rho_m^\vee$ with $\rho_m^\vee=\frac12 \sum_{\alpha\in R_+}m_\alpha\alpha^\vee$. 
Given $b\in P_+^\vee$, define $L^{b,\vee}\in \D_q$ as the unique difference operator such that $Y^{b,\vee} e= L^{b,\vee} e$. Then each $L^{b,\vee}$ is $W$-invariant, 
and $L^{b,\vee}$ with $b\in P^\vee_+$ form a commutative family of difference operators.
\end{theorem}

\begin{theorem}\la{dhamm} (i) Let $b$ be a minuscule coweight. Then we have
\begin{equation}
L^{b,\vee}=\sum_{\pi\in Wb} \,A_\pi^\vee t(\pi)\,,\qquad A_\pi^\vee=\prod_{\genfrac{}{}{0pt}{}{\alpha\in R}{\dpr{\pi, \alpha}>0}}\, \sigma_{{m}_\alpha}(\alpha^\vee)\,.
\end{equation}

(ii) Let $b$ be a quasi-minuscule coweight of the form $b=\varphi^\vee$, with $\varphi\in R_+$ the highest root. Then
\begin{align}\la{ch1}
L^{b, \vee}&=\sum_{\pi\in Wb} \, (A_\pi^\vee t(\pi)-B_\pi^\vee)\,,\qquad 
A_\pi^\vee=\sigma_{{m}_\varphi}\left(\pi+2\delta/\dpr{\pi, \pi}\right)
\prod_{\genfrac{}{}{0pt}{}{\alpha\in R}{\dpr{\pi, \alpha}>0}}\, \sigma_{{m}_\alpha}(\alpha^\vee)\,,
\\
\la{ch2}
B_\pi^\vee&=\sigma_{\dpr{\varphi , -\rho_m^\vee}}\left(\pi+2\delta/\dpr{\pi, \pi}\right)
\prod_{\genfrac{}{}{0pt}{}{\alpha\in R}{\dpr{\pi, \alpha}>0}}\, \sigma_{{m}_\alpha}(\alpha^\vee)\,.
\end{align}
\end{theorem}

\subsection{}\la{arelel}
Before proceeding to the general construction of quantum Lax pairs, it will be instructive to discuss the $\GL_{n}$-case. Our setting will be similar to \ref{aq}: we take $V=\c^n$, with an orthonormal basis $\epsilon_1, \dots, \epsilon_n$ and the associated coordinates $x_1, \dots, x_n$, and with the standard action of $W=\mathfrak{S}_n$ on $V$. We set $\Lambda=\sum_{i=1}^n \Z\epsilon_i$, and consider the algebra of difference operators ${\D}_q=\c(x)\ltimes t(\Lambda)$, generated by $\c(x)$ and the shift operators $t(\epsilon_i)=q^{\partial_i}$, where $q=e^c$ with fixed $c$. As before, we will view reflection-difference operators in $\D_q*W$ acting on the module $M$ \eqref{mod}, and identify $M\cong \c W\otimes\c(x)$. As a result, we represent elements of $\D_q*W$ as operator-valued matrices of size $|W|$ \eqref{rep}.

We have one coupling constant ${m}_\alpha={\mu}$ for all $\alpha\in R$. For $\alpha=\epsilon_i-\epsilon_j$, $i\ne j$, the $R$-matrices \eqref{rdef} take the form \cite{SU}: 
\begin{equation}\la{rije}
\RR_{ij}=\sigma_{{\mu}}(x_i-x_j)-\sigma_{\xi_i-\xi_j}(x_i-x_j)\ss_{ij}\,.
\end{equation}
They satisfy the Yang--Baxter relations,
$\RR_{ij}\RR_{ik}\RR_{jk}=\RR_{jk}\RR_{ik}\RR_{ij}$. 
The elliptic Cherednik operators $Y_i:=Y^{\epsilon_i}$ can be calculated from Definition \ref{rwdef} (cf. \eqref{uni}):
\begin{equation}\la{yie}
Y_i=\RR_{i, i+1}\RR_{i, i+2}\dots \RR_{i, n}\,t(\epsilon_i)\,\RR_{i1}\dots \RR_{i, i-1}\quad (i=1, \dots, n)\,.
\end{equation}  
The elliptic Ruijsenaars operator $\widehat H=L^{\epsilon_1}$ is obtained from $Y_1=Y^{\epsilon_1}$ by Theorems \ref{emr}, \ref{lb}:
\begin{equation}\la{emr1}
\widehat H=\sum_{i=1}^n \prod_{j\ne i}^n \sigma_{\mu}(x_i-x_j)\,t(\epsilon_i)\,.
\end{equation}  
Up to a gauge transformation, this is the quantum Hamiltonian from \cite{R87}.
  
A quantum Lax pair will be constructed using $Y_1$ and $Y_2$, which are: 
\begin{equation}\la{y12}
Y_1=\RR_{12}\dots \RR_{1n}\,t(\epsilon_1)\,,\qquad Y_2=\RR_{23}\dots \RR_{2n}\,t(\epsilon_2)\RR_{21}\,.
\end{equation}

\begin{lemma}\la{inv12}
Write $W'=\mathfrak{S}_{n-1}$ for the subgroup fixing $\epsilon_1$, and $e'$ for the corresponding symmetrizer. If $\xi_i-\xi_{i+1}=-{\mu}$ for all $1<i<n$, then $Y_1$ and $Y_2$ preserve the subspace $M'=e'M$.
\end{lemma}

\begin{proof} This can be proved similarly to \cite[Theorem 4.5]{KH98}, but for the reader's convenience we give a self-contained proof. From the assumptions on $\xi$, for any $1<i<n$ we have 
\begin{equation*}
\RR_{i, i+1}=(1+\ss_{i, i+1})\sigma_{\mu}(x_i-x_{i+1})\,,\quad \RR_{i+1, i}=\sigma_{\mu}(x_{i+1}-x_i)(1-\ss_{i, i+1})\,,\quad \RR_{i+1, i}\RR_{i, i+1}=0\,.
\end{equation*}
Next, the Yang--Baxter relations imply that
\begin{equation*}
\RR_{i+1, i}Y_1=\RR_{12}\dots \RR_{1, i-1}\RR_{1, i+1}\RR_{1, i}\RR_{1, i+2}\dots \RR_{1n}\,t(\epsilon_1)\RR_{i+1, i}\,.
\end{equation*}
Assuming $i>1$, and multiplying by $e'$ from the right, we obtain
\begin{equation}\la{ybc}
(1-\ss_{i, i+1})Y_1e'=0\quad\text{for}\   i>1\,,
\end{equation} 
so $Y_1$ acts on $M'=e'M$ as a consequence.
For $Y_2$, we first notice that $\RR_{32}Y_2=0$ since $\RR_{32}\RR_{23}=0$, and so $(1-\ss_{23})Y_2=0$ as a consequence. Also, we have 
$(1-\ss_{i, i+1})Y_2e'=0$ for $i>2$; 
this follows from the Yang--Baxter relations in the same way as \eqref{ybc}. Putting  this together, we conclude that $(1-w)Y_2e'=0$ for all $w\in W'$, hence $Y_2(M')\subset M'$.       
\end{proof}

\begin{lemma}\la{nsel} Assume that $\xi_i-\xi_{i+1}=-{\mu}$ for all $1<i<n$. We have:
\begin{gather*}
Y_1\,|_{M'}=(A+\sum_{i\ne 1}^n B_i\ss_{1i})\,t(\epsilon_1)\,,\qquad Y_2\,|_{M'}=E+\sum_{i\ne 1}^n F_i\ss_{1i}\,, \quad\text{ where}\\ 
A=\prod_{l\ne 1}^n \sigma_{\mu}(x_1-x_l)\,,\qquad B_i=-\sigma_{\xi_1-\xi_2}(x_1-x_i)\prod_{l\ne 1, i} \sigma_{\mu}(x_i-x_l)\,,\\
E=\sum_{i>1}\sigma_{\mu}(x_i-x_1+c)\prod_{l\ne 1,i}\sigma_{\mu}(x_i-x_l)t(\epsilon_i)\,,\qquad F_i=\sigma_{\xi_1-\xi_2}(x_1-x_i-c)\prod_{l\ne 1, i}\sigma_{\mu}(x_i-x_l)\,t(\epsilon_i)\,.
\end{gather*}
\end{lemma}

\begin{proof} The formula for $Y_1$ is proved in the same way as in Lemma \ref{ns}. For $Y_2$ we can argue similarly. Namely, after 
expanding $Y_2$ \eqref{y12}, we obtain a sum of terms of the form
\begin{equation*}
\ss_{2i_1}\dots \ss_{2i_k} t(\epsilon_2) w_0\,,\quad\text{with}\  3\le i_1<\dots <i_k\le n \ \text{and}\ w_0\in\{\id, \ss_{12}\}\,. 
\end{equation*}
Multiplcation by $e'$ reduces this to $\ss_{1i}e'$ with $1 \le i\le n$, which equals $\ss_{12}e'$ only if the set $\{i_1, \dots, i_k\}$ is empty and $w_0=\ss_{12}$. Therefore, the coefficient $F_2$ is found from
\begin{equation*}
F_2 \ss_{12}=\prod_{l\ne 1,2}\sigma_{\mu}(x_2-x_l)t(\epsilon_2) \sigma_{\xi_1-\xi_2}(x_1-x_2) \ss_{12}\,.
\end{equation*}
This gives $F_2$ as stated in the lemma. The other coefficients $F_i$ are determined by the symmetry, $F_i=(F_2)^{\ss_{2i}}$.
It remains to determine the coefficient $E$. Arguing as above, we have several terms that reduce to $g e'$ with $g\in\D_q$, but only one of them will contain $t(\epsilon_2)$. Namely, this happens if 
$\{i_1, \dots, i_k\}=\emptyset$ and $w_0=\id$, so the corresponding term is 
\begin{equation*}
\prod_{l\ne 1,2}\sigma_{\mu}(x_2-x_l)t(\epsilon_2) \sigma_{{\mu}}(x_2-x_1)\,. 
\end{equation*}
If we combine this with the fact that $E$ is $W'$-invariant, we will arrive at the expression given in the lemma.  
\end{proof}

\medskip

To construct a Lax pair, we set $\widehat A=Y_1+Y_2-\widehat H$, where $\widehat H$ is the Hamiltonian \eqref{emr1}. Then from the above lemma, we have:
\begin{align}\la{ae1}
\widehat A\,|_{M'}&=\sum_{i=1}^n Z_i\ss_{1i}\,,\quad\text{with}\quad 
Z_1=\sum_{i\ne 1}(\sigma_{\mu}(x_{i1}+c)-\sigma_{\mu}(x_{i1}))\prod_{l\ne 1, i} \sigma_{\mu}(x_{il})\,t(\epsilon_i)\,,\\
\la{ae2}
Z_i&=(\sigma_\eta(x_{1i}-c)-\sigma_\eta(x_{1i}))\prod_{l\ne 1, i} \sigma_{\mu}(x_{il})\,t(\epsilon_i)\qquad(i\ne 1)\,.
\end{align}
Here $\eta:=\xi_1-\xi_2$ and $x_{il}:=x_i-x_l$. 

Using Lemma \ref{alem}, we can now calculate the matrices $\mathcal L$, $\mathcal A$ of size $n$ that represent the action of $Y_1$ and $\widehat A$ on $M'$. The result is:
\begin{align*}
\mathcal L_{ij}&=
\begin{cases}
\left(\prod_{l\ne j} \sigma_{\mu}(x_{jl})\right) t(\epsilon_j)\quad&\text{for}\ i=j\\
-\sigma_\eta(x_{ij}) \left(\prod_{l\ne i,j}\sigma_{\mu}(x_{jl})\right) t(\epsilon_j)\quad&\text{for}\ i\ne j\,,
\end{cases}
\\
\mathcal A_{ij}&=
\begin{cases}
\sum_{k\ne j}(\sigma_{\mu}(x_{kj}+c)-\sigma_{\mu}(x_{kj}))\left(\prod_{l\ne j, k} \sigma_{\mu}(x_{kl})\right) t(\epsilon_k)\quad&\text{for}\ i=j\\
(\sigma_\eta(x_{ij}-c)-\sigma_\eta(x_{ij})) \left(\prod_{l\ne i,j}\sigma_{\mu}(x_{jl})\right) t(\epsilon_j)\quad&\text{for}\ i\ne j\,.
\end{cases}
\end{align*}
Since $[Y_1, Y_1+Y_2]=0$, the above $\mathcal L$, $\mathcal A$ and $\mathcal H=\widehat H \mathbb{1}$ satisfy the quantum Lax equation \eqref{qleq}.
Note that $\mathcal L$ is equivalent to the Lax matrix found by Hasegawa using a different approach, cf. \cite[eq. (38)]{Ha}. 


The classical limit corresponds to $c=-\mathrm{i}\hbar\beta\to 0$. In $\mathcal L$ we simply replace $t(\epsilon_j)=e^{c\partial_j}$ with $e^{\beta p_j}$:
\begin{equation*}
L_{ij}=
\begin{cases}
\left(\prod_{l\ne j} \sigma_{\mu}(x_{jl})\right) e^{\beta p_j}\quad&\text{for}\ i=j\\
-\sigma_\eta(x_{ij}) \left(\prod_{l\ne i,j}\sigma_{\mu}(x_{jl})\right) e^{\beta p_j} \quad&\text{for}\ i\ne j\,.
\end{cases}
\end{equation*}
The Lax partner $A$ is then found as the classical limit of $(\mathrm{i}\hbar)^{-1}\mathcal A$, which gives 
\begin{equation*}
A_{ij}=
\begin{cases}
-\sum_{k\ne j}  
\beta\sigma'_{\mu}(x_{kj})\left(\prod_{l\ne j, k} \sigma_{\mu}(x_{kl})\right) e^{\beta p_k}\quad&\text{for}\ i=j\\
\beta\sigma'_\eta(x_{ij})
\left(\prod_{l\ne i,j}\sigma_{\mu}(x_{jl})\right) e^{\beta p_j} \quad&\text{for}\ i\ne j\,.
\end{cases}
\end{equation*}
These matrices are equivalent to those known from \cite{R87, BC87, KZ95}.

\subsection{}\la{regu1} Let us now discuss a method for constructing Lax pairs in general. As before, we will view reflection-difference operators in $\D_q*W$ acting on the module $M$ \eqref{mod}, and identify $M\cong \c W\otimes\c(V)$. As a result, we represent elements of $\D_q*W$ as operator-valued matrices of size $|W|$ \eqref{rep}. Take an elliptic Cherednik operator $Y^b$, $b\in P^\vee_+$ and the corresponding quantum Hamiltonians $L^b\in\D_q^W$, constructed in Theorem \ref{emr}.  Our task is to find a combination of Cherednik operators which has the same classical limit as $L^b$. In the $\GL_n$-case above such a combination was $Y_1+Y_2$, but this does not seem very helpful in regards to the general case. There is, nevertheless, a natural analogue of Proposition \ref{elcl}, but it requires a renormalisation of the operators $\RR(\alpha)$ and $Y^b$.
Namely, we define the {\it unitary} affine $R$-matrices as follows: 
\begin{equation}\la{rh}
\RH(\aalpha)=\sigma_{{m}_\alpha}(\dpr{\alpha^\vee, \xi})^{-1}\RR(\aalpha)=\frac{\sigma_{{m}_\alpha}(\aalpha)}{\sigma_{{m}_\alpha}(\dpr{\alpha^\vee, \xi})}-\frac{\sigma_{\dpr{\alpha^\vee,\,\xi}}(\aalpha)}
{\sigma_{{m}_\alpha}(\dpr{\alpha^\vee, \xi})}s_{\aalpha}\,.
\end{equation}
This can be rewritten as
\begin{equation}\la{rh1}
\RH(\aalpha)=
\frac
           {\theta(\aalpha-{m}_\alpha)\theta(\dpr{\alpha^\vee, \xi})}
{\theta(\aalpha)\theta(\dpr{\alpha^\vee, \xi}-{m}_\alpha)}
-\frac
         {\theta(\aalpha-\dpr{\alpha^\vee, \xi})\theta(-{m}_\alpha)}
{\theta(\aalpha)\theta(\dpr{\alpha^\vee, \xi}-{m}_\alpha)}s_{\aalpha}\,.
\end{equation}
From \eqref{uni} it follows that 
\begin{equation}\la{uni1}
\RH(\aalpha)\RH(-\aalpha)=1\quad\text{for all}\  \aalpha\in\Ra\,.
\end{equation}
We also define $\RH_w$, $w\in\WW$ in the same way as in Definition \ref{rwdef}, but using the unitary $R$-matrices instead, and we set $\YH^b=\RH_{t(b)}\,t(b)$ for $b\in P^\vee_+$. The elements $\RH_w, \YH^b$ differ from $R_w, Y^b$ by a $\xi$-depending factor, and Theorem \ref{yb} remains valid for them.
The classical Cherednik operators are $Y^\lambda_c:=\eta_0(Y^\lambda)$ and $\YH^{\lambda}_{c}:=\eta_0(\YH^\lambda)$, where the classical-limit map $\eta_0$ is defined in \ref{climq}.  
We also wirte $L^{b,\vee}_{c}$ for the classical limit of the Hamiltonians $L^{b,\vee}$ from Theorems \ref{dham}, \ref{dhamm}; these are elements of $A_0=\c(V)[P^\vee]$, i.e. linear combinations of the terms $g(x)e^{\beta p_\lambda}$ with $\lambda\in P^\vee$, see \ref{climq} for the notation. Denote by    
$L^{b,\vee}_{c}(\xi, \YH)$ the result of replacing each term $g(x)e^{\beta p_\lambda}$ by $g(\xi)\YH^{\lambda}$. Similarly, we write $L^{b,\vee}_{c}(\xi, \YH_c)$ for the result of substituting the classical operators $\YH^\lambda_c$. 
We then have the following analogue of Proposition \ref{elcl}.

\begin{prop}\label{elclq} 
Let $L^{b,\vee}_{c}$, $b\in P^\vee_+$ be one of the dual classical Hamiltonians. 
Then, assuming that $b$ is (quasi-)minuscule, we have:

(i) $L^{b,\vee}_{c}(\xi, \YH)$, viewed as an element of $\D(V)*W$ depending on $\xi$, is regular for $\xi$ near $\xi=0$;

(ii)  $L^{b,\vee}_{c}(\xi, \YH_{c})$, viewed as an element of $\c(V\times V)*W$ depending on $\xi$, is regular for all $\xi\in V$;

(iii) $L^{b,\vee}_{c}(\xi, \YH_{c})$ is constant in $\xi$. Moreover, expanding $L^{b,\vee}_{c}(\xi, \YH_{c})$ as $\sum_{w\in W} a_w w$ with $a_w\in\c(V\times V)$, we have $a_w=0$ for $w\ne \id$.

(iv) We have 
$L^{b,\vee}_{c}(\xi, \YH_c)=L^{b}_c+\mathrm{const}$.
\end{prop}   
We will prove the proposition in \ref{regu}. We expect that a similar result is true for any of the classical Hamiltonians $L^{b,\vee}_c$, not only for those associated with (quasi-)minuscule coweights.

\medskip

A quantum Lax pair can now be constructed as in \eqref{d1}--\eqref{d2}, by writing 
\begin{equation}\la{de}
L^{b,\vee}_{c}(\xi, \YH)=L^b+\widehat A\,,
\end{equation}
for a suitable $\widehat A\in\D_q*W$. Pick a Cherednik operator $Y^\lambda$, $\lambda\in P^\vee_+$. Then $Y^\lambda$ and $L^b+\widehat A$ commute, and $L^b$ is $W$-invariant, which leads to a quantum Lax pair of size $|W|$. 
Since $L^{b,\vee}_{c}(\xi, \YH_{c})=L^{b}_{c}+\mathrm{const}$ by Proposition \ref{elclq}(iv), the classical limit of $\widehat A$ is a constant. We can modify $\widehat A$ by subtracting this constant; as a result, the constructed Lax pair will admit a classical limit. Finally, to reduce it to a Lax pair of a smaller size, we use the following lemma whose proof is postponed to \ref{regu}.

\begin{lemma}\la{il}
Let $\lambda=b_{k}$, $1\le k\le n$ be one of the fundamental coweights, $W'$ be its stabiliser and $e'$ be the corresponding symmetrizer. Furthermore, assume that $\xi=-\rho_m+\eta b_k$, with arbitrary $\eta$. 
Then for such $\lambda, \xi$ the action of $Y^\lambda$ and $L^{b,\vee}_c(\xi, \YH)$ on the module $M$ preserve the subspace $M'=e'M$. 
\end{lemma}
This means that the quantum Lax pair constructed from \eqref{de} 
can be restricted to $M'$, giving matrices of size $|W|/|W'|$. Note that the Lax pair depends on a spectral parameter $\eta$. The following theorem summarizes the obtained results.

\begin{theorem}\la{mainthm} For any fundamental coweight $\lambda=b_k\in P^\vee_+$, with the stabiliser $W'$, and any (quasi-)minuscule coweight $b$, there exists a quantum Lax pair $\mathcal L, \mathcal A$ of size $|W|/|W'|$ satisfying the Lax equation \eqref{qleq} with $\mathcal H=\widehat H\mathbb{1}$, $\widehat H=L^b$. This Lax pair depends on a spectral parameter and admits a classical limit.   
\end{theorem}

\subsection{}\la{regu} In this subsection we prove Proposition \ref{elclq} and Lemma \ref{il}. For this we need to analyse the operators $\RH_w$ and $\YH^b$ in more detail. First, let us introduce some notation. The $R$-matrices $\RR(\aalpha)$, $\RH(\aalpha)$ depend on $\xi\in V$, so it will be convenient to introduce $\mathbb V=V^\vee\times V$ and $\c(\mathbb V)$, where the first factor $V^\vee\cong V$ represents the $\xi$-variable. We also introduce $\mathbb W=W\times \Wh$, to allow $W$ acting on $\xi$; the subgroup $\id\times \Wh$ will be identified with $\Wh$. This makes $\c(\mathbb V)$ into a $\c\mathbb W$-module, so we form the product $\c(\mathbb V)\rtimes\c\mathbb W$. With each ${w}\in\Wh$ we associate an element $w^\vee\otimes w\in\mathbb W$, where $w^\vee\in W$ is the linear part of $w$; in particular, for $w=t(b)$ we have $w^\vee=1$, so $w^\vee\otimes w=1\otimes t(b)$. 

With this notation, the $R$-matrices can be viewed as elements of $\c(\mathbb V)\rtimes \c\mathbb W$. An important property is their equivariance in the following sense:
\begin{equation}\la{eqv}
(w^\vee\otimes w)\RR(\aalpha) (w^\vee\otimes w)^{-1}=\RR(w\aalpha)\,,\qquad\aalpha\in \Ra\,,\ w\in\Wh\,, 
\end{equation}
and the same for $\RH(\aalpha)$. In particular, for $w=t(b)$, $b\in P^\vee$ this implies (cf. \cite[(4.6)]{KH98})
\begin{equation}\la{rtr}
t(b)\RR(\alpha+k\delta) t(-b)=\RR(\alpha+k'\delta)\,,\quad k'=k+\dpr{\alpha, b}\,. 
\end{equation}
Another crucial property  is that $\RR(\aalpha)$ and $\RH(\aalpha)$ satisfy the affine Yang--Baxter relations \cite[(3.1)(a)-(d)]{KH98}. This property can be reformulated by setting 
\begin{equation}\la{thi}
\TH_i=\RH(a_i) (s_i^\vee\otimes s_i)\,,\quad \TH_\pi=\pi^\vee\otimes \pi
\end{equation}
for $i=0, \dots, n$ and $\pi\in\Omega$. Then one can check that the relations \cite[(3.1)(a)-(d)]{KH98} imply that $\TH_i$ and $\TH_\pi$ satisfy the relations of the braid group \eqref{braid}--\eqref{braid2}. Moreover, the unitarity \eqref{uni1} implies that $\TH_i^2=1$ for all $i$. Therefore, we have the following result. 
\begin{prop}\la{ass}
The assignment 
$s_i\mapsto \TH_i$ ($i=0,\dots, n$), $\pi\mapsto \TH_\pi$ ($\pi\in\Omega$)
extends to a group homomorphism $w\mapsto \TH_w$, $w\in\Wh$. In particular, the elements $\TH_{t(b)}$, $b\in P^\vee$ pairwise commute. 
\end{prop}
The elements $\TH_w$ can be rewritten in terms of the affine $R$-matrices, giving the following result.
\begin{prop}\la{prophat}
 For any reduced decomposition of $w$ into $w=s_{i_1}\dots s_{i_l}\pi$ with $\pi\in\Omega$ we have $\TH_w=\RH_w (w^\vee\otimes w)$, where 
\begin{equation}\la{rw1}
\RH_w=\RH(\alpha^{1})\dots \RH(\alpha^{l})\,,\qquad \alpha^{1}=a_{i_1},\ \alpha^{2}=s_{i_1}(a_{i_2})\,,\ \dots,\ \alpha^{l}=s_{i_1}\dots s_{i_{l-1}}(a_{i_l})\,. 
\end{equation}
As a result, $\RH_w$ does not depend on the choice of a decomposition. For $b\in P^\vee$ we have $\TH_{t(b)}=\RH_{t(b)}(1\otimes t(b))$, so it can be identified with $\YH^b$ as defined in \ref{regu1}. 
\end{prop}

The roots $\alpha^k$ appearing in \eqref{rw1} can be characterised geometrically. For this it will be convenient to work over $\R$, assuming $V\cong \R^n$ and setting $c=1$ in the definition of $t(\lambda)$ \eqref{tl}; then $\delta\equiv 1$ on $V$. Let $C_a$ be the Weyl alcove,
\begin{equation*}
C_a=\{x\in V\,|\,a_i(x)>0\ \text{for}\ i=0,\dots, n\}\,. 
\end{equation*}  
This is a fundamental domain for $\WW$, and each $\pi\in \Omega$ maps $C_a$ to itself. The set of positive roots $\Ra^+$ consists, by definition, of all those $\aalpha$ which take positive values on $C_a$. Then for any $w\in\Wh$, the set $\{\alpha^1, \dots, \alpha^l\}$ as defined in \eqref{rw1} consists of all $\aalpha\in \Ra^+$ for which the hyperplane $\aalpha=0$ separates $C_a$ and $w(C_a)$, see \cite[2.2]{M03} (note that our $w$ corresponds to $w^{-1}$ in \cite{M03}). Furthermore, the sequence of the hyperplanes $\alpha^1=0, \dots, \alpha^l=0$ can be obtained by taking a straight line between two generic points $x\in C_a$, $y\in wC_a$ and by listing the reflection hyperplanes which this line intersects as you go from $x$ to $y$. This determines each $\alpha^k$ up to a sign which can be further fixed by prescribing that $\alpha^k$ decreases as you move from $x$ to $y$ (this is because $\alpha^k$ is $>0$ on $C_a$ and $<0$ on $wC_a$).   

\begin{lemma}\la{lemmay}
For any $\alpha\in R$, we have $\YH^{\alpha^\vee}=1$ if $\dpr{\alpha^\vee, \xi}=0$.
\end{lemma}      
\begin{proof}
We may assume that $\alpha>0$. Consider $w=t(\alpha^\vee)\in \Wh$; it acts on $V$ by $w(x)=x-\alpha^\vee$. We need to show that for $\dpr{\alpha^\vee, \xi}=0$ we have $\RH_w=t(-\alpha^\vee)$. Take a generic $x\in C_a$ and consider a line from $x$ to $y=x-\alpha^\vee\in wC_a$. We have $0<\dpr{\alpha, x}<1$ and $\dpr{\alpha, y}=\dpr{\alpha, x}-2$, therefore the line intersects the hyperplanes $\alpha=0$ and $\alpha+\delta=0$. This tells us that $\alpha^i=\alpha$ and $\alpha^j=\alpha+\delta$ for some $1\le i<j\le l$. Note that $s_\alpha s_{\alpha+\delta}$ is the translation by $\alpha^\vee$, which is $t(-\alpha^\vee)$ under our conventions. 

Let $z=s_\alpha(x)=s_{\alpha+\delta}(y)$; this point divides the interval between $x, y$ in two parts. The interval between $x, z$ is symmetric under $s_\alpha$, therefore the collection of the hyperplanes it intersects with will be symmetric as well, giving $\alpha^{i+r}=\pm s_\alpha(\alpha^{i-r})$ for $0<r<i$. Moreover, the rate of change of each $\alpha^k$ in direction of $\alpha^\vee$ should be positive; this gives $\alpha^{i+r}=-s_\alpha(\alpha^{i-r})$. Now, from the formula \eqref{rh1} we get $\RH(\alpha)=-s_{\alpha}$ 
when $\dpr{\alpha^\vee, \xi}=0$. Therefore, assuming $\dpr{\alpha^\vee, \xi}=0$, we get
\begin{equation*}
\RH(\alpha^1)\dots \RH(\alpha^{2i-1})=\RH(\alpha^1)\dots \RH(\alpha^{i-1})(-s_\alpha) \RH(-s_\alpha(\alpha^{i-1}))\dots \RH(-s_\alpha(\alpha^1))\,.
\end{equation*}    
The middle factor $s_\alpha$ can be replaced by $s_\alpha^\vee\otimes s_\alpha\in\mathbb W$, since $s_\alpha^\vee$ acts trivially if $\dpr{\alpha^\vee, \xi}=0$. Using the equivariance \eqref{eqv}, we can move $s_\alpha^\vee\otimes s_\alpha$ to the left and obtain
 \begin{equation*}
\RH(\alpha^1)\dots \RH(\alpha^{2i-1})=-\RH(\alpha^1)\dots \RH(\alpha^{i-1}) \RH(-\alpha^{i-1})\dots \RH(-\alpha^1) (s_\alpha^\vee\otimes s_\alpha)=-s_\alpha\,,
\end{equation*}    
by unitarity. By the same arguments applied to the interval from $z$ to $y$, we get $\alpha^{j+r}=-s_{\alpha+\delta}(\alpha^{j-r})$ for all $1\le r\le l-j$ and so the product $\RH(\alpha^{2i})\dots \RH(\alpha^l)$ reduces to $-s_{\alpha+\delta}$. Putting this together gives us $\RH_w=s_{\alpha}s_{\alpha+\delta}=t(-\alpha^\vee)$, as needed.
\end{proof}

Recall that we have previously defined $\RR_{t(b)}$ and $Y^b$ in accordance with definition \ref{rwdef}; let us now compare them to $\RH_{t(b)}$ and $\YH^b$.  
\begin{lemma} \la{rrrh} 
For any $b\in P^\vee$, we have $\RR_{t(b)}=G_b\RH_{t(b)}$ and $Y^b=G_b\YH^b$, where
\begin{equation*}
{G_b(\xi)=\prod_{\genfrac{}{}{0pt}{}{\alpha\in R}{\dpr{\alpha, b}>0}}\, \left(\sigma_{m_\alpha}(\dpr{\alpha^\vee, \xi})\right)^{\dpr{\alpha, b}}}\,.
\end{equation*} 
\end{lemma}
\begin{proof}
This follows directly from the geometric description of the sequence $\alpha^1, \dots, \alpha^l$ given above. The exponent $\dpr{\alpha, b}$ in the formula counts how many roots of the form $\alpha+k\delta$, $k\in\Z$ will appear in that sequence. 
\end{proof}

\begin{remark}\la{ycom}
It follows that the elements $Y^b$ pairwise commute for all $b\in P^\vee$. However, the formula $Y^\lambda Y^\mu=Y^{\lambda+\mu}$ does not hold in general: it is only valid up to a $\xi$-depending factor. 
\end{remark}

We also need to know how the classical operators $Y^b_c$, $\YH^b_c$ behave under the shifts $\xi\mapsto \xi+u+\tau v$ with $u,v\in P$. 
\begin{lemma}\la{try} 
Under $\xi\mapsto \xi+ u+\tau v$, $u, v\in P$, the operator $Y^b_c$ changes to $e^{2\pi \mathrm{i}\dpr{v,x}} Y^b_c e^{-2\pi \mathrm{i}\dpr{v,x}}$.  
\end{lemma}
\begin{proof}
We have $\sigma_{\mu+1}=\sigma_\mu(z)$ and $\sigma_{\mu+\tau}(z)=e^{2\pi \mathrm{i} z}\sigma_\mu(z)$. Hence,
\begin{equation*}
\RR(\aalpha)\  \xrightarrow{\xi\mapsto \xi+\tau v} \sigma_{m_\alpha}(\aalpha)-e^{2\pi \mathrm{i}\dpr{\alpha^\vee, v}\aalpha} \sigma_{\dpr{\alpha^\vee, \xi}}(\aalpha)s_{\aalpha}=e^{2\pi \mathrm{i}\dpr{v,x}}\RR(\aalpha) \, e^{-2\pi \mathrm{i}\dpr{v,x}}\,.
\end{equation*}
Also, $t(b)=e^{2\pi \mathrm{i}c\dpr{v, b}}e^{2\pi \mathrm{i}\dpr{v,x}}\, t(b) \, e^{-2\pi \mathrm{i}\dpr{v,x}}$. Therefore, for $Y^b=R_{t(b)}t(b)$ we obtain
\begin{equation*}
Y^b\  \xrightarrow{\xi\mapsto \xi+\tau v} e^{2\pi \mathrm{i}c\dpr{v, b}}e^{2\pi \mathrm{i}\dpr{v,x}}\, Y^b \, e^{-2\pi \mathrm{i}\dpr{v,x}}\,.
\end{equation*}
In the classical limit we have $c=0$, so the lemma follows.
\end{proof}


\begin{lemma}\la{23}
For a function $g(\xi)$ and $Y^b$, $b\in P^\vee$, we have $\TH_w\,g(\xi)\YH^b=g(w^{-1}\xi)\YH^{wb}\,\TH_w$ for any $w\in W$.
\end{lemma}  
\begin{proof}
It suffices to check that $\TH_ig(\xi)=g(s_i\xi)\TH_i$ and $\TH_i\YH^b=\YH^{s_ib}\TH_i$ for $i=1, \dots, n$. The first claim is immediate from \eqref{rh}, \eqref{thi}.  Also, since $\YH^b=\TH_{t(b)}$, we have
$\TH_i\YH^b=\TH_{s_it(b)}=\TH_{t(s_ib)s_i}=\YH^{s_ib}\TH_i$,
as needed. 
\end{proof}

For the rest of this subsection, $b\in P^\vee$ is assumed to be dominant and (quasi-)minuscule. 
\begin{lemma}\la{re} The only possible singularities of $L^{b,\vee}_c(\xi, \YH)$ in the $\xi$-variable are along the hyperplanes $\dpr{\alpha^\vee, \xi}=k+l\tau$, with $\alpha\in R$, $k,l\in\Z$.
\end{lemma}
\begin{proof}
The operators $\YH^\pi$ are defined in terms of $\RH(\aalpha)$, which have singularities along hyperplanes $\theta(\dpr{a_i^\vee,\xi_0}-m_\alpha)=0$. However, these are removable due to the coefficients in $L^{b,\vee}_c(\xi, \gamma)$, as can be readily seen from the formulas in Theorem \ref{dhamm}. Indeed, in the classical limit the coefficients $A_\pi^\vee$ coincide with $G_b(\xi)$ from Lemma \ref{rrrh}. Therefore, $A_\pi^\vee\YH^\pi=Y^\pi$, and the latter has singularitites only where $\theta(\dpr{\alpha^\vee, \xi})=0$.
\end{proof}

\begin{lemma} \la{trans} For $u, v\in P$, we have
\begin{equation*}
L^{b,\vee}_c(\xi, \YH_c)\  \xrightarrow{\xi\mapsto \xi+\tau v}  e^{2\pi \mathrm{i}\dpr{v,x}} L^{b,\vee}_c(\xi, \YH_c) e^{-2\pi \mathrm{i}\dpr{v,x}}\,.
\end{equation*} 
\end{lemma}
\begin{proof}
If $b$ is minuscule, then from the proof of the previous lemma we see that $L^{b,\vee}_c(\xi, \YH_c)$ is a sum of the terms $Y^\pi_c$, so the result follows from Lemma \ref{try}. If $b$ is quasi-minuscule, then we have in addition the sum of $B^\vee_\pi$ over $\pi\in Wb$, which is a function of $\xi$ only. One checks directly from the formula that each $B^\vee_\pi$ is elliptic in $\xi$, in both the quantum and classical settings. This implies that the lemma is true in this case as well.
\end{proof}

\begin{cor}
(i) For any $b\in P^\vee$ we have $\TH_iL^{b,\vee}_c(\xi, \YH)=L^{b,\vee}_c(\xi, \YH)\TH_i$ for $i=1, \dots, n$.

(ii) If $\lambda=b_k$ then $\TH_iY^\lambda=Y^\lambda\TH_i$ for $i\ne k$.
\end{cor}
\begin{proof}
The first statement follows from Lemma \ref{23}, since the classical Hamiltonian $L^{b,\vee}_c(\xi, \gamma)$ is $W$-invariant. By the same lemma, $\TH_i\YH^\lambda=\YH^\lambda \TH_i$ if $s_i(\lambda)=\lambda$. The second statement now follows from Lemma \ref{rrrh}, since it is easy to check that the factor $G_b$ in this case is invariant under $s_i$.  
\end{proof}

\medskip

\noindent {\bf Proof of Lemma \ref{il}}. Using the above corollary together with \eqref{thi}, we have $\RH(a_i)(s_i^\vee\otimes s_i)Y^\lambda=Y^\lambda\RH(a_i)(s_i^\vee\otimes s_i)$ for $i\ne k$, from which $\RH(-a_i)Y^\lambda=(s_i^\vee\otimes s_i)Y^\lambda (s_i^\vee\otimes s_i)\,\RH(-a_i)$. Multiplying this by $\sigma_{m_{a_i}}(-\dpr{a_i^\vee, \xi})$, we obtain (cf. \cite[(4.7)]{KH98}): 
\begin{equation*}
\RR(-a_i)Y^\lambda=(s_i^\vee\otimes s_i) Y^\lambda (s_i^\vee\otimes s_i)\,\RR(-a_i)\quad\text{for}\ i\ne k\,.
\end{equation*}
If $\xi=-\rho_m+\eta b_k$ then $\dpr{a_i^\vee, \xi}=-m_{a_i}$ for $i\ne k$. In this case directly from the definitions, $\RR(-a_i)=\sigma_{m_{a_i}}(-a_i)(1-s_i)$. So if we use this in the previous relation, and multiply it by $e'$, we get 
$(1-s_i)Y^\lambda e'=0$. Since this holds for all $i\ne k$, we conclude that $Y^\lambda$ preserves the subspace $M'=e'M$.
The statement about $L^{b,\vee}_c(\xi, \YH)$ is proved in exactly the same way. \qed

\medskip

\noindent {\bf Proof of Proposition \ref{elclq}}. {\it Part (i)}. We pick $\alpha\in R$ and want to show that $L^{b,\vee}_c(\xi, \YH)$ is regular along $\dpr{\alpha, \xi}=0$. Let us first take $b$ to be minuscule.  
The orbit $Wb$ breaks into pairs $\pi,\,\pi'$ with $\pi'=s_\alpha(\pi)$, plus a
number of $s_\alpha$-invariant $\pi$'s. If $s_\alpha(\pi)=\pi$, then the coefficient $A_\pi^\vee$ does not contain the factor $\sigma_{m_\alpha}(\dpr{\alpha^\vee, \xi})$ and so is regular. The other case leads to 
\begin{equation*}
A\YH^\pi+A'\YH^{\pi'}\,,\qquad A=A_\pi^\vee\,,\ A'=A_{\pi'}^\vee\,,\quad  
\pi'=\pi-\alpha^\vee\,.
\end{equation*}
The coefficients $A, A'$ have first order poles along $\dpr{\alpha^\vee, \xi}=0$ and satisfy $A'=A^{s_\alpha}$ due to the $W$-symmetry of $L^{b, \vee}$. As a result, $A+A'$ is regular along the hyperplane $\dpr{\alpha^\vee, \xi}=0$. Now, by Lemma \ref{lemmay} we have $\YH^{\pi'}=\YH^\pi \YH^{\alpha^\vee}=\YH^\pi$ at the hyperplane $\dpr{\alpha^\vee, \xi}=0$. It follows that at this hyperplane $A\YH^\pi+A'\YH^{\pi'}=(A+A')\YH^\pi$, and this expression is regular.

Now consider the case when $b=\varphi^\vee$ is quasi-minuscule. In this case, there is one additional possibility when $\pi'=s_\alpha(\pi)$ with $\pi'=\pi-2\alpha^\vee$; this happens only when $\alpha^\vee\in Wb$. In this case, we are led to consider 
\begin{equation}\la{3term}
A\YH^{\alpha^\vee}+A'\YH^{-\alpha^\vee}+B\,,\qquad A=(A_\pi^\vee)_c\,,\ A'=(A_{\pi'}^\vee)_c\,,\quad  B=-\sum_{\pi\in Wb}(B_\pi^\vee)_c\,.
\end{equation}
Note that in the classical limit we have $\delta=0$, so the formulas \eqref{ch1}--\eqref{ch2} tell us that the coefficients $A, A', B$ will have second order poles along $\dpr{\alpha^\vee, \xi}=0$ (the singularity in $B$ comes from $B_{\alpha}^\vee+B_{-\alpha^\vee}^\vee$). We have the folowing properties of $A, A', B$, first two of which follow from the $W$-symmetry of $L^{b,\vee}$, and the last one can be checked by inspecting the formulas \eqref{ch1}--\eqref{ch2} (in the case $\delta=0$):
\begin{gather}\nonumber
\text{$A'=A^{s_\alpha}$, $B=B^{s_\alpha}$;}
\\\la{aab}
\text{$A+A'$ and $B$ have zero residue at $\dpr{\alpha^\vee, \xi}=0$;}
\\\nonumber
\text{$A+A'+B$ has at most simple pole along $\dpr{\alpha^\vee, \xi}=0$.}
\end{gather}

\noindent In addition, it follows from Lemma \ref{lemmay} that
\begin{equation*}
\YH^{\alpha^\vee}=1+\varepsilon\dpr{\alpha^\vee, \xi}+o(\dpr{\alpha^\vee, \xi})\quad\text{ near}\  \dpr{\alpha^\vee, \xi}=0\,,
\end{equation*}
for some $\varepsilon\in\c(\mathbb V)\rtimes \c \Wh$, which is regular along $\dpr{\alpha^\vee, \xi}=0$. It follows that $\YH^{-\alpha^\vee}=(\YH^{\alpha^\vee})^{-1}$ satisfies
\begin{equation*}
\YH^{-\alpha^\vee}=1-\varepsilon\dpr{\alpha^\vee, \xi}+o(\dpr{\alpha^\vee, \xi})\quad\text{ near}\  \dpr{\alpha^\vee, \xi}=0\,.
\end{equation*} 
Putting this together, one sees that \eqref{3term} is regular at $\dpr{\alpha^\vee, \xi}=0$, which proves part (i).    

{\it Parts (ii), (iii)}. These are proved in the same way as in Proposition \ref{elcl}, using Lemmas \ref{re}, \ref{trans}. 

{\it Part (iv)}. It is sufficient to prove that for $\xi=-\rho_m$, we have $L^{b,\vee}_{c}(\xi, \YH)e=(L^{b}_c+\mathrm{const})e$. Indeed, the statement (iv) is then obtained by passing to the classical limit and by using part (iii).  
We already know that $L^{b,\vee}_{c}(\xi, \YH)=\sum_{\pi\in Wb} Y^\pi+c_0$, where $c_0$ is a function of $\xi$. Therefore, it suffices to prove that for $\xi=-\rho_m$ we have $Y^\pi e=0$ if $\pi$ is not dominant. 
Recall that 
\begin{equation*}
Y^\pi=\RR(\alpha^1)\dots \RR(\alpha^l)\,t(\pi)\,,
\end{equation*} 
where the sequence $\alpha^1, \dots, \alpha^l$ is obtained by going from $x\in C_a$ to $x-\pi$ and lisiting all reflection hyperplanes transversed in that process; in addition, each $\alpha^i$ should be positive at $x$ and negative at $x-\pi$. Using the translation properties \eqref{rtr}, we can rewrite this as 
\begin{equation*}
Y^\pi=t(\pi)\,\RR(\beta^1)\dots \RR(\beta^l)\,,
\end{equation*} 
where the sequence of affine roots $\beta^i$ is obtained similarly by going from $x+\pi$ to $x\in C_a$, listing all transversed hyperplanes. Since $x$ lies inside the Weyl alcove $C_a$, the last root $\beta^l$ should correspond to one of the faces of the alcove; also, we know that $\beta^l(x)<0$. This tells us that $\beta^l=- a_i$ with $0\le i\le n$. Moreover, if we choose $x$ inside $C_a$ sufficiently close to $0$, then we necessarily have $\beta^l=-a_i$ with $i\ne 0$ (otherwise $\pi$ would be dominant). Now, for $\xi=-\rho_m$ we have $\dpr{a_i^\vee, \xi} = -m_{a_i}$, and so $\RR(-a_i)=\sigma_{m_{a_i}}(-a_i)(1-s_i)$ as a result. Therefore, 
$\RR(\beta^l)e=\RR(-a_i)e=0$, implying $Y^\pi e=0$.  \qed


\subsection{}\la{ecc}
Let us proceed to the case of the affine root system $C^\vee C_n$. In the setting of \cite{KH98} this corresponds to the case of a reduced root system $R=BC_n$, but for us, as in Section \ref{cc}, $R$ will be a root system \eqref{cn} of type $C_n$, and $\Ra$ will denote the associated affine root system \eqref{rrelq}, with the following basis of simple roots:
\begin{equation*}
a_0=\delta-2\epsilon_1\,,\qquad a_i=\epsilon_i-\epsilon_{i+1}\quad (i=1,\dots, n-1)\,,\qquad \alpha_n=2\epsilon_n\,,
\end{equation*}
where $\delta\equiv c$ on $V$.
We have the Weyl group $W=\mathfrak{S}_n\ltimes \{\pm 1\}^n$ and the group $\Wh=\WW$ generated by $s_{i}=s_{a_i}$, acting on $V$ in accordance with \eqref{wwact}.
Note that $\Wh\cong W\ltimes \Lambda$ with $\Lambda=\sum_{i=1}^n \Z\epsilon_i$. The dominant cone is $\Lambda_+=\{(\lambda_1, \dots, \lambda_n)\in\Lambda\,|\, \lambda_1\ge\dots \ge \lambda_n\ge 0\}$.

The $R$-matrices are elements of the algebra $\c(V)*\Wh\cong {\D}_q * W$ of reflection-difference operators on $V$.  
They depend on dynamical parameters $\xi\in V$, coupling constants $\mu$, $\nu$, $\ov{\nu}$, $g=(g_i)$, $\ov{g}=(\ov{g}_i)$ ($i=0\dots 3$), and are as follows \cite{KH98}:
\begin{align}\label{cel1}
\RR(\aalpha) &=\sigma_\mu(\aalpha)-\sigma_{\dpr{\alpha^\vee, \xi}}(\aalpha)s_{\aalpha}\,\quad &&\text{for\ }\aalpha=k\delta\pm\epsilon_i\pm\epsilon_j\quad(k\in\Z,\ i\ne j)\,, \\\la{cel2}
\RR(\aalpha) &= v_{\nu, g}(\aalpha/2) -v_{\dpr{\alpha^\vee, \xi}, g}(\aalpha/2)s_{\aalpha}\,\quad &&\text{for\ }\aalpha=2k\delta\pm 2\epsilon_i\quad(k\in\Z)\,,\\\la{cel3}
\RR(\aalpha) &= {v}_{\ov{\nu}, \ov{g}}(\aalpha/2) -{v}_{\dpr{\alpha^\vee, \xi}, \ov{g}}(\aalpha/2)s_{\aalpha}\,\quad &&\text{for\ }\aalpha=(2k+1)\delta\pm 2\epsilon_i\quad(k\in\Z)\,.
\end{align}
In these formulas, $v_{\nu, g}(z)=v_\nu(z; g_0, g_1, g_2, g_3)$ is the function \eqref{vmu}. 

According to \cite[Theorems 4.1, 4.2]{KH98}, the elements $\RR(\aalpha)$ satisfy the affine Yang--Baxter relations. We can now define the elements $\RR_w$ and $Y^b=R_{t(b)}t(b)$, $b\in \Lambda$ in the same way as in \ref{rwdef} (note that the group $\Omega$ is trivial in this case), and Theorem \ref{yb} remains valid in this setting. 

Let us write down the expressions for $Y^{\pm\epsilon_i}$. We have a reduced decomposition (cf. \cite[(3.6)]{St1}) 
\begin{equation*}
t(\epsilon_i)=s_i\dots s_{n-1}s_{n}s_{n-1}\dots s_1s_0s_1\dots s_{i-1}\,,\quad 1\le i\le n\,.
\end{equation*} 
Applying Definition \ref{rwdef}, we calculate
\begin{align}
Y^{\epsilon_i}=&\RR (\epsilon_i-\epsilon_{i+1})\RR (\epsilon_i-\epsilon_{i+2})\dots \RR (\epsilon_i-\epsilon_n)\RR (2\epsilon_i)\nonumber\\
&\times\RR (\epsilon_i+\epsilon_n)\dots \RR (\epsilon_i+\epsilon_{i+1})\RR(\epsilon_i+\epsilon_{i-1})\dots \RR(\epsilon_i+\epsilon_1) \la{yicc}\\
&\times\RR (\delta+2\epsilon_i)t(\epsilon_i)\RR(\epsilon_i-\epsilon_1)\dots \RR(\epsilon_i-\epsilon_{i-1})\,.\nonumber
\end{align}

\begin{remark}
In \cite{KH98}, a slightly different affine root system is considered, with the roots in \eqref{cel2} replaced by $k\delta\pm \epsilon_i$, but the corresponding $R$-matrices and the resulting elements $\RR_w$ are the same.
\end{remark}

For $i=1, \dots, n$ define $m_i$ by $m_1=\dots =m_{n-1}=\mu$, $m_n=\nu$. Let $\xi=\xi_0$ be a solution to the system of equations
\begin{equation}\la{csystem}
\dpr{a_i^\vee, \xi_0}=-{m}_i\,,\qquad i=1, \dots, n\,.
\end{equation} 
Explicitly, we have $\xi_0=(\xi_1,\dots, \xi_n)$ with $\xi_i=-\nu-(n-i)\mu$.

\begin{theorem}[\cite{KH98}, Theorems 4.5\,\& 6.5]\la{cemr} Let $\xi=\xi_0$ as above.

(1) Given $b\in \Lambda_+$, let $L^b\in \D_q$ be the unique difference operator such that $Y^b e= L^b e$. Then each $L^b$ is $W$-invariant, 
and the difference operators $L^b$, $b\in \Lambda_+$ form a commutative family.

(2) Let $b=\epsilon_1$. Then
\begin{align}
\la{clb1}
L^{b}&=\sum_{\pi\in Wb} \, (A_\pi t(\pi)-B_\pi)\,,\qquad A_\pi=v_{\nu, g}(\pi){v}_{\ov{\nu}, \ov{g}}(\pi+\delta/2)\,
\prod_{\genfrac{}{}{0pt}{}{\alpha\in R}{\dpr{\pi, \alpha}=1}}\, \sigma_{\mu}(\alpha)\,,
\\
\la{clb2}
B_\pi&=v_{\nu, g}(\pi)
{v}_{-\nu-(n-1)\mu, \ov{g}}
(\pi+\delta/2)
\prod_{\genfrac{}{}{0pt}{}{\alpha\in R}{\dpr{\pi, \alpha}=1}}\, \sigma_{\mu}(\alpha)\,.
\end{align}
\end{theorem}
The operator \eqref{clb1}--\eqref{clb2} contains $11$ parameters $\mu, \nu, \ov{\nu}, g_i, \ov{g}_i$, but multiplying all $g_i$ (or all $\ov{g}_i$) by a constant simply rescales the Hamiltonian. Thus, effectively we have $9$ coupling parameters. This Hamiltonian was first introduced by van Diejen \cite{vD1}, in a different form and under an additional constraint on $\nu, \ov{\nu}$. For general coupling parameters it was introduced by Komori and Hikami in \cite{KH97}, where higher quantum Hamiltonians were also constructed. See \cite[(4.21)]{KH97} for an alternative presentation of $L^b$ which links it to \cite{vD1}. The classical Hamiltonian $L^b_c$ looks as follows:
\begin{align}
\la{cllb1}
L^{b}_c&=\sum_{\pi\in Wb} \, (A_\pi e^{\beta p_\pi}-B_\pi)\,,\qquad A_\pi=v_{\nu, g}(\pi){v}_{\ov{\nu}, \ov{g}}(\pi)\,
\prod_{\genfrac{}{}{0pt}{}{\alpha\in R}{\dpr{\pi, \alpha}=1}}\, \sigma_{\mu}(\alpha)\,,
\\
\la{cllb2}
B_\pi&=v_{\nu, g}(\pi)
{v}_{-\nu-(n-1)\mu, \ov{g}}
(\pi)
\prod_{\genfrac{}{}{0pt}{}{\alpha\in R}{\dpr{\pi, \alpha}=1}}\, \sigma_{\mu}(\alpha)\,.
\end{align}

\subsection{}
Before proceeding to constructing a quantum Lax pair, we need to renormalise the $R$-matrices. The $R$-matrices \eqref{cel1} have the property \eqref{uni}, 
so we can define $\RH(\aalpha)$ by \eqref{rh}, with $m_\alpha=\mu$. For the $R$-matrices \eqref{cel2}--\eqref{cel3} the procedure is more subtle. 
First, considering \eqref{cel2}, we obtain:
\begin{equation}\la{cuni}
R(\aalpha)R(-\aalpha)=v_{\nu, g}(\aalpha/2)v_{\nu, g}(-\aalpha/2)+v_{\dpr{\alpha^\vee, \xi}, g}(\aalpha/2)v_{-\dpr{\alpha^\vee, \xi}, g}(\aalpha/2)\,.
\end{equation}  
Using \eqref{vv}, this can be rewritten as (cf. \cite[(4.5)]{KH98}) 
\begin{equation*}
R(\aalpha)R(-\aalpha)=\sum_{r=0}^3(g_r^\vee)^2(\wp(\nu+\omega_r)-\wp(\dpr{\alpha^\vee, \xi}+\omega_r))\,, 
\end{equation*}
which is independent of $x$.

\begin{defi}\la{dpa} Let $v_{\nu, g}(z)$ be the function \eqref{vmu} with parameters $\nu$, $g=(g_0, g_1, g_2, g_3)$. The dual parameters $\nu^\vee$, $g^\vee=(g^\vee_0, g^\vee_1, g^\vee_2, g^\vee_3)$ are defined by \eqref{ghat} and the condition $v_{\nu, g}(\nu^\vee)=0$. 
\end{defi}

\begin{remark} \la{zpar} The function $v_{\nu, g}(z)$ can be parametrised by its zeros. Following \cite{KH97}, let  $v_{\nu, g}(z)=A\prod_{r=0}^3\sigma^r_{\nu_r}(z)$. Then we have $2\nu=\sum_{r=0}^3 \nu_r$ and $g_r$ can be expressed in terms of $A$ and $\nu_r$, see \cite[Lemma 4.5]{KH97}. If such a parametrisation is used, then $\nu^\vee$ can be taken simply as $\nu_0$ (or any of $\nu_r+\omega_r$, $r=0\dots 3$). However, the parametrisation of $v_{\nu, g}$ by its zeros is inconvenient for writing the $R$-matrices. 
\end{remark}

\begin{lemma}\la{bb}
Let $\aalpha\in \Ra$ and $R(\aalpha)$ be as in \eqref{cel2}. 
Then $\RH(\aalpha):=v_{\nu^\vee, g^\vee}(\dpr{\alpha^\vee, \xi})^{-1}R(\aalpha)$ satisfies $\RH(\aalpha)\RH(-\aalpha)=1$. We have $\RH(\aalpha)=-e^{\pi i (\aalpha  - 2\nu^\vee)\beta_r} s_{\aalpha}$ when $\dpr{\alpha^\vee, \xi}=\omega_r$, $r=0\dots 3$. Here $\omega_r$ are the half-periods, and $(\beta_0, \beta_1, \beta_2, \beta_3)=(0,0,1,1)$.  Similarly, for the case \eqref{cel3} we define $\RH(\aalpha):=v_{\ov{\nu}^\vee, \ov{g}^\vee}(\dpr{\alpha^\vee, \xi})^{-1}R(\aalpha)$. Then  
$\RH(\aalpha)=-e^{\pi i (\aalpha  - 2\ov{\nu}^\vee)\beta_r} s_{\aalpha}$ when $\dpr{\alpha^\vee, \xi}=\omega_r$.  
\end{lemma}
\begin{proof}
Since the expression \eqref{cuni} is independent of $x$, we may assume $\aalpha/2=\nu^\vee$. Then 
\begin{equation*}
R(\aalpha)R(-\aalpha)=v_{\dpr{\alpha^\vee, \xi}, g}(\nu^\vee)v_{-\dpr{\alpha^\vee, \xi}, g}(\nu^\vee)=
v_{\nu^\vee, g^\vee}(\dpr{\alpha^\vee, \xi})v_{\nu^\vee, g^\vee}(-\dpr{\alpha^\vee, \xi})\,,
\end{equation*} 
by \eqref{vsym}. Hence, $\RH(\aalpha)\RH(-\aalpha)=1$. 

From the definition of $R(\aalpha)$ we see that it has a first order pole at the hyperplane $\dpr{\alpha^\vee, \xi}=0$, with the residue equal to $-1/2(\sum_{r=0}^3 g_r) s_{\aalpha}$. On the other hand,  $v_{\nu^\vee, g^\vee}(\dpr{\alpha^\vee, \xi})$ also has a first order pole at $\dpr{\alpha^\vee, \xi}=0$, with the residue equal to $g_0^\vee$, which is $1/2(\sum_{r=0}^3 g_r)$ from \eqref{ghat}. This implies that $\RH(\aalpha)$ tends to $-s_{\aalpha}$ as $\dpr{\alpha^\vee, \xi}$ approaches zero.  This proves the claim for $\omega_r=0$; for other half-periods $\omega_r$ proof is similar.  
\end{proof}
Motivated by the above, for given $\nu, g$ and $\ov{\nu}, \ov{g}$ we choose (and fix) $\nu^\vee, \ov{\nu}^\vee$ and define \emph{unitary $R$-matrices} by 
\begin{equation*}
\RH(\aalpha)=\begin{cases}
\sigma_{\mu}(\dpr{\alpha^\vee, \xi})^{-1}R(\aalpha)&\quad\text{for\ }\aalpha=k\delta\pm\epsilon_i\pm\epsilon_j\quad(k\in\Z,\ i\ne j)\,,\\
\RH(\aalpha)=v_{\nu^\vee, g^\vee}(\dpr{\alpha^\vee, \xi})^{-1}R(\aalpha)&\quad \text{for\ }\aalpha=2k\delta\pm 2\epsilon_i\quad(k\in\Z)\,,\\
\RH(\aalpha)=v_{\ov{\nu}^\vee, \ov{g}^\vee}(\dpr{\alpha^\vee, \xi})^{-1}R(\aalpha)&\quad \text{for\ }\aalpha=(2k+1)\delta\pm 2\epsilon_i\quad(k\in\Z)\,.
\end{cases}
\end{equation*} 
Then we have $\RH(\aalpha)\RH(-\aalpha)=1$ for all $\aalpha\in\Ra$.

\subsection{}

Now let us recall the notation of \ref{regu}, by which we view $R(\aalpha)$ and $\RH(\aalpha)$ as elements of $\c(\mathbb V)\rtimes \c\mathbb W$, where $\mathbb V=V^\vee\times V$ incorporates the dynamical variables. Both $\RR(\aalpha)$ and $\RH(\aalpha)$ satisfy the affine Yang--Baxter relations \cite[(3.1)(a)-(c)]{KH98}. This allows us to introduce
\begin{equation}\la{cthi}
\TH_i=\RH(a_i)(s_i^\vee\otimes s_i)\,,\quad i=0, \dots, n\,.
\end{equation}
Then one checks that the relations \cite[(3.1)(a)-(c)]{KH98} imply that $\TH_i$ satisfy the relations \eqref{b1}--\eqref{b3}, while the unitarity of $\RH(\aalpha)$ implies that $\TH_i^2=1$ for all $i$. Therefore, we have the following result, analogous to Propositions \ref{ass}, \ref{prophat}. 
\begin{prop}
(1) The assignment 
$s_i\mapsto \TH_i$ ($i=0,\dots, n$)
extends to a group homomorphism $w\mapsto \TH_w$, $w\in\Wh$. In particular, the elements $\TH_{t(b)}$, $b\in \Lambda$ pairwise commute. 

(2) For any reduced decomposition of $w$ into $w=s_{i_1}\dots s_{i_l}$ we have $\TH_w=\RH_w (w^\vee\otimes w)$, where 
\begin{equation*}
\RH_w=\RH(\alpha^{1})\dots \RH(\alpha^{l})\,,\qquad \alpha^{1}=a_{i_1},\ \alpha^{2}=s_{i_1}(a_{i_2})\,,\ \dots,\ \alpha^{l}=s_{i_1}\dots s_{i_{l-1}}(a_{i_l})\,. 
\end{equation*}
As a result, $\RH_w$ does not depend on the choice of a decomposition. For $b\in \Lambda$ we have $\TH_{t(b)}=\RH_{t(b)}(1\otimes t(b))=\YH^b$. This implies the commutativity of $\YH^b$ and, therefore, of $Y^b:=R_{t(b)}t(b)$.   
\end{prop}

We also have an analogue of Proposition \ref{elclq}. Let $b=\epsilon_1$ and $L^b_c$ is the classical van Diejen Hamiltonian \eqref{cllb1}--\eqref{cllb2}. By $L^{b, \vee}_c$ we denote the classical operator with the dual coupling parameters $\nu^\vee, \ov{\nu}^\vee$, $g^\vee$, $\ov{g}^\vee$ (and with $\mu^\vee=\mu$), see Definition \ref{dpa}. 

\begin{prop}\label{celclq} 
Let $L^{b,\vee}_{c}(\xi, \YH)$ and $L^{b,\vee}_{c}(\xi, \YH_c)$  be the result of substituting the dynamical variables and Cherednik operators into the dual classical Hamiltonian. Then we have:

(i) $L^{b,\vee}_{c}(\xi, \YH)$, viewed as an element of $\D(V)*W$ depending on $\xi$, is regular for $\xi$ near $\xi=0$;

(ii)  $L^{b,\vee}_{c}(\xi, \YH_{c})$, viewed as an element of $\c(V\times V)*W$ depending on $\xi$, is regular for all $\xi\in V$;

(iii) $L^{b,\vee}_{c}(\xi, \YH_{c})$ is constant in $\xi$. Moreover, expanding $L^{b,\vee}_{c}(\xi, \YH_{c})$ as $\sum_{w\in W} a_w w$ with $a_w\in\c(V\times V)$, we have $a_w=0$ for $w\ne \id$.

(iv) We have 
$L^{b,\vee}_{c}(\xi, \YH_c)=L^{b}_c+\mathrm{const}$.
\end{prop}   
This is proved in the same way as Proposition \ref{elclq}. One additional complication arises when considering what happens for $\xi_i=\omega_r$ (this is further explained in a more general situation in the next subsection).  
Here is a suitable generalisation of Lemma \ref{lemmay}, which is proved by a similar argument.

\begin{lemma}\la{lemmayc}
For $\alpha=2\epsilon_l$ we have $\YH^{\alpha^\vee}=e^{-\pi c\beta_r-\lambda_r}$ if $\dpr{\alpha^\vee, \xi}=\omega_r$, where 
\begin{equation}\la{lam}
\lambda_r=2\pi i\beta_r(\nu^\vee+\ov{\nu}^\vee+(n-1)\mu)\,,\quad r=0\dots 3\,,
\end{equation} 
and $\beta_r$ are the same as in Lemma \ref{bb}. In particular, for the classical operator $\YH^{\alpha^\vee}_{c}$ we have $\YH^{\alpha^\vee}_c=e^{-\lambda_r}$ if $\dpr{\alpha^\vee, \xi}=\omega_r$. 
\end{lemma}      

The same arguments as in Section \ref{regu1} lead to a construction of a Lax pair.
\begin{theorem}\la{lpc} For any fundamental coweight $\lambda$
with the stabiliser $W'$, there exists a quantum Lax pair $\mathcal L, \mathcal A$ of size $|W|/|W'|$ satisfying the Lax equation \eqref{qleq} with $\mathcal H=\widehat H\mathbb{1}$, $\widehat H=L^{\epsilon_1}$. This Lax pair depends on a spectral parameter and admits a classical limit.   
\end{theorem}
The smallest Lax pair of size $2n$ is obtained for $\lambda=\epsilon_1$. We calculate the corresponding Lax matrix in \ref{calvd}.  

\subsection{}\la{rai}
Let us generalise Proposition \ref{celclq} to any of the higher Hamiltonians of the van Diejen system. A direct proof is problematic since we do not know an explicit formula for these Hamiltonians. Instead, we will use a result of Rains, who in \cite{Rains} developed a geometric approach to elliptic DAHAs. To formulate his result we will need some notation.  
Let us introduce 
\begin{equation*}
v_{\aalpha}=\begin{cases}
\sigma_{\mu}(\aalpha)&\quad\text{for\ }\aalpha=k\delta\pm\epsilon_i\pm\epsilon_j\quad(k\in\Z,\ i\ne j)\,,\\
v_{\nu, g}(\aalpha/2)&\quad \text{for\ }\aalpha=2k\delta\pm 2\epsilon_i\quad(k\in\Z)\,,\\
v_{\ov{\nu}, \ov{g}}(\aalpha/2)&\quad \text{for\ }\aalpha=(2k+1)\delta\pm 2\epsilon_i\quad(k\in\Z)\,,
\end{cases}
\end{equation*} 
as well as $\widehat{G}_\pi$, $\pi\in\Lambda$:
\begin{equation*}
\widehat{G}_\pi=\prod_{\genfrac{}{}{0pt}{}{\alpha\in R:\, \dpr{\pi, \alpha}>0}{0\le k\le \dpr{\pi, \alpha}}}\, v_{\alpha+k\delta}\qquad (\widehat{G}_0:=1)\,.
\end{equation*}
The coupling parameters $\mu, \nu, \ov{\nu}, g_i, \ov{g}_i$ will be assumed generic. 
Next, let 
\begin{equation*}
\Pi=\{\lambda=(\lambda_1,\dots, \lambda_n) \,|\, \lambda_i\in\{-1, 0, 1\}\}\,. 
\end{equation*}  
For any $\alpha\in R$, define an $\alpha^\vee$-{\it string} in $\Pi$ as $\{\pi+\Z\alpha^\vee\}\cap \Pi$ where $\pi\in \Pi$; the number of the lattice points on a string will be called its {\it length}. It is easy to see that $\alpha^\vee$-strings in $\Pi$ are of one of the following types:
\begin{align}
\text{length-one:}\quad &\{\pi\}\,,\quad\text{with}\ s_\alpha\pi=\pi\,,\la{st1}\\
\text{length-two:}\quad &\{\pi, \pi'\}\,,\quad\text{with}\ s_\alpha\pi=\pi'\,,\quad \pi'-\pi=\pm \alpha^\vee\,,\la{st2}\\
\text{length-three:}\quad &\{\pi, \pi\pm\alpha^\vee\}\,,\quad\text{with}\ s_\alpha\pi=\pi\,.\la{st3}
\end{align}  

\begin{defi}\la{calv} For given generic $\mu$, $\nu_r, \ov{\nu}_r$, let $\mathcal V$ be the vector space of all difference operators $L\in \D_q(V)$ of the form
\begin{equation}\la{fop}
L=\sum_{\pi\in \Pi}a_\pi t(\pi)
\,, \qquad 
a_\pi\in\c(V)\,,
\end{equation}
satisfying the following conditions: 
(1) $L$ is $W$-invariant;
(2) $a_\pi/{\widehat{G}}_\pi$ is elliptic {w.r.t.} the lattice $\Lambda+\tau\Lambda$;
(3) $a_\pi$ and $a_\pi/\widehat{G}_\pi$ have at most simple poles along the hyperplanes $\aalpha=l+m\tau$ with $\aalpha\in \Ra$ and $l, m\in\Z$, and no other singularities.
\end{defi}


In addition, we are going to impose certain 
``residue conditions'' on the coefficients $a_\pi$ for each $\alpha^\vee$-string in $\Pi$. In what follows, we call a function $f\in \c(V)$ {\it $\alpha$-regular} if it has no singularities along hyperplanes $\dpr{\alpha, x}=\mathrm{const}$. 
For the length-one strings \eqref{st1} the residue condition is simply that 
\begin{equation}\la{1res}
\text{$a_\pi$ is $\alpha$-regular.}
\end{equation}
For the length-two strings \eqref{st2} the conditions are that
\begin{equation}\la{2res}
\text{$\theta(\alpha)a_\pi$ and $\theta(\alpha)a_{\pi'}$ are $\alpha$-regular.}
\end{equation}
Here, as before, we view the roots as affine-linear functions, e.g., $\theta(\alpha)=\theta(
\dpr{\alpha, x})$ for $\alpha\in R$. 
For the length-three strings \eqref{st3}, the conditions are more involved:  
\begin{gather}\la{3res}
\text{$\theta(\alpha)\theta(\alpha+\delta)a_{\pi+\alpha^\vee}$, $\theta(\alpha+\delta)\theta(\alpha-\delta)a_{\pi} $ and 
$\theta(\alpha)\theta(\alpha-\delta)a_{\pi-\alpha^\vee} $ are $\alpha$-regular;}
\\
\la{4res}
\text{$a_{\pi+\alpha^\vee}+a_{\pi}+a_{\pi-\alpha^\vee}$ is regular for $\dpr{\alpha, x}=0, \pm c$.}
\end{gather}
Additionally, for $\alpha=2\epsilon_l$ 
we require that
\begin{equation}\la{5res}
\text{$e^{-\lambda_r}a_{\pi+\alpha^\vee}+a_{\pi}+e^{\lambda_r}a_{\pi-\alpha^\vee}$ is regular for $\dpr{\alpha, x}=2\omega_r, 2\omega_r\pm c$,}
\end{equation}
where $\lambda_r=2\pi i\beta_r(\nu+\ov{\nu}+(n-1)\mu)$, cf. \eqref{lam}. See \cite{RR} where similar residue conditions were considered in the rank-one case.



The following result can be extracted from \cite{Rains}.

\begin{theorem}[cf. Theorem 7.24 of \cite{Rains}]\la{rath} For each $b_k=\epsilon_1+\dots +\epsilon_k$ ($0\le k\le n$), there exists an operator $L_k\in \mathcal V$ with the leading terms $\sum_{\pi\in Wb_k} \widehat{G}_\pi t(\pi)$, whose coefficients satisfy the residue conditions \eqref{1res}--\eqref{5res} for any $\alpha^\vee$-string, $\alpha\in R$. The operators $L_k$ pairwise commute and admit a classical limit.
\end{theorem}

\begin{remark} The Hamiltonian $L_0$ is trivial, $L_0=1$. It is expected that $L_k$ are the Hamiltonians of the van Diejen system, but the construction below makes no use of this or the commutativity of $L_k$. Note that for the first van Diejen Hamiltonian $L^{\epsilon_1}$ \eqref{clb1}--\eqref{clb2} the residue conditions can be easily checked from the explicit formula.
\end{remark}

\begin{remark}\la{clres}
In the classical limit, the residue conditions for $\alpha^\vee$-strings of length one or two remain the same. For a length-three string, \eqref{3res}--\eqref{4res} are replaced with their $c=0$ limit:
 \begin{gather}\la{3resc}
\text{$a_{\pi+\alpha^\vee} \theta^2(\alpha)$, $a_{\pi} \theta^2(\alpha)$ and 
$a_{\pi-\alpha^\vee} \theta^2(\alpha)$ are $\alpha$-regular;}
\\
\la{4resc}
\text{$a_{\pi+\alpha^\vee}+a_{\pi}+a_{\pi-\alpha^\vee}$ is regular for $\dpr{\alpha, x}=0$.}
\end{gather}
Similarly, for $\alpha=2\epsilon_l$ and $r=1,2,3$ the conditions \eqref{5res} are replaced with:
\begin{equation}\la{5resc}
\text{$e^{-\lambda_r}a_{\pi+\alpha^\vee}+a_{\pi}+e^{\lambda_r}a_{\pi-\alpha^\vee}$ is regular for $\dpr{\alpha, x}=2\omega_r$.}
\end{equation}
This tells us that the second-order poles in this sum must cancel. Also, from $\dpr{\alpha, \pi}=0$ it can be checked that $\widehat G_\pi$ is periodic with respect to translations by $\alpha^\vee$ and $\tau\alpha^\vee$, and so must be $a_\pi$, by the definition of $\mathcal V$. Together with the $s_\alpha$-invariance of $a_\pi$, this gives that $a_\pi$ has zero residue at $\dpr{\alpha, x}=2\omega_r$. As a result,  $e^{-\lambda_r}a_{\pi+\alpha^\vee}+e^{\lambda_r}a_{\pi-\alpha^\vee}$ also have zero residue at $\dpr{\alpha, x}=2\omega_r$ (cf. the properties \eqref{aab}).
\end{remark}

We now have the following analogue of Proposition \eqref{celclq}.
\begin{prop}\label{celclqh} 
Consider the classical limits of the operators $L_k$ from Theorem \ref{rath}, and with the dual parameters $\nu^\vee, g^\vee$, $\ov{\nu}^\vee, \ov{g}^\vee$. Denote these classical Hamiltonians as $L_k^\vee$, $k=0, \dots, n$, with $L_0^\vee=1$.    
Let 
$L^{\vee}_{k}(\xi, \YH)$ and $L^{\vee}_{k}(\xi, \YH_c)$  denote the result of substituting the dynamical variables and Cherednik operators into $L_k^\vee$. Then we have:

(i) $L^{\vee}_{k}(\xi, \YH)$, viewed as an element of $\D(V)*W$ depending on $\xi$, is regular for $\xi$ near $\xi=0$;

(ii)  $L^{\vee}_{k}(\xi, \YH_{c})$, viewed as an element of $\c(V\times V)*W$ depending on $\xi$, is regular for all $\xi\in V$;

(iii) $L^{\vee}_{k}(\xi, \YH_{c})$ is constant in $\xi$. Moreover, expanding $L^{\vee}_{k}(\xi, \YH_{c})$ as $\sum_{w\in W} a_w w$, we have $a_w=0$ for $w\ne \id$.

(iv) We have 
$\displaystyle{L^{\vee}_{k}(\xi, \YH_c)=L^{b_k}_c+\sum_{0\le l<k} a_{kl} L^{b_l}_c}$, with some $a_{kl}\in\c$, where $L^{b_l}_c$ are the higher (classical) van Diejen Hamiltonians for $b_l=\epsilon_1+\dots +\epsilon_l$.
\end{prop}   
This is proved similarly to Proposition \ref{elclq}. Namely, the regularity of $L^{\vee}_{k}(\xi, \YH_{c})$ for $\xi$ close to $0$, as well as for $\xi_i=\omega_r$, follows from the residue conditions on the coefficients of $L_k^\vee$, as specified in Remark \ref{clres} (used together with Lemmas \ref{lemmay}, \ref{lemmayc}). The global regularity then follows from an analogue of Lemma \ref{trans}; to prove such an analogue we do not need an explicit formula for $L_k^\vee$, but make use of the Definition \ref{calv} instead. We leave the details to the reader. \qed 
 
As a consequence, we can perform a construction of a Lax pair for each of the higher van Diejen Hamiltonians, and so Theorem \ref{lpc} remains valid for any of $L^b$ with $b=b_k$, $k=1, \dots, n$.

\subsection{}\la{calvd}
Let us calculate a quantum Lax matrix corresponding to $Y_1=Y^{\epsilon_1}$. The calculation and the notation will be very similar to those in \ref{clax}. 
The stabiliser $W'$ of $\lambda=\epsilon_1$ is the subgroup of signed permutations of $x_2, \dots, x_n$. The dynamical variable $\xi$ needs to satisfy the conditions $\dpr{a_i^\vee, \xi}=-m_i$ for $i=2, \dots, n$, cf. \eqref{csystem}, that is,
\begin{equation}\la{spec} 
\xi_1=\eta\,,\quad \xi_i=-\nu-(n-i)\mu \quad (i=2, \dots, n)\,,
\end{equation}
with $\eta$ being a spectral parameter. 
Let us abbreviate $\RR(\epsilon_i-\epsilon_j)$ and $\RR(\epsilon_i+\epsilon_j)$ to $\RR_{ij}$ and $\RR_{ij}^+$, respectively. Using \eqref{yicc}, we get  
\begin{gather*}
Y_1=\RR_{12}\RR_{13}\dots \RR_{1n}\,\RR(2\epsilon_1) \RR_{1n}^+\dots \RR_{12}^+\, \RR(\delta+2\epsilon_1) \,t(\epsilon_1)\,,\quad\text{with}\\
\RR_{ij}=\sigma_{{\mu}}(x_{ij})-\sigma_{\xi_{ij}}(x_{ij})\ss_{ij}\,,\qquad\RR_{ij}^+=\sigma_{{\mu}}(x^+_{ij})-\sigma_{\xi^+_{ij}}(x^+_{ij})\ss^+_{ij}\,,
\\
\RR(2\epsilon_1)=v_{\nu, g}(x_1)-v_{\xi_1, g}(x_1){\ss}_1\,, \qquad \RR(\delta+2\epsilon_1)\,t(\epsilon_1)={v}_{\ov{\nu}, \ov{g}}(x_1+c/2)\,t(\epsilon_1)- {v}_{\xi_1, \ov{g}}(x_1+c/2){\ss}_1\,.
\end{gather*}
Introducing 
$\mathcal R=\RR_{12}\dots \RR_{1n}$ and $\mathcal R^+=\RR_{1n}^+\dots \RR_{12}^+$, we have the following.

\begin{lemma}[cf. \cite{KH97}, Lemma 4.4] 
$\mathcal R$ and $\mathcal R^+$ preserve the subspace $M''=e''M$, where $e''=\frac{1}{n!}\sum_{{w} \in \mathfrak{S}_{n-1}} {w}$. Their restriction onto $M''$ is calculated as follows: 
\begin{gather*}
\mathcal R |_{M''}=U-\sum_{i\ne 1}^n V_i{\ss}_{1i}\,,\qquad U=\prod_{l\ne 1}^n \sigma_\mu(x_{1l})\,,\quad V_i=\sigma_{\xi_{12}}(x_{1i})\prod_{l\ne 1, i} \sigma_\mu(x_{il})\,,
\\
\mathcal R^+ |_{M''}=U^+-\sum_{i\ne 1}^n V_i^+{\ss}_{1i}^+\,, \qquad
U^+=\prod_{l\ne 1}^n \sigma_\mu(x_{1l}^+)\,,\quad V_i^+=\sigma_{\xi_{1n}^+}(x_{1i}^+)\prod_{l\ne 1, i} \sigma_\mu(x_{li})\,.
\end{gather*}
\end{lemma}
\begin{proof}
Note that the dynamical variables satisfy $\xi_{i+1}-\xi_i=-\mu$ for $1<i<n$, so the statement about $\mathcal R$ is known already from Lemmas \ref{inv12}, \ref{nsel}. The result for $\mathcal R^+$ follows by observing that it can be obtained as $(\omega^\vee\otimes\omega)\mathcal R (\omega^\vee\otimes\omega)^{-1}$, where $\omega\in W'$ is the transformation \eqref{omeg}. Note that under this transformation, the dynamical variables change to $(\xi_1, -\xi_n, \dots, -\xi_2)$, so $\xi_{12}=\xi_1-\xi_2$ becomes $\xi_{1n}^+=\xi_{1}+\xi_n$.   
\end{proof}

Next, we restrict $Y_1$ further onto $M'=e'M$, as in Section \ref{clax}. From the above,
\begin{equation}\la{esplit} 
Y_1|_{M'}=(U-\sum_{i\ne 1}^n V_i{\ss}_{1i})\RR(2\epsilon_1)(U^+-\sum_{i\ne 1}^n V_i^+{\ss}_{1i}^+)\RR(\delta+2\epsilon_1)\,t(\epsilon_1)|_{M'}\,.
\end{equation}
The main step is to work out the restriction onto $M'$ of the product 
\begin{equation}\la{esplit1} 
(U-\sum_{i\ne 1}^n V_i{\ss}_{1i})
(v_{\nu, g}(x_1)-v_{\xi_1, g}(x_1){\ss}_1)(U^+-\sum_{i\ne 1}^n V_i^+{\ss}_{1i}^+)\,.
\end{equation}

\medskip
\begin{lemma}
The operator \eqref{esplit1} 
preserves the subspace $M'=e'M$, and its restriction onto $M'$ is given by $A+B{\ss}_1-\sum_{i\ne 1} (C_i{\ss}_{1i}+D_i{\ss}_{1i}^+)$
with 
 \begin{gather}
A=v_{\nu, g}(x_1)\prod_{l\ne 1}^n \sigma_\mu(x_{1l})\sigma_\mu(x_{1l}^+)\,,\qquad B=\alpha v_{\eta, g}(x_1)+\beta v_{\nu, g}(-x_1)\,,
\\
C_i=v_{\nu, g}(x_i)\sigma_{\xi_{12}}(x_{1i}) \sigma_\mu(x_{1i}^+)  \prod_{l\ne 1, i} \sigma_\mu(x_{il}) \sigma_\mu(x_{il}^+)
\,,
\\
D_i=(C_i)^{\ss_i}=v_{\nu, g}(-x_i)\sigma_{\xi_{12}}(x_{1i}^+) \sigma_\mu(x_{1i})  \prod_{l\ne 1, i} \sigma_\mu(-x_{li}^+) \sigma_\mu(x_{li})
\,, 
\intertext{where} 
\la{alpc}
\alpha=\left\{-1 + \sum_{i=1}^{n-1}
\frac
{ 
\sigma_{\xi_{12}}(-i/n)\sigma_{\eta+\nu}(i/n)} 
{
\sigma_\mu(i/n)^2}
\right\}
\prod_{l=1}^{n-1}\sigma_\mu(l/n)^2\,,
\\ \la{bet}
\beta=\sum_{i\ne 1} 
\sigma_{\xi_{12}}(x_{1i}) \sigma_{\eta+\nu}(x_{1i}^+)\prod_{l\ne 1, i} \sigma_\mu(x_{il})\sigma_\mu(x_{l1})
\,.
\end{gather}
In these formulas $\xi_1=\eta$ is a spectral parameter, $\xi_{12}=\eta+\nu+(n-2)\mu$. 
\end{lemma}


\begin{proof} Let us write $v_{\nu}(z)$ instead of $v_{\nu, g}(z)$ to simplify the notation. 
The coefficients $A$, $C_i$ and $D_i$ are calculated in the same way as in Proposition \ref{abcd}. 
For calculating $B$ we expand the product \eqref{esplit1} and collect the terms that reduce to $\ss_1$ when restricted on $M'$. 
It is easy to check that this happens only for the following choices of the terms in each of the factors: (1) $U$, $v_{\xi_1}(x_1)\ss_1$, $U^+$; (2)  $V_i\ss_{1i}$, $v_{\nu}(x_1)$, $V_i^+\ss_{1i}^+$; (3) $V_i\ss_{1i}$, $v_{\xi_1}(x_1)\ss_1$, $V_i^+\ss_{1i}^+$. As a result, $B$ is calculated from 
\begin{equation*}
B=-v_{\xi_1}(x_1)U(U^+)^{\ss_1}+\sum_{i\ne 1} v_\nu(x_i)V_i(V_i^+)^{\ss_{1i}}-\sum_{i\ne 1} v_{\xi_1}(x_i)V_i(V_i^+)^{\ss_{1i}\ss_1}\,,
\end{equation*}
resulting in
\begin{multline*}
B=-v_{\xi_1}(x_1)\prod_{l\ne 1} \sigma_\mu(x_{1l})\sigma_\mu(x_{l1})
\\
+\sum_{i\ne 1} 
\sigma_{\xi_{12}}(x_{1i}) 
\left\{v_\nu(x_i)
\sigma_{\xi_{1n}^+}(x_{1i}^+)
-v_{\xi_1}(x_i)
\sigma_{\xi_{1n}^+}(x_{1i})\right\} \prod_{l\ne 1, i} \sigma_\mu(x_{il})\sigma_\mu(x_{l1})\,.
\end{multline*}
Here $\xi_{12}=\eta+\nu+(n-2)\mu$, $\xi_{1n}^+=\eta-\nu$ by \eqref{spec}. We now use the following identity \cite[(2.8a)]{KH97} which follows from the addition formulas for $\sigma_\mu(z)$:
\begin{equation*}
v_{\nu}(x_i)\sigma_{\eta-\nu}(x_{1i}^+)-v_{\eta}(x_i)\sigma_{\eta-\nu}(x_{1i})=v_{\nu}(-x_1)\sigma_{\eta+\nu}(x_{1i}^+)+v_{\eta}(x_1)\sigma_{\eta+\nu}(x_{i1})\,.
\end{equation*}
Using this in the previous formula leads to $B=\alpha v_{\eta}(x_1)+\beta v_{\nu}(-x_1)$, with  
\begin{equation}
\la{alpc1}
\alpha=-\prod_{l\ne 1} \sigma_\mu(x_{1l})\sigma_\mu(x_{l1}) + \sum_{i\ne 1} 
\sigma_{\xi_{12}}(x_{1i})\sigma_{\eta+\nu}(x_{i1})\prod_{l\ne 1, i} \sigma_\mu(x_{il})\sigma_\mu(x_{l1})\,,
\end{equation}
and with $\beta$ as in \eqref{bet}.
To see why $\alpha$ is, in fact, constant in $x$, we use symmetry arguments. Indeed, the operator given in the lemma must commute with the action of any $w\in W'$. Therefore, the coefficient $B$ must be invariant under signed permutations of $x_2, \dots, x_n$. Now view $B$ as a function of the parameters $\eta, \nu, \mu$. Then the residue of $B$ at $\nu=0$ is given by
\begin{equation*}
\beta|_{\nu=0}=\sum_{i\ne 1} 
\sigma_{\eta+(n-2)\mu}(x_{1i}) \sigma_{\eta}(x_{1i}^+)\prod_{l\ne 1, i} \sigma_\mu(x_{il})\sigma_\mu(x_{l1})
\,.
\end{equation*} 
Thus, this also must be $W'$-invariant for all $\eta, \mu$. Obviously, replacing in this expression $\eta$ by $\eta+\nu$ gives us back $\beta$, hence $\beta$ is $W'$-invariant, and so must be $\alpha$.

Now, from the formula for $\alpha$ it is easy to see that it is elliptic function of $x_i$, regular at $x_1+x_i=0$ and $x_i+x_l=0$ with $i,l>1$. By $W'$-symmetry, it follows that $\alpha$ is also regular at hyperplanes $x_1-x_i=0$ and $x_i-x_l=0$. As a result, $\alpha$ is globally regular, so is a constant (depending on $\eta$, $\mu$, $\nu$). It can now be evaluated by setting $x_l=l/n$, $1\le l\le n$ in \eqref{alpc1}, which leads to the expression \eqref{alpc}.
\end{proof}

\begin{remark} We see from the proof that $\beta$ is $W'$-invariant. This is easy to confirm for $n=1$ when $B=-v_{\xi_1}(x_1)$, and for $n=2$, in which case $\xi_{12}=\eta+\nu$ so we have 
\begin{equation*}
\alpha=
-\wp(\mu)+\wp(\eta+\nu)\,,\quad
\beta= \sigma_{\eta+\nu}(x_{12}) \sigma_{\eta+\nu}(x_{12}^+)\,.
\end{equation*} 
However, for $n>2$ such a symmetry is not obvious from the formulas. Note that by this symmetry $\beta$ does not have poles at $x_i-x_l=0$ for $i, l>1$. 
\end{remark}

To write down the Lax matrix, we use the same notation as in Proposition \ref{lmct}, namely, extend the set of vectors $\epsilon_i$ and variables $x_i$ to the range $1\le i \le 2n$ by setting $\epsilon_{i+n}=-\epsilon_i$ and $x_{n+i}=-x_i$ for $1\le i\le n$. 
Denote by $\mathcal P, \mathcal Q$ the following $2n\times 2n$ matrices:
\begin{align*}
\mathcal P_{ij}&=-v_{\nu, g}(x_j)\sigma_{\xi_{12}}(x_{ij})\sigma_{\mu}(x_{ij}^+)  
\prod_{l=1}^{2n}{\vphantom{\prod}}'
\sigma_\mu(x_{jl})\quad(i-j\ne 0, \pm n)\,,\qquad
\mathcal P_{ii}=v_{\nu, g}(x_i)\prod_{l=1}^{2n}{\vphantom{\prod}}' \sigma_\mu(x_{il})\,,\\
&\mathcal P_{i, n+i}=\alpha v_{\eta, g}(x_i)+\beta^{\ss_{1i}}v_{\nu, g}(-x_i)\,,\quad 
\mathcal P_{n+i, i}=\alpha v_{\eta, g}(-x_i)+\beta^{\ss_{1i}^+}v_{\nu, g}(x_i)\qquad (1\le i \le n)\,,
\\
\mathcal Q_{ii}&={v}_{\ov{\nu}, \ov{g}}(x_i+c/2)\,t(\epsilon_i)\,,\qquad 
\mathcal Q_{ij}=-{v}_{\xi_1, \ov{g}}(x_i+c/2)\quad (i-j=\pm n)\,,\qquad
\mathcal Q_{ij}=0\quad (i-j\ne 0, \pm n)\,.
\end{align*}
Here $\alpha, \beta$ are given by \eqref{alpc}, \eqref{bet}, and the symbol $\prod{\vphantom{\prod}}'$ in the formula for $\mathcal P_{ij}$ indicates that we exclude those values of $l$ where either $l-i$ or $l-j$ equals $0, \pm n$ (e.g., two values are excluded if $i=j$).
Explicitly, we have
\begin{gather*}
\beta^{\ss_{1i}}=\sum_{j: j\ne i}^n
\sigma_{\xi_{12}}(x_{ij}) \sigma_{\eta+\nu}(x_{ij}^+)\prod_{l\ne i, j} \sigma_\mu(x_{jl})\sigma_\mu(x_{li})\,,
\\
\beta^{\ss_{1i}^+}=\sum_{j: j\ne i}^n 
\sigma_{\xi_{12}}(-x_{ij}^+) \sigma_{\eta+\nu}(x_{ji})\prod_{l\ne i, j} \sigma_\mu(x_{jl})\sigma_\mu(x_{li}^+)
\,.
\end{gather*}

\begin{prop} 
The quantum Lax matrix $\mathcal L$ for the elliptic van Diejen system is $\mathcal L=\mathcal P\mathcal Q$. It satisfies the quantum Lax equation \eqref{qleq} for every quantum Hamiltonian $\widehat H=L^b$ of the van Diejen system and suitable $\mathcal A$. The classical Lax matrix is $L=PQ$ where $P=\mathcal P$, while $Q$ is obtained from $\mathcal Q$ by setting $c=0$ and replacing $t(\epsilon_i)$ with $e^{\beta p_i}$ (with $p_{n+i}=-p_i$ for $1\le i\le n$). The matrix $L$ deforms isospectrally under each of the Hamiltonian flows of the classical van Diejen system. As a result, the functions $\tr L^k$ are in involution.   
\end{prop} 

\begin{remark} Let us remark on how one can calculate a Lax partner for the above $\mathcal L$. Let us look back at the calculation for the $\GL_n$-case in Subsection \ref{arelel}. According to Proposition \ref{elclq}, the Lax partner $\mathcal A$ in that case can be found by considering 
\begin{equation*}
\sum_{i=1}^n \prod_{j\ne i}^n \sigma_\mu(\xi_i-\xi_j) \YH^{\epsilon_i}=\sum_{i=1}^n  Y^{\epsilon_i}\,.
\end{equation*}  
The reason why in Subsection \ref{arelel} we used $Y^{\epsilon_1}+Y^{\epsilon_2}$ is that all other terms $Y^{\epsilon_j}$ vanish after specialising $\xi$ and restricting onto $M'$. 
By a similar reasoning, in constructing a Lax partner for the van Diejen system 
one can use $Y^{\epsilon_1}+Y^{\epsilon_2}+Y^{-\epsilon_1}$. 
\end{remark}


\end{document}